\documentclass[a4paper,11pt]{article}

\usepackage{mathrsfs}
\usepackage{amsmath}
\usepackage{bm}
\usepackage{amssymb}
\usepackage{makeidx}
\makeindex
\usepackage[dvipsnames]{xcolor}
\usepackage[colorlinks=true, linkcolor=blue, citecolor=blue, urlcolor=blue, hypertexnames=false]{hyperref}
\usepackage{titlesec}
\usepackage{amsthm}
\usepackage{setspace}
\usepackage{pifont}
\usepackage{esint}
\usepackage{graphicx}
\usepackage{float}
\usepackage{enumerate}

\numberwithin{equation}{section}
\usepackage[
backend=biber,
style=numeric,
sorting=anyt
]{biblatex}

\DeclareFieldFormat[article]{title}{#1}

\renewbibmacro{in:}{}

\DeclareFieldFormat{pages}{#1}
\DeclareFieldFormat[article]
{journaltitle}{\mkbibemph{#1}\addcomma}

\renewbibmacro*{volume+number+eid}{%
	\iffieldundef{volume}{\printfield{number}}{\printfield{volume}}%
}
\DeclareFieldFormat{eid}{Art.\ #1}
\renewbibmacro*{note+pages}{%
	\printfield{note}%
	\setunit{\bibpagespunct}%
	\iffieldundef{eid}{\printfield{pages}}{\printfield{eid}}%
	\newunit}

\renewbibmacro*{eprint}{%
	\iffieldundef{eprint}
	{}
	{%
		\setunit{\addcomma\space}
		\printtext{%
			\printfield[eprint:arxiv]{eprint}}%
	}%
}

\titleformat{\section}[block]{\bfseries\Large}{\thesection}{1em}{}

\newtheoremstyle{mystyle}
{}{}{\upshape}{}{\bfseries}{.}{.5em}{}

\theoremstyle{mystyle}
\newtheorem{theorem}{Theorem}[section]
\newtheorem{definition}[theorem]{Definition}

\newtheorem{rmk}[theorem]{Remark}

\newtheorem{proposition}[theorem]{Proposition}

\newtheorem{lemma}[theorem]{Lemma}

\newtheorem{remark}{Remark}[section]

\theoremstyle{plain}
\renewenvironment{proof}{\noindent\textbf{Proof:}}{\qed}

\newcommand{\R}{\mathbb{R}}

\newcommand{\bi}{{\bf i}}
\newcommand{\cP}{\mathcal{P}}

\newcommand{\cA}{\mathcal{A}}
\newcommand{\cT}{\mathcal{T}}

\newcommand{\cF}{\mathcal{F}}

\newcommand{\bA}{{\bf A}}

\begin{document} 
	
	\title{\large\bfseries
		On Poincaré–Sobolev Levels for Fractional GJMS Operators on Hyperbolic Space}

	\author{\small Huyuan Chen,\qquad Rui Chen}
	\date{}
	\maketitle
	\thispagestyle{empty}
	\pagenumbering{arabic} 
	
	\noindent\textbf{Abstract:}
	We study the monotonicity, attainment, and strict-gap regions of the
	Poincar\'e--Sobolev levels associated with fractional GJMS operators
	$\cP_s$ on hyperbolic space, where $0<s<n/2$.
	For noninteger $s$, $\cP_s$ is not conformally intertwined with the
	Euclidean fractional Laplacian $(-\Delta)^s$, and we therefore also
	consider an auxiliary family $\widetilde{\cP}_s$ for which this
	intertwining property holds. We establish explicit lower bounds for the
	Hardy term in the corresponding fractional Hardy--Sobolev--Maz'ya
	inequalities and derive strict-gap regimes for both families.
	
	To prove attainment, we decompose minimizing sequences into bounded-frequency
	and high-frequency parts. The high-frequency contribution is controlled by
	the spectral growth of the multipliers, while the bounded-frequency part,
	after suitable hyperbolic isometries, yields a nonzero weak limit, including
	at the spectral bottom. The Brezis--Lieb lemma then gives strong convergence
	and attainment of the infimum. Consequently, the corresponding
	Brezis--Nirenberg problems admit nontrivial weak solutions for parameters at
	or below the respective spectral bottoms: for all subcritical exponents, and
	for the critical exponent under the strict-gap condition. These results extend
	the corresponding integer-order theory to fractional GJMS operators.
	
	\medskip
	
	{\setlength{\emergencystretch}{2em}
		\par}
	
	\medskip
	
	
	\small 
	\bigskip
	
	\tableofcontents
	\thispagestyle{empty} 
	
	\setcounter{page}{1} 
	
	\section{Introduction and Main Results}
	
	The aim of this paper is to investigate the quantitative behavior of the
	Poincar\'e--Sobolev level on the hyperbolic space \(\mathbb H^n\), defined by
	\begin{equation}\label{eq:Sns-Hn-lambda}
		H_{n,s}(\lambda)
		:= \inf_{u\in C_c^\infty(\mathbb H^n)\setminus\{0\}}
		\frac{\displaystyle \int_{\mathbb H^n} (\cP_s u)\,u\,dV_{\mathbb H^n}
			- \lambda \int_{\mathbb H^n} |u|^2\,dV_{\mathbb H^n}}
		{\Bigl(\displaystyle\int_{\mathbb H^n} |u|^{2_s^*}\,dV_{\mathbb H^n}\Bigr)^{\!2/2_s^*}},
		\quad \lambda \in \mathbb{R},
	\end{equation}
	where \(n\ge2\), \(s\in\big(0,\frac{n}{2}\big)\setminus\mathbb{N},\ 2_s^*:=\frac{2n}{n-2s}\), and
	\(\cP_s\) denotes the $s$-order GJMS operator on \(\mathbb H^n\).
	It admits the explicit spectral representation
	\begin{equation}\label{all pass}
		\cP_s
		= 2^{2s}\,
		\frac{\Bigl|\Gamma\!\Bigl(\frac{3+2s}{4}+\frac{\bi}{2}\,\cA\Bigr)\Bigr|^{2}}
		{\Bigl|\Gamma\!\Bigl(\frac{3-2s}{4}+\frac{\bi}{2}\,\cA\Bigr)\Bigr|^{2}}
		\qquad
		\text{with}\quad
		\cA:=\sqrt{-\Delta_{\mathbb H^n}-\rho^{2}},\ \ \rho:=\frac{n-1}{2}.
	\end{equation}
	Here \(\bi=\sqrt{-1}\), \(\Gamma\) denotes the Gamma function, and the functional
	calculus is taken on \(L^2(\mathbb H^n)\); see \cite{lu2023explicit}.
	We emphasize that the right-hand side of \eqref{all pass} is well defined for all
	\(s>0\), and we adopt \eqref{all pass} as the definition of \(\cP_s\) for
	\(s\in(0,\infty)\).
	
	For integer orders \(s=k\in\mathbb N\), the GJMS operator \(\cP_k\) is conformally
	intertwined with the Euclidean fractional Laplacian in the standard models of \(\mathbb H^n\).
	More precisely, in the upper half--space model \((\mathbb H^n,g_{\mathbb H^n})\),
	\begin{equation}\label{eq:tildePs-halfspace-intertwine11}
		x_1^{\,k+\frac n2}\,(-\Delta )^k\!\bigl(x_1^{\,k-\frac n2}u\bigr)
		\;=\; {\cP}_k\,u
		\quad{\rm for}\ \, u\in C^\infty({\mathbb{R}^n_+}),
	\end{equation}
	while in the Poincar\'e ball model \((\mathbb B^n,g_{\mathbb B^n})\),
	\begin{equation}\label{eq:tildePs-ball-intertwine11}
		\Bigl(\frac{1-|x|^2}{2}\Bigr)^{\!k+\frac n2}
		(-\Delta)^k\!\Bigl[\Bigl(\frac{1-|x|^2}{2}\Bigr)^{\!k-\frac n2}u\Bigr]
		\;=\; {\cP}_k\,u	\quad{\rm for}\ \,  u\in C^\infty({\mathbb{B}^n}),
	\end{equation}
	where $\Delta$ is the Laplacian in Euclidean space. 
	In contrast, for non-integer orders \(s\in(0,\frac n2)\setminus\mathbb N\),
	the operator \(\cP_s\) is not conformally equivalent to the Euclidean fractional
	Laplacian on \(\mathbb R^n_+\) or \(\mathbb B^n\). To recover a usable intertwining structure,  
	we use the auxiliary operator 
	\begin{equation}\label{specps}
		\widetilde{\cP}_s
		:= \frac{\bigl|\Gamma\!\bigl(s+\tfrac12+i\cA\bigr)\bigr|^2}
		{\bigl|\Gamma\!\bigl(\tfrac12+i\cA\bigr)\bigr|^2},
	\end{equation}
	which satisfies the conformal intertwining identities
	\eqref{eq:tildePs-halfspace-intertwine11}--\eqref{eq:tildePs-ball-intertwine11}
	for all \(s>0\); see \cite[Theorem~1.7]{lu2023explicit}.
	This operator provides a convenient bridge between analysis on \(\mathbb H^n\)
	and the Euclidean setting.
	Moreover, the precise relation between \(\cP_s\) and \(\widetilde{\cP}_s\) is given by
	\cite[Corollary~5.3]{lu2023explicit}:
	\begin{equation}\label{eq:P-vs-tildeP}
		\cP_s
		\;=\;
		\widetilde{\cP}_s
		\;+\;
		\frac{\sin(\pi s)}{\pi}\,
		\Bigl|\Gamma\!\Bigl(s+\tfrac12+i\cA\Bigr)\Bigr|^2,
		\quad s\in (0,\infty).
	\end{equation}
	In particular, for integer orders \(s\in(0,\tfrac n2)\cap\mathbb N\), one has $\cP_s \;=\; \widetilde{\cP}_s.$
	
	We also introduce the hyperbolic Poincar\'e--Sobolev level
	associated with \(\widetilde{\cP}_s\). This auxiliary level will serve as a key
	tool in our analysis of \(H_{n,s}(\lambda)\), allowing us to circumvent the fact
	that \(\cP_s\) does not admit a direct reduction to the Euclidean fractional
	Laplacian in the non-integer regime:
	\begin{equation}\label{eq:Sns-Hn-lambda111}
		\widetilde H_{n,s}(\lambda)
		:= \inf_{u\in C_c^\infty(\mathbb H^n)\setminus\{0\}}
		\frac{\displaystyle \int_{\mathbb H^n} (\widetilde\cP_s u)\,u\,dV_{\mathbb H^n}
			- \lambda \int_{\mathbb H^n} |u|^2\,dV_{\mathbb H^n}}
		{\Bigl(\displaystyle\int_{\mathbb H^n} |u|^{2_s^*}\,dV_{\mathbb H^n}\Bigr)^{\!2/2_s^*}},
		\quad \lambda\in\mathbb R.
	\end{equation}
	In this paper, we will mainly analyze qualitative properties of the two level
	functions \(\lambda\mapsto H_{n,s}(\lambda)\) and
	\(\lambda\mapsto \widetilde H_{n,s}(\lambda)\), including monotonicity,
	attainability, as well as the associated threshold phenomena.

	The motivation for introducing \(H_{n,s}(\lambda)\) and \(\widetilde H_{n,s}(\lambda)\) stems from their connection with the Brezis--Nirenberg problems
	\begin{equation}\label{eq:BN-Hn-frac}
		\cP_s u =\lambda u + |u|^{p-1}u
		\quad\ \text{in }\, \mathbb H^n
	\end{equation}
	and
	\begin{equation}\label{eq:BN-Hn-frac2}
		\widetilde\cP_s u =\lambda u + |u|^{p-1}u
		\quad\ \text{in }\, \mathbb H^n,
	\end{equation}
	where $	1<p\le 2_s^*-1$ and $2_s^*:=\frac{2n}{n-2s}.$
	In fact, to establish the existence of a nontrivial solution to \eqref{eq:BN-Hn-frac},
	a key step is to determine for which values of \(\lambda\) the following strict
	inequality holds:
	\begin{equation}\label{ineq-1}
		H_{n,s}(\lambda)\;<\;H_{n,s}(0),\quad \widetilde H_{n,s}(\lambda)\;<\;\widetilde H_{n,s}(0).
	\end{equation}
	
	For $p=2_s^*-1$ and $\lambda\le\lambda_{0,s}^{\mathrm{conf}}$,
	the first strict inequality in \eqref{ineq-1} implies that
	$H_{n,s}(\lambda)$ attains its infimum. The analogous conclusion holds
	for $\widetilde H_{n,s}(\lambda)$ under the second strict inequality and
	$\lambda\le\widetilde\lambda_{0,s}^{\mathrm{conf}}$.
	A suitable scalar multiple of each minimizer is a nontrivial weak
	solution of the corresponding Brezis--Nirenberg problem at the critical
	exponent. We therefore study the attainability and monotonicity of the
	two levels, together with their strict-gap regions
	$\mathcal G_{n,s}[H_{n,s}]$ and
	$\mathcal G_{n,s}[\widetilde H_{n,s}]$, where
	\begin{equation}\label{gap se}  
		\mathcal{G}_{n,s}[f]  
		\;:=\;  
		\big\{\lambda\in \mathbb{R}:\ f(\lambda)<f(0)\big\}.  
	\end{equation}

	\subsection{Euclidean (Fractional) Laplacian and Integer-Order GJMS Operators}
	
	To clarify how the strict-gap region (\ref{ineq-1}) governs solvability of the Brezis--Nirenberg problem \eqref{eq:BN-Hn-frac} and \eqref{eq:BN-Hn-frac2} on $\mathbb{H}^n$,  we review the known results in the following settings: the classical and fractional Laplacians on bounded Euclidean domains  and  then  local (integer-order) conformal Laplacian on  hyperbolic space. We first recall
	\begin{equation}\label{eq:upper-bound}
		S_{n,s}:=\inf_{v\in C_c^\infty(\R^n)\setminus\{0\}}
		\frac{\displaystyle\int_{\R^n} v\,(-\Delta)^s v\,dx}
		{\Bigl(\displaystyle\int_{\R^n}|v|^{2_s^*}\,dx\Bigr)^{\!2/2_s^*}} 
	\end{equation}
	with $s\in\big(0,\frac{n}{2}\big) $ and $(-\Delta)^s$ denoting the fractional Laplacian on the Euclidean space $\R^n$ by
	\[	\widehat{\,(-\Delta)^s v\,}(\xi) = |\xi|^{2s}\,\hat v(\xi),
	\quad \xi\in\R^n,\ v\in C_c^\infty(\R^n). \]
	
	For $s\in(0,1]$ and $n>2s$, denote by $S_{n,s,\Omega}(\lambda)$ the Poincar\'e--Sobolev level associated with
	$(-\Delta)^s$ and $\lambda\in\R$:
	\begin{equation}\label{eq:Sns-lambda}
		S_{n,s,\Omega}(\lambda)
		:=\inf_{\substack{v\in C_c^\infty(\Omega)\setminus\{0\}}}
		\frac{\displaystyle\int_{\Omega} v\,(-\Delta)^s v\,dx-\lambda\int_{\Omega} v^2\,dx}
		{\Bigl(\displaystyle\int_{\Omega}|v|^{2_s^*}\,dx\Bigr)^{2/2_s^*}},
	\end{equation}
	where $2_s^*=\frac{2n}{n-2s}$ and $\Omega\subset\R^n$ is either a bounded Lipschitz domain or $\Omega=\R^n$. When $\Omega=\mathbb{R}^n$, we simply write $S_{n,s}(\lambda):=S_{n,s,\mathbb{R}^n}(\lambda)$.
	In fact, for a bounded Lipschitz domain $\Omega,$ the infimum in \eqref{eq:Sns-lambda} is unchanged if one replaces
	$C_c^\infty(\Omega)$ by $H_0^s(\Omega)$, where
	\[
	H_0^s(\Omega):=\bigl\{u\in H^s(\R^n):\ u=0 \ \text{in }\Omega^c\bigr\},
	\]
	since $C_c^\infty(\Omega)$ is dense in $H_0^s(\Omega)$ with respect to the $H^s$-norm (see \cite{BisciRadulescuServadei2016}). 
	Remark that 
	$$S_{n,s,\Omega}(0)=S_{n,s}\quad{\rm and}\quad  
	S_{n,s,\Omega}(\lambda)>0\ \ {\rm if}\ \ \lambda<\lambda_{1,s}(\Omega),$$
	where  $S_{n,s}$ is given in (\ref{eq:upper-bound}) and $\lambda_{1,s}(\Omega)$ is the first   eigenvalue of $(-\Delta)^{s}$ on $\Omega$ subject to the boundary condition that $u=0$ on $\partial\Omega$ for $s=1$ or
	$u=0$ in $\R^n\setminus \Omega$ for $s\in(0,1)$.  Obviously, the function $\lambda\in\R\mapsto S_{n,s,\Omega}(\lambda)$ is non-increasing.   
	By (\ref{gap se}), the strict-gap set is defined by
	\[
	\mathcal{G}_{n,s,\Omega}[S_{n,s,\Omega}]
	\;=\;
	\{\lambda\in\mathbb{R}:\ S_{n,s,\Omega}(\lambda)<S_{n,s}\}. 
	\]

	
	When $s=1$,  we start from  the seminal  paper of Ha\"im Brezis and Louis Nirenberg \cite{brezis1983positive} in 1983, which  
	concerns the existence  of positive solutions to the critical semilinear Dirichlet problem
	\begin{equation}\label{eq:BN-or}
		\left\{
		\begin{aligned}
			-\Delta u &= \lambda u + |u|^{2^*-2}u && \text{in }\Omega,\\[1mm]
			u&=0 && \text{on }\partial\Omega,
		\end{aligned}
		\right.
	\end{equation}
	where $\lambda\in\mathbb{R}$ is a real parameter, and $2^*=\frac{2n}{n-2}$ is the critical Sobolev exponent for the embedding $H_0^1(\Omega)\hookrightarrow L^{2^*}(\Omega)$. 
	{\it Later on,  problem (\ref{eq:BN-or}) is named as the Brezis--Nirenberg problem.  }

	Note  that  if $\Omega$ is star-shaped, then \eqref{eq:BN-or} admits no positive solution for $\lambda\le0$ by Pohozaev's identity and  one rules out positive solutions when $\lambda\ge\lambda_{1,1}(\Omega)$ by testing \eqref{eq:BN-or} against the first eigenfunction.  Brezis and Nirenberg built the crucial inequalities  
	$0<S_{n,1,\Omega}(\lambda)<S_{n,1}$ to guarantee the existence for  $\lambda\in\big(0,\lambda_{1,1}(\Omega)\big)$   when $n\geq 4$
	and for $\lambda\in\big(\tfrac14\lambda_{1,1}(B_1),\lambda_{1,1}(B_1)\big)$ when $n=3$ and $\Omega=B_1$. In these cases, $S_{n,1,\Omega}(\lambda)$ is  achieved in $H_0^1(\Omega)$  and by a Lagrange multiplier argument, a suitable scalar multiple of the minimizer solves \eqref{eq:BN-or}.  Moreover,  for $n=3$, $\Omega=B_1$, and $\lambda\le \tfrac14\lambda_{1,1}(B_1)$, they also showed that problem \eqref{eq:BN-or} admits no positive solution. 
	
	Recently, research on the Brezis--Nirenberg problem has significantly expanded in scope and depth; for a comprehensive overview, we refer the reader to \cite{WW25,FK21,FK25,PV22,SZ10,MR24} and the references therein. \smallskip

	When $s\in(0,1)$, if the infimum in \eqref{eq:Sns-lambda} is achieved by some nontrivial $v\in H_0^s(\Omega)$ and $S_{n,s,\Omega}(\lambda)>0$, then a suitable scaling of $v$ yields a positive solution to the fractional Brezis--Nirenberg problem.
	
	\begin{equation}\label{eq:BN-frac}
		\left\{
		\begin{aligned}
			(-\Delta)^s u &= \lambda u + |u|^{2_s^*-2}u && \text{in }\Omega,\\
			u&=0 && \text{in }\R^n\setminus\Omega.
		\end{aligned}
		\right.
	\end{equation}

	A remarkable feature of the nonlocal regime is that the existence theory depends on the interplay between the dimension $n$ and the order $s$. In the range $n>2s$, one has the following picture:
	\begin{itemize}
		\item \textbf{High dimensions relative to $s\in(0,1)$: $n\ge 4s$.}
		One has  (\cite[proof of Claim~14.1]{BisciRadulescuServadei2016})
		\[
		(0,\infty)\subset \mathcal{G}_{n,s,\Omega}[S_{n,s,\Omega}]
		\]
		and consequently \eqref{eq:BN-frac} admits a nontrivial weak solution for $\lambda\in (0,\lambda_{1,s}(\Omega
		)).$ (\cite[Theorem 14.1]{BisciRadulescuServadei2016}) 
		\item \textbf{Low dimensions relative to $s\in(0,1)$: $2s<n<4s$.}
		There exists a constant $\lambda_s^*>0$ such that (\cite[Proposition 16.4]{BisciRadulescuServadei2016})
		\[
		(\lambda_s^*,\infty)\subset \mathcal{G}_{n,s,\Omega}[S_{n,s,\Omega}].
		\]
		In this regime, the relation between $\lambda_s^*$ and $\lambda_{1,s}(\Omega)$ is unknown---even when $\Omega$ is the unit ball; consequently, the standard mountain--pass scheme alone does not guarantee existence. In \cite[Theorem~16.1]{BisciRadulescuServadei2016}, the authors combine mountain--pass and linking arguments, treating separately the cases $\lambda_s^*<\lambda_{1,s}(\Omega
		)$ and $\lambda_s^*\ge \lambda_{1,s}(\Omega)$, and thereby obtain nontrivial solutions of \eqref{eq:BN-frac} for every $\lambda>\lambda_s^*$ that is not a Dirichlet eigenvalue of $(-\Delta)^s$.
	\end{itemize}
	
	Brezis--Nirenberg problems involving nonlocal operators have been studied extensively; we refer the reader to \cite{DK23,servadei2015brezis,BC15,MS15,MM17,GL21} and the references therein.  Building upon the aforementioned findings, we derive the following conclusion.

	\begin{proposition}\label{prop:Lp-perturbation-frac-prop}
		Let $\Omega\subset\mathbb R^n$ be a bounded Lipschitz domain, $s\in(0,1]$ and $n>2s$, and set $2_s^*=\frac{2n}{n-2s}$. 
		Then the following statements hold:
		\begin{enumerate}
			\item[(i)] For $n\ge4s$, $\mathcal{G}_{n,s,\Omega}[S_{n,s,\Omega}]=(0,+\infty)$ and  $S_{n,s,\Omega}(\lambda)$ is achieved   in $H_0^s\left(\Omega\right)$ if and only if $\lambda \in \mathcal{G}_{n,s,\Omega}[S_{n,s,\Omega}]$.   Moreover,  $S_{n,s,\Omega}(\cdot)$ is strictly decreasing in $\mathcal{G}_{n,s,\Omega}[S_{n,s,\Omega}]$, 
			\[-\lambda |\Omega
			|^{\frac{2s}{n}}\le S_{n,s,\Omega}(\lambda)<S_{n,s,\Omega}(\lambda_{1,s}\left(\Omega\right))=0 \ \, {\rm for}\  \lambda>\lambda_{1,s}(\Omega
			)\quad{\rm and}\quad  
			S_{n,s,\Omega}(\lambda)=S_{n,s} \ \, {\rm for}\  \lambda\leq 0.\]

			\item[(ii)] For $s\in(0,1)$ and $2s<n<4s$,  $(\lambda_{s}^*,+\infty)\subset\mathcal{G}_{n,s,\Omega}[S_{n,s,\Omega}]$ 
			and $S_{n,s,\Omega}(\lambda)$ is achieved  in $H_0^s\left(\Omega\right)$  if $\lambda \in (\lambda_s^*,+\infty)$. Moreover,  $S_{n,s,\Omega}(\cdot)$ is strictly decreasing in $ (\lambda_s^*,+\infty)$ and for any $\lambda\le 0$,
			\[
			S_{n,s,\Omega}(\lambda)=S_{n,s}.
			\]

			\item[(iii)] For  $\Omega=B_1$, $\mathcal{G}_{3,1,B_1}[S_{3,1,B_1}]=(\tfrac14  \lambda_{1,1}(B_1),+\infty)$, 
			$S_{3,1,B_1}(\lambda)$ is achieved in $H_0^1(B_1)$ if and only if $\lambda \in  \mathcal{G}_{3,1,B_1}[S_{3,1,B_1}]$. Moreover,
			$S_{3,1,B_1}(\cdot)$ is strictly decreasing in $\mathcal{G}_{3,1,B_1}[S_{3,1,B_1}]$, and for $\lambda\le \tfrac14  \lambda_{1,1}(B_1)<\lambda_{1,1}(B_1)<\mu$
			\[-\mu|B_1|^{\frac{2}{3}}\le S_{3,1,B_1}(\mu)<S_{3,1,B_1}(\lambda_{1,1}\left(B_1\right))=0  \quad{\rm and}\quad  
			S_{3,1,B_1}(\lambda)=S_{3,1} .  \]
		\end{enumerate}
	\end{proposition}
	
	\begin{remark}\label{comtim}
		Let \(\mathcal{G}_{n,s,\Omega}[S_{n,s,\Omega}]\) denote the strict--gap region, i.e.,
		the set of parameters for which \(S_{n,s,\Omega}(\lambda)<S_{n,s}\).
		When \(n\ge 4s\), this region is completely understood: one has the full
		classification
		\[
		\mathcal{G}_{n,s,\Omega}[S_{n,s,\Omega}] \;=\; (0,+\infty).
		\]
		In contrast, in the intermediate regime \(2s<n<4s\), a sharp description of
		\(\mathcal{G}_{n,s,\Omega}[S_{n,s,\Omega}]\) remains open. In the local case \(s=1\),
		the situation is settled for the unit ball \(\Omega=B_1\). However, for
		\(s\in(0,1)\), it remains to determine the sharp value of the critical threshold
		\(\Lambda\le \lambda_s^*\) such that
		\[
		S_{n,s,\Omega}(\lambda)<S_{n,s}\quad \text{for } \lambda>\Lambda,
		\qquad
		S_{n,s,\Omega}(\lambda)=S_{n,s}\quad \text{for } \lambda\le \Lambda,
		\]
		even when \(\Omega=B_1\). Furthermore, for \(s\in (1,\tfrac{n}{2})\), the strict--gap region
		poses significant analytical challenges and  the Brezis--Nirenberg problem remains comparatively underexplored.
	\end{remark}

	On the hyperbolic space $\mathbb H^n$,  we introduce the conformal GJMS operators (see \cite{graham1992conformally,fefferman2013juhl,gover2006laplacian,juhl2013explicit})
	\[\cP_1=-\Delta_{\mathbb{H}^n}-\tfrac{n(n-2)}{4}= \cA^2+\tfrac14\]
	and for integer order  \( k\ge 2\)  
	\[
	\cP_k
	\;=\;
	\cP_1\,(\cP_1+2)\cdots\bigl(\cP_1+k(k-1)\bigr)
	\;=\;
	\prod_{j=1}^k\Bigl(\cA^2+\bigl(j-\tfrac12\bigr)^2\Bigr),
	\]
	where $-\Delta_{\mathbb H^n}$ is the Laplace--Beltrami operator and $\cA,\rho$ are given in (\ref{all pass}).
	Let \(2_k^*:=\frac{2n}{n-2k}\),
	which is the critical exponent for order \(2k\).
	In particular, the bottom of $\cP_k$ satisfies
	\[
	\lambda_{0,k}^{\mathrm{conf}}:=\inf\sigma(\cP_k)
	\;=\;
	\prod_{j=1}^k\Bigl(j-\tfrac12\Bigr)^{\!2}.
	\]
	Recently,   Brezis--Nirenberg type problems on the hyperbolic space have been studied in two principal settings: (i) on bounded domains, and (ii) on the whole space. 
	
	In what follows, we focus on the whole space problem; for results on bounded domains we refer to~\cite{stapelkamp2002brezis,benguria2016solution,li2022higher}. Specifically, we consider
	\begin{equation}\label{eq:BN-Hn-local}
		-\Delta_{\mathbb H^n} u = \lambda u + |u|^{2^*-2}u
		\qquad\text{in }\mathbb H^n,
	\end{equation}
	where $2^*=\frac{2n}{n-2}$, and the bottom of the $L^2$--spectrum of $-\Delta_{\mathbb H^n}$ equals $\frac{(n-1)^2}{4}$. In \cite{ManciniSandeep2008}, the authors established the following existence result for positive solutions:
	\begin{itemize}
		\item If $n\ge4$, then for $\frac{n(n-2)}{4}<\lambda\le \frac{(n-1)^2}{4},$
		problem \eqref{eq:BN-Hn-local} admits a positive entire solution (\cite[Theorem 1.5]{ManciniSandeep2008}); for $\lambda\le \frac{n(n-2)}{4}$, problem \eqref{eq:BN-Hn-local} does not have any positive entire solution.
		\item If $n=3$, then for $	\lambda \;\le\; 1$,
		problem \eqref{eq:BN-Hn-local} has no positive entire solution.
	\end{itemize}

	The appearance of the quantity $\frac{n(n-2)}{4}$ in \eqref{eq:BN-Hn-local} is in fact natural from the conformal viewpoint. In fact, \eqref{eq:BN-Hn-local} rewrites as
	\begin{equation}\label{eq:BN-Hn-local1}
		\cP_1 u = \mu u + |u|^{2^*-2}u
		\quad\text{in }\mathbb H^n,\qquad \mu=\lambda-\frac{n(n-2)}{4}.
	\end{equation}
	Since the bottom of the $L^2$--spectrum of $-\Delta_{\mathbb H^n}$ is $\frac{(n-1)^2}{4}$, the bottom of the spectrum of $\cP_1$ is $\frac14$,
	so the admissible window for $\lambda$ is exactly a shift of size $\frac{n(n-2)}{4}$ from the spectral bottom. With this normalization, the results of \cite{ManciniSandeep2008} can be restated as follows:
	
	\begin{itemize}
		\item If $n\ge 4$, then $0 < \mu \le\frac{1}{4}$
		guarantees a positive entire solution to \eqref{eq:BN-Hn-local1} \cite[Theorem~1.5]{ManciniSandeep2008}; 
		whereas for $\mu \;\le\;0,$
		problem \eqref{eq:BN-Hn-local1} admits no positive entire solution \cite[Theorem 1.6]{ManciniSandeep2008}.
		\item If $n=3$, then for $	\mu \;\le\; \frac{1}{4},$
		problem \eqref{eq:BN-Hn-local1} has no positive entire solution \cite[Theorem 1.7]{ManciniSandeep2008}.
	\end{itemize}
	
	When the integer $k\geq 2$, the higher integer order Brezis--Nirenberg problems
	on hyperbolic spaces have been studied in \cite{li2022higher,lu2022green}. Based upon these results, we derive the following conclusion.

	\begin{proposition}\label{prop:Lp-perturbation-frac-prophn11}
		Assume that   integers $k\geq1,\,  n> 2k $    and $H_{n,k}(\lambda),S_{n,k}, \mathcal{G}_{n,k}$ be given in  (\ref{eq:Sns-Hn-lambda}),(\ref{eq:upper-bound}) and (\ref{gap se}) respectively. 
		\begin{enumerate}
			\item[(i)] When $n\ge 4k$, \ $\mathcal{G}_{n,k}[	H_{n,k}]=(0,+\infty)$. \(H_{n,k}(\lambda)\) is achieved and strictly decreasing for every
			\(\lambda\in(0,\lambda_{0,k}^{\mathrm{conf}})\) when \(k\ge2\); while for \(k=1,\)
			it is achieved if and only if \(\lambda\in(0,\lambda_{0,1}^{\mathrm{conf}}]\) and strictly decreasing on \((0,\lambda_{0,1}^{\mathrm{conf}}]\),
			$$-\infty=H_{n,k}(\lambda)<0<H_{n,k}(\lambda_{0,k}^{\mathrm{conf}}) \ \, {\rm for}\ \lambda>\lambda_{0,k}^{\mathrm{conf}}\quad{\rm and}\quad 	 H_{n,k}(\lambda)=S_{n,k}\ \text{ for $\lambda\leq0$}. $$

			\item[(ii)] When $2k+2\le n\le 4k-1$,  there exists $\lambda_k^{\mathrm{conf}}\in (0,\lambda_{0,k}^{\mathrm{conf}})$ such that $(\lambda_k^{\mathrm{conf}},\infty)\subset \mathcal{G}_{n,k}[	H_{n,k}]$. Moreover, $H_{n,k}(\lambda)$ is achieved and strictly decreasing for $\lambda \in (\lambda_k^{\mathrm{conf}},\lambda_{0,k}^{\mathrm{conf}})$, 
			$$-\infty=H_{n,k}(\lambda)<0<H_{n,k}(\lambda_{0,k}^{\mathrm{conf}})  \ \, {\rm for}\ \lambda>\lambda_{0,k}^{\mathrm{conf}}\quad{\rm and}\quad 	 H_{n,k}(\lambda)=S_{n,k}\ \text{ for $\lambda\leq0$}. $$

			\item[(iii)] When $n=2k+1$,  $\mathcal{G}_{n,k}[	H_{n,k}]=(\lambda_{0,k}^{\mathrm{conf}},+\infty)$. Moreover,
			\[
			-\infty=H_{n,k}(\mu)	<0< S_{n,k}=H_{n,k}(\lambda ),\quad \mu> \lambda_{0,k}^{\mathrm{conf}}\geq \lambda.
			\]
			In particular, \(H_{n,k}(\lambda)\) is never achieved for any \(\lambda\in\mathbb R\) when $k=1$. 
		\end{enumerate}
	\end{proposition}

	\begin{remark}\label{rmk:gap-contrast}
		(a) There is a jump discontinuity of \(H_{n,k}(\cdot)\) at
		\(\lambda=\lambda_{0,k}^{\mathrm{conf}}\). This phenomenon stems from the fact
		that \(\lambda_{0,k}^{\mathrm{conf}}\) is not an eigenvalue, but rather the bottom of the continuous
		spectrum of \(\cP_k\). This is in sharp contrast with the Euclidean quantity
		\(S_{n,s,\Omega}(\lambda)\) at \(\lambda=\lambda_{1,s}(\Omega)\), since
		\(\lambda_{1,s}(\Omega)\) is the first Dirichlet eigenvalue of
		\((-\Delta)^s\).

		$(b)$ Note  the strict--gap region $\mathcal{G}_{n,k}[H_{n,k}]$ can be given explicitly for  $n\geq 4k$ and $n=2k+1$.
		In contrast, for $2k+2\le n\le 4k-1$, it remains open to give $\mathcal{G}_{n,k}[H_{n,k}]$ with an explicit interval. 
		
		(c) We conjecture that, in the borderline dimension \(n=2k+1\), the level
		\(H_{n,k}(\lambda)\) is not attained for any \(\lambda\in\mathbb R\)
		whenever \(k\ge 2\).
		
		(d) The results of \cite{li2022higher,lu2022green} establish
		attainment below the spectral bottom. In the present work, we establish
		endpoint attainment in both the integer and fractional order settings.
		After obtaining these results, we became aware that Lu and Tao
		\cite[Theorem~1.1]{LuTao2026} had independently proved attainment at
		\(\lambda=\lambda_{0,k}^{\mathrm{conf}}\) for integer \(k\ge2\) and
		\(n\ge2k+2\). Their result does not cover the fractional-order case.
	\end{remark}
	
	The proofs of Propositions~\ref{prop:Lp-perturbation-frac-prop} and~\ref{prop:Lp-perturbation-frac-prophn11}
	mainly rely on the classical analysis of Poincar\'e--Sobolev levels, together with several auxiliary lemmas,
	and on combining these tools with previously known results in the literature.
	For the convenience of the reader, we provide detailed proofs in the Appendix \ref{appendix}.

	\subsection{Fractional GJMS Operators}
	
	Fractional conformally covariant operators on the conformal infinity of a Poincar\'e--Einstein manifold
	were introduced by Graham and Zworski through scattering theory in their seminal work in \cite{graham2003scattering}, building on the foundational analytic framework
	of Mazzeo and Melrose on meromorphic continuation of the resolvent in
	\cite{mazzeo1987meromorphic}. More precisely, let $(X^{n+1},g_+)$ be a conformally compact Einstein manifold with conformal
	infinity $(M^n,[g])$, and fix $\gamma\in(0,\tfrac n2)\setminus\mathbb N$ with $s=\tfrac n2+\gamma$.
	Given boundary data $f$ on $M$, consider the generalized eigenvalue problem
	\[
	-\Delta_{g_+}u - s(n-s)u=0,
	\]
	whose solutions admit an expansion
	$u=r^{n-s}(f+\cdots)+r^{s}(h+\cdots)$ near $M$. The scattering operator
	$S(s)$ is defined by $S(s)f=h$, and the fractional GJMS operator $\cP_\gamma$ is obtained (up to a
	normalization) from $S(\tfrac n2+\gamma)$. The resulting $\cP_\gamma$ is an elliptic nonlocal
	pseudodifferential operator of order $2\gamma$ and obeys the expected conformal covariance law.
	We refer to Chang--Gonz\'alez in \cite{chang2011fractional} for the
	extension-type characterization and further analytic developments, and to related works on fractional
	Yamabe-type problems (e.g., in \cite{gonzalez2013fractional}), as well as to the broader scattering literature
	on asymptotically hyperbolic geometry (e.g., in \cite{joshi2000inverse}).
	
	We now return to the general fractional conformal operators \(\cP_s\) with
	\(s\in(0,\tfrac n2)\setminus\mathbb N\) and  \(\widetilde\cP_s\) with
	\(s\in(0,\tfrac n2)\).
	Recalling the spectral representation \eqref{all pass} and \eqref{specps}, we obtain (see Section \ref{subsec:conf-frac-lap})
	\begin{equation}\label{conf}
		\lambda_{0,s}^{\mathrm{conf}}
		\;:=\;\inf\sigma(\cP_s)
		\;=\;
		2^{2s}\,
		\frac{\Gamma\!\bigl(\tfrac{3+2s}{4}\bigr)^2}
		{\Gamma\!\bigl(\tfrac{3-2s}{4}\bigr)^2}\ge 0
	\end{equation}
	and
	\begin{equation}\label{la0s}
		\widetilde\lambda_{0,s}^{\mathrm{conf}}
		=\inf\sigma(\widetilde\cP_s)=\frac{\Gamma\!\bigl(s+\frac12\bigr)^2}{\Gamma\!\bigl(\frac12\bigr)^2}>0.
	\end{equation}
	
	We first establish explicit lower bounds for the Hardy-term coefficient in the fractional
	Hardy--Sobolev--Maz'ya inequalities involving $\cP_s$ and $\widetilde{\cP}_s$.
	These bounds, in turn, yield partial characterizations of the Poincar\'e--Sobolev levels
	$H_{n,s}$ and $\widetilde H_{n,s}$.

	\begin{theorem}\label{thm:Hardy-lowerbound-Ptilde}
		Assume that  $n\ge2,s\in(0,\tfrac n2)$, $S_{n,s},\, \widetilde \lambda_{0,s}^{\mathrm{conf}}$ are defined in (\ref{eq:upper-bound})(\ref{la0s}) respectively.  Consider the optimal lower shift for the  inequality
		\[
		\widetilde\Lambda_{n,s}^{\mathrm{HS}}
		\;:=\;
		\inf\Bigl\{\lambda\in\mathbb R:\ 
		\int_{\mathbb H^n} (\widetilde \cP_s u)\,u\,dV_{\mathbb H^n}
		+ \lambda \int_{\mathbb H^n} |u|^2\,dV_{\mathbb H^n}
		\;\ge\;
		S_{n,s}\Bigl(\int_{\mathbb H^n}|u|^{2_s^*}\,dV_{\mathbb H^n}\Bigr)^{\!\frac{2}{2_s^*}}
		\ \ \forall\,u\in C_c^\infty(\mathbb H^n)\Bigr\}.
		\]
		Then:
		\begin{enumerate}
			\item[(i)] If $s\in(0,\tfrac n4]$, one has $	\widetilde\Lambda_{n,s}^{\mathrm{HS}}\ge 0$. 
			\item[(ii)] If $s\in(\tfrac n4,\tfrac n2)$, then there exists $\widetilde\lambda_{s}^{\mathrm{conf}}\in(0,\widetilde{\lambda}_{0,s}^{\mathrm{conf}}]$ such that $		\widetilde\Lambda_{n,s}^{\mathrm{HS}}\;\ge\;-\widetilde\lambda_{s}^{\mathrm{conf}}.$
		\end{enumerate}
	\end{theorem}
	
	For $s \in (0,\, n/4] \cap \mathbb{N}$, the above result recovers
	\cite[Theorem~1.7]{lu2022green} and yields an alternative proof thereof; whereas for
	$s \in (n/4,\, n/2)$, they established that $\lambda_{s}^{\mathrm{conf}} \in \bigl(0,\, \widetilde{\lambda}_{0,s}^{\mathrm{conf}}\bigr).$
	
	Next, from the relation identity (\ref{eq:P-vs-tildeP}) between $\cP_s$ and $\widetilde\cP_s$ we introduce an important parameter
	\begin{equation}\label{bsss}
		b_s
		:= \max\Bigl\{0,\ \frac{\sin(\pi s)}{\pi}\Bigr\}\,
		\bigl|\Gamma\!\bigl(s+\tfrac12\bigr)\bigr|^{2}
		\quad{\rm for}\ \, s>0.
	\end{equation}
	We have the following result:
	\begin{theorem}
		\label{prop:gap-bs-lambda011}
		Let $n\ge 2$, $s\in \left(0,\frac{n}{2}\right)\setminus \mathbb{N}$, $b_s,\lambda_{0,s}^{\mathrm{conf}}$ be defined in (\ref{bsss}), (\ref{conf}) respectively  and
		the optimal lower shift for the  inequality
		\[
		\Lambda_{n,s}^{\mathrm{HS}}
		\;:=\;
		\inf\Bigl\{\lambda\in\mathbb R:\ 
		\int_{\mathbb H^n} ( \cP_s u)\,u\,dV_{\mathbb H^n}
		+ \lambda \int_{\mathbb H^n} |u|^2\,dV_{\mathbb H^n}
		\;\ge\;
		S_{n,s}\Bigl(\int_{\mathbb H^n}|u|^{2_s^*}\,dV_{\mathbb H^n}\Bigr)^{\!\frac{2}{2_s^*}}
		\ \ \forall\,u\in C_c^\infty(\mathbb H^n)\Bigr\}.
		\]
		
		$(a)$ One has that 
		\begin{equation}\label{eq:gap-explicit}
			\lambda_{0,s}^{\mathrm{conf}}-b_s
			=
			\begin{cases}
				\displaystyle \frac{\Gamma\bigl(s+\tfrac12\bigr)^2}{\pi},
				& \sin(\pi s)>0,\\[1.2ex]
				\displaystyle \frac{1+\sin(\pi s)}{\pi}\,\Gamma\bigl(s+\tfrac12\bigr)^2,
				& \sin(\pi s)\le 0.
			\end{cases}
		\end{equation}

		$(b)$  One has that 
		\[ \Lambda_{n,s}^{\mathrm{HS}}\ge -b_s\quad\text{for $s\in(0,\tfrac n4]$}  \]
		and 
		\[	\Lambda_{n,s}^{\mathrm{HS}}\;\ge\;-\min\left\{b_s+\widetilde\lambda_{s}^{\mathrm{conf}},\lambda_{0,s}^{\mathrm{conf}}\right\}\quad \text{for $s\in(\tfrac n4,\tfrac n2)$.  }\]
	\end{theorem}

	We next provide a characterization of the Poincar\'e--Sobolev level $\widetilde H_{n,s}$.
	The proof combines Theorem~\ref{thm:Hardy-lowerbound-Ptilde} with the attainability analysis for $\widetilde H_{n,s}$.

	\begin{theorem}\label{prop:strict-gap-fractional}
		Let $n\ge2$ and $s\in (0,\frac{n}{2})$. Then:
		\begin{enumerate}
			\item[(i)] 
			For $s\in \left(0,\frac{n}{4}\right]$, $\mathcal{G}_{n,s}[\widetilde H_{n,s}]=(0,+\infty)$, \(\widetilde H_{n,s}(\lambda)\) is achieved if and only if \(\lambda\in(0,\widetilde\lambda_{0,s}^{\mathrm{conf}}]\) and strictly decreasing in 
			\((0,\widetilde\lambda_{0,s}^{\mathrm{conf}}]\). Moreover,
			$$-\infty=\widetilde H_{n,s}(\mu)<0<\widetilde H_{n,s}(\widetilde \lambda_{0,s}^{\mathrm{conf}}) \ \, {\rm for}\ \mu>\widetilde\lambda_{0,s}^{\mathrm{conf}}\quad{\rm and}\quad 	\widetilde H_{n,s}(\lambda)=S_{n,s}\ \text{ for $\lambda\leq0$}. $$

			\item[(ii)] 
			For $s\in (\frac{n}{4},\frac{n-1}{2})$ when $n\ge 3$ and $s\in (\frac{1}{2},1)$ when $n=2$, then there exists $\widetilde\lambda_{s}^{\mathrm{conf}}\in(0,\widetilde{\lambda}_{0,s}^{\mathrm{conf}}]$ such that $\bigl(\widetilde\lambda_{s}^{\mathrm{conf}},\infty\bigr)
			\;\subset\;
			\mathcal{G}_{n,s}\bigl[\widetilde H_{n,s}\bigr]$ and for any
			$\lambda\le0<\widetilde\lambda_{0,s}^{\mathrm{conf}}<\mu$, 
			\[
			-\infty=\widetilde H_{n,s}(\mu)
			\;<\;\widetilde H_{n,s}(\lambda)
			\;=\;S_{n,s}.
			\]
			
			\item[(iii)]  When $n\ge3$ and $s\in[\frac{n-1}{2},\frac{n}{2}),$ $\mathcal{G}_{n,s}\bigl[\widetilde H_{n,s}\bigr]
			\;=\;\bigl(\widetilde{\lambda}_{0,s}^{\mathrm{conf}},\infty\bigr)$. Moreover,
			\[
			-\infty=\widetilde H_{n,s}(\mu)\;<\;0
			\;<\;\widetilde H_{n,s}(\lambda)
			\;=\;\widetilde H_{n,s}\!\bigl(\widetilde{\lambda}_{0,s}^{\mathrm{conf}}\bigr)
			\;=\;S_{n,s},\text{ for $\lambda\le \widetilde{\lambda}_{0,s}^{\mathrm{conf}}<\mu$}.
			\]
		\end{enumerate}
	\end{theorem}
	
	\begin{remark}
		We conjecture that, when $n\ge3$, $s\in[\frac{n-1}{2},\frac{n}{2}),$ \(\widetilde H_{n,s}(\lambda)\) is never achieved
		for any \(\lambda\in\mathbb R\).    
	\end{remark}

	We introduce the following notation
	\newcommand{\Bplus}{\mathcal{B}_+}
	\newcommand{\Bzero}{\mathcal{B}_0}
	\[
	\Bplus
	:= \Bigl\{ s\in(0,\infty)\,:\, \sin(\pi s)>0 \Bigr\},
	\qquad
	\Bzero
	:= \Bigl\{ s\in(0,\infty)\,:\, \sin(\pi s)\le 0 \Bigr\}.
	\]
	Accordingly, we have $\lambda_{0,s}^{\mathrm{conf}}>b_s>0$ for $s\in\Bplus$, while
	$\lambda_{0,s}^{\mathrm{conf}}\ge b_s=0$ for $s\in\Bzero$.

	\begin{theorem}\label{thm:strict-gap-Bplus-Bzero}
		Let \(n\ge2\) and \(s\in(0,\tfrac{n}{2})\setminus\mathbb N\).  
		\begin{enumerate}
			
			\item[(i)] For $s\in \left(0,\frac{n}{4}\right]\cap\Bzero$, one has that 
			$$H_{n,s}(\lambda)<S_{n,s} \ \, {\rm for}\ \lambda>0\quad{\rm and}\quad 	 H_{n,s}(\lambda)\le S_{n,s}\ \text{ for $\lambda\leq0$}. $$

			\item[(ii)] For $s\in \left(0,\frac{n}{4}\right]\cap\Bplus$, $(b_s,\infty)\subset\mathcal{G}_{n,s}\bigl[ H_{n,s}\bigr]$ and \(H_{n,s}(\lambda)\) is achieved and strictly decreasing for
			\(\lambda\in(b_s,\lambda_{0,s}^{\mathrm{conf}}]\). Moreover,
			$$H_{n,s}(\lambda)=S_{n,s}\quad \text{ for $\lambda\le 0$}.$$
			

			\item[(iii)] 
			For $s\in (\frac{n}{4},\frac{n-1}{2})$ when $n\ge 3$ and $s\in (\frac{1}{2},1)$ when $n=2$, 
			then there exists 
			\[\lambda_{s}^{\mathrm{conf}}:=\min \left\{{b_s+\widetilde\lambda_{s}^{\mathrm{conf}},\lambda_{0,s}^{\mathrm{conf}}}\right\}\in [0,\lambda_{0,s}^{\mathrm{conf}}]\]
			such that for every
			\(\lambda\le0\le\lambda_{s}^{\mathrm{conf}}<\mu\),
			\[
			H_{n,s}(\mu)\;< \;S_{n,s},\quad H_{n,s}(\lambda)\le S_{n,s}.
			\]
			
			\item[(iv)] 
			When $n\ge3$ and $s\in[\frac{n-1}2,\frac n2)$, 
			$\mathcal{G}_{n,s}\bigl[ H_{n,s}\bigr]
			= \bigl({\lambda}_{0,s}^{\mathrm{conf}},\infty\bigr).$ Moreover, 
			\[
			-\infty=H_{n,s}(\mu)\;<\;0
			\;<\; H_{n,s}(\lambda)
			\;=\;S_{n,s}\quad \text{ for $\lambda\le {\lambda}_{0,s}^{\mathrm{conf}}<\mu$}.
			\]	
		\end{enumerate}
	\end{theorem}
	
	Below we systematically collect several open questions concerning the operators considered above
	and, in particular, the still largely unresolved issue of attainability for the associated Poincar\'e--Sobolev levels.\\[1mm] 
	\textbf{Unsolved problems:}
	\begin{quote}
		(i) Let $k\ge 2$ and $n=2k+1$, we conjecture that problem \eqref{eq:BN-Hn-frac} admits no positive solution for any $\lambda\in \mathbb{R}$ in the critical case $p=2_k^*-1$, thus $H_{n,k}$ is never achieved.

		(ii) Let $k\ge 2$ and $2k+2\le n\le 4k-1$. Find an explicit optimal threshold $\Lambda_k^{\mathrm{conf}}$ satisfying
		\[
		0\le \Lambda_k^{\mathrm{conf}}\le \lambda_k^{\mathrm{conf}}<\lambda_{0,k}^{\mathrm{conf}}
		\]
		such that the strict--gap region for the hyperbolic level $H_{n,k}$ satisfies
		\[
		\mathcal{G}_{n,k}[H_{n,k}]=\bigl(\Lambda_k^{\mathrm{conf}},\,\infty\bigr).
		\]
		Equivalently, find the value $\Lambda_k^{\mathrm{conf}}$ for which
		$H_{n,k}(\lambda)<S_{n,k}$ holds if and only if $\lambda>\Lambda_k^{\mathrm{conf}}$.
		
		(iii)  When $n\ge3$ and $s\in[\frac{n-1}{2},\frac{n}{2}),$  we conjecture that problem \eqref{eq:BN-Hn-frac2} admits no positive solution for any $\lambda\in \mathbb{R}$ in the critical case $p=2_s^*-1$, thus $\widetilde H_{n,s}$ is never achieved.
		
		(iv) Assume $s\in (\frac{n}{4},\frac{n-1}{2})$ when $n\ge 3$ and $s\in (\frac{1}{2},1)$ when $n=2$ with $s\notin\mathbb N$. Find an explicit optimal threshold $\Lambda_s^{\mathrm{conf}}$ satisfying
		\[
		0\le \Lambda_s^{\mathrm{conf}}\le \widetilde{\lambda}_{s}^{\mathrm{conf}}\le \widetilde{\lambda}_{0,s}^{\mathrm{conf}}
		\]
		such that the strict--gap region for the fractional hyperbolic level $\widetilde H_{n,s}$ satisfies
		\[
		\mathcal{G}_{n,s}\!\bigl[\widetilde H_{n,s}\bigr]
		=\bigl(\Lambda_s^{\mathrm{conf}},\,\infty\bigr).
		\]
		Equivalently, find the value $\Lambda_s^{\mathrm{conf}}$ for which
		$\widetilde H_{n,s}(\lambda)<S_{n,s}$ holds if and only if $\lambda>\Lambda_s^{\mathrm{conf}}$.
		
		(v) For $\cP_s$, we can consider analogous problems to (iii) (iv).
	\end{quote}
	
	Finally, we address the Brezis--Nirenberg problem driven by $\cP_s$ and by $\widetilde{\cP}_s$,
	and describe the existence of nontrivial solutions in terms of the parameter $\lambda$.
	In fact, these existence results follow as a consequence of the attainability of the corresponding
	Poincar\'e--Sobolev levels $H_{n,s}$ and $\widetilde H_{n,s}$.
	\begin{theorem}\label{mainthe2}
		Let $s\in (0,\frac{n}{2})$, then problem~\eqref{eq:BN-Hn-frac2} admits a nontrivial solution
		under assumptions:
		\begin{itemize}
			\item[(i)] (Subcritical case) If $1<p<2_s^*-1$, assume that
			$\lambda\le\widetilde\lambda_{0,s}^{\mathrm{conf}}$.
			
			\item[(ii)] (Critical case) If $p=2_s^*-1$, assume that $s$ and $\lambda$ satisfy
			\[
			\lambda \in \mathcal{G}_{n,s}\!\big[\widetilde H_{n,s}\big]\cap
			(-\infty,\widetilde\lambda_{0,s}^{\mathrm{conf}}].
			\]
			In particular, when $s\in (0,\frac{n}{4}]$, one has
			\[
			\mathcal{G}_{n,s}\!\big[\widetilde H_{n,s}\big]\cap
			(-\infty,\widetilde\lambda_{0,s}^{\mathrm{conf}}]
			=(0,\widetilde\lambda_{0,s}^{\mathrm{conf}}].
			\]
		\end{itemize}
	\end{theorem}
	\begin{theorem}\label{mainthe}
		Let $s\in (0,\frac{n}{2})\setminus\mathbb{N}$. Problem~\eqref{eq:BN-Hn-frac} admits a nontrivial solution
		under assumptions:
		\begin{itemize}
			\item[(i)] (Subcritical case) If $1<p<2_s^*-1$, assume that
			$\lambda\le\lambda_{0,s}^{\mathrm{conf}}$.
			
			\item[(ii)] (Critical case) If $p=2_s^*-1$, assume that $s$ and $\lambda$ satisfy
			\[
			\lambda \in \mathcal{G}_{n,s}\!\big[H_{n,s}\big]\cap
			(-\infty,\lambda_{0,s}^{\mathrm{conf}}].
			\]
			In particular, for $s\in \left(0,\frac{n}{4}\right]\cap\Bplus$, one has
			\[
			(b_s,\lambda_{0,s}^{\mathrm{conf}}]\subset\mathcal{G}_{n,s}\!\big[ H_{n,s}\big]\cap
			(-\infty,\lambda_{0,s}^{\mathrm{conf}}].
			\]
		\end{itemize}
	\end{theorem}
	
	\smallskip
	\noindent\textbf{Comments on the main difficulties.}
	At the spectral bottom, the energy norm does not control the global
	$L^2$ norm, and a bounded minimizing sequence may have zero weak limit.
	We separate bounded and high spectral frequencies. Interpolation in the
	subcritical case, and the strict-gap condition in the critical case,
	ensure that the bounded-frequency part retains nonzero $L^q$ mass.
	Point evaluation after a bounded spectral cutoff is continuous in the
	energy norm, even at the spectral bottom. Hyperbolic isometries can
	therefore move a point with nonzero filtered value to the origin and
	produce a nonzero weak limit. The Brezis--Lieb lemma and the splitting
	of the energy then give strong convergence and attainment.
	
	The nonlocal operators also require control of the tails created by
	cutoffs. For the truncated bubbles, we estimate these tails on dyadic
	annuli in the homogeneous Sobolev norm; see
	Proposition~\ref{lem:energy-asymptotics-tilde}.
	For interactions between separated supports, conformal intertwining
	and an explicit spherical transform give the off-diagonal kernels and
	their exponential decay; see
	Proposition~\ref{lem:offdiag-exp-decay-Ptilde}.
	This decay is used to prove that the levels equal $-\infty$ above the
	respective spectral bottoms.
	
	\smallskip
	
	The remainder of this paper is organized as follows. In Section~2, we recall basic facts on the
	Helgason--Fourier transform and on fractional GJMS operators on the hyperbolic space.
	Section~3 is devoted to the Poincar\'e--Sobolev levels associated with $\cP_s$ and $\widetilde{\cP}_s$,
	and contains the proofs of Theorems~\ref{thm:Hardy-lowerbound-Ptilde} and~\ref{prop:gap-bs-lambda011}.
	In Section~4, we analyze the attainability of these Poincar\'e--Sobolev levels and, as a consequence,
	establish the existence of nontrivial solutions to the corresponding Brezis--Nirenberg problems,
	proving Theorems~\ref{mainthe2} and~\ref{mainthe}; we also provide the proofs of Theorems~\ref{prop:strict-gap-fractional} and~\ref{thm:strict-gap-Bplus-Bzero} for
	$H_{n,s}$ and $\widetilde H_{n,s}$.
	Finally, in the Appendix we present the proofs of
	Propositions~\ref{prop:Lp-perturbation-frac-prop} and~\ref{prop:Lp-perturbation-frac-prophn11}.

	\section{Geometric and Analytic Preliminaries}
	In this section, we present two fundamental models of the hyperbolic space, which will be used interchangeably throughout the paper according to the needs of different arguments. We then recall the Helgason--Fourier transform on hyperbolic space and explain its connection with functional calculus, which allows us to define general spectral fractional Laplacian operators. In the final part, we introduce the fractional GJMS operators that constitute the main object of this work, and state a key inequality that will be used repeatedly in the sequel.

	\subsection{Half-Space Model and Poincar\'e Ball Model}
	\label{fouri}
	Throughout the paper, we set
	\[	\rho:=\frac{n-1}{2}, \quad \phi(x):=\frac{2}{1-|x|^2},\ \ n\ge 2, \]
	then the bottom of the $L^2$--spectrum equals
	$\rho^2$.
	
	We first recall the two most classical models of the hyperbolic space. Let $\mathbb{H}^n$ denote the $n$--dimensional hyperbolic space, realized in the upper half--space model
	\[
	\mathbb{H}^n
	:=\{x=(r,z)\in\mathbb{R}\times\mathbb{R}^{n-1}: r>0,\ z\in\mathbb{R}^{n-1}\}.
	\]
	It is endowed with the hyperbolic metric $g_{\mathbb{H}^n}(r,z)=\frac{dr^2+|dz|^2}{r^2}$
	and the corresponding volume element $dV_{\mathbb{H}^n}(r,z)=r^{-n}\,dr\,dz.$
	We write $\nabla_{\mathbb{H}^n}$ and $\Delta_{\mathbb{H}^n}$ for the hyperbolic
	gradient and Laplace--Beltrami operator, respectively. In the coordinates
	$(r,z)$, for any smooth function $u$, one has
	\[\nabla_{\mathbb{H}^n}u(r,z)
	= r^2\big(\partial_r u(r,z)\,\partial_r+\nabla_z u(r,z)\big)\]
	and therefore
	\[|\nabla_{\mathbb{H}^n}u(r,z)|_{g_{\mathbb{H}^n}}^2
	= r^2\Big(|\partial_r u(r,z)|^2+|\nabla_z u(r,z)|^2\Big),\]
	the Laplace--Beltrami operator is given by
	\[\Delta_{\mathbb{H}^n}
	= r^2\big(\partial_r^2+\Delta_z\big)-(n-2)\,r\,\partial_r,\]
	where $\nabla_z$ and $\Delta_z$ denote the Euclidean gradient and Laplacian
	in the $z$--variables.
	
	Moreover, the hyperbolic space can also be identified with the unit ball $\mathbb B^n:=\big\{x\in\mathbb R^n:\ |x|<1\big\}$
	endowed with the metric
	\[  g_{\mathbb B^n}
	=\frac{4\,(dx_1^2+\cdots+dx_n^2)}{(1-|x|^2)^2}
	=\phi(x)^2\,dx^2.\]
	The corresponding volume element is $dV_{\mathbb B^n}(x)=\phi(x)^n\,dx.$
	With this normalization, the Laplace--Beltrami operator in ball
	coordinates reads
	\[\Delta_{\mathbb B^n}
	=\frac{1-|x|^2}{4}\Big\{(1-|x|^2)\sum_{i=1}^n\partial_{x_ix_i}
	+2(n-2)\sum_{i=1}^n x_i\partial_{x_i}\Big\}.\]
	
	The isometry group of the Poincar\'e ball model $(\mathbb B^n,g_{\mathbb B^n})$
	consists precisely of those M\"obius transformations that preserve $\mathbb B^n$.
	Moreover, the hyperbolic volume measure $dV_{\mathbb B^n}$ is invariant under
	these transformations. For any $y\in\mathbb B^n$, define the M\"obius transformation
	$T_y:\mathbb B^n\to\mathbb B^n$ by
	\[	T_y(x)
	:=\frac{|x-y|^2\,y-(1-|y|^2)(x-y)}{1-2x\cdot y+|x|^2|y|^2} 
	\quad{\rm for}\ \,  x\in\mathbb B^n .\]
	A direct computation shows that
	\[
	|T_y(x)|^2
	=\frac{|x-y|^2}{1-2x\cdot y+|x|^2|y|^2},
	\]
	and consequently the hyperbolic distance between $x$ and $y$ admits the Euclidean
	representation
	\[
	\cosh d(x,y)
	=\frac{1+|T_y(x)|^2}{1-|T_y(x)|^2}=1+\frac{2|x-y|^2}{\left(1-|x|^2\right)\left(1-|y|^2\right)}.
	\]
	Equivalently, one has the distance formula
	\[	d(x,y)
	=\log\frac{1+|T_y(x)|}{1-|T_y(x)|}
	\quad\ {\rm for}\ \, x,y\in\mathbb B^n.\]
	
	Using M\"obius transformations, we may define the convolution of measurable
	functions $f$ and $g$ on $\mathbb B^n$ by (see, e.g.,~\cite{LiuPeng2009})
	\begin{equation}\label{sjak}
		(f\ast g)(x)
		:=\int_{\mathbb B^n} f(y)\,g\bigl(T_x(y)\bigr)\,dV_{\mathbb B^n}(y),
	\end{equation}
	whenever the integral is well defined.

	\subsection{Helgason Fourier Transform}
	In this subsection, we briefly recall the Helgason--Fourier analysis on the
	hyperbolic space, working in the Poincar\'e ball model
	$\bigl(\mathbb{B}^n,g_{\mathbb{B}^n}\bigr)$. Since $(\mathbb{B}^n, g_{\mathbb{B}^n})$ is a complete Riemannian
	manifold, the Laplace--Beltrami operator
	\(\Delta_{\mathbb{B}^n}\) with initial domain \(C_c^\infty(\mathbb{B}^n)\subset L^2(\mathbb{B}^n)\) is essentially self--adjoint on
	\(C_c^\infty(\mathbb{B}^n)\).
	We denote its unique self--adjoint extension again
	by \(\Delta_{\mathbb{B}^n}\). The quadratic form associated with \(-\Delta_{\mathbb{B}^n}\) is given by
	\[
	{\bf a}(u,v)
	:= \int_{\mathbb{B}^n}\langle\nabla u,\nabla v\rangle_g\,dV_{\mathbb{B}^n},
	\quad u,v\in C_c^\infty(\mathbb{B}^n),
	\]
	and it extends by closure to a densely defined, closed, nonnegative form on
	\(L^2(\mathbb{B}^n)\) with form domain $\mathcal{D}({\bf a}) = H^1(\mathbb{B}^n),$
	where the integer-order Sobolev space on the hyperbolic space is defined by
	\[
	H^k(\mathbb B^n)
	:=
	\Bigl\{
	u\in L^2(\mathbb B^n): \nabla_{\mathbb B^n}^{\,j}u\in L^2(\mathbb B^n)
	\ \text{for all }0\le j\le k
	\Bigr\},
	\qquad k\in\mathbb N,
	\]
	where \(\nabla_{\mathbb B^n}^{\,j}u\) denotes the \(j\)-th covariant derivative of \(u\). It is equipped with the norm
	\[
	\|u\|_{H^k(\mathbb B^n)}^2
	:=
	\sum_{j=0}^k \|\nabla_{\mathbb B^n}^{\,j}u\|_{L^2(\mathbb B^n)}^2.
	\]
	Since \((\mathbb B^n,g_{\mathbb B^n})\) is a complete manifold of bounded geometry, \(C_c^\infty(\mathbb B^n)\) is dense in \(H^k(\mathbb B^n)\), and hence \(H^k(\mathbb B^n)\) is equivalently the completion of \(C_c^\infty(\mathbb B^n)\) with respect to the above norm.
	
	In particular, \(H^1(\mathbb B^n)\) is the form domain, and hence the natural energy space, associated with \(-\Delta_{\mathbb B^n}\). Since hyperbolic space is a complete Riemannian manifold of bounded geometry, global elliptic regularity yields
	\[
	\mathcal D(-\Delta_{\mathbb B^n})
	=
	\bigl\{
	u\in L^2(\mathbb B^n):
	\Delta_{\mathbb B^n}u\in L^2(\mathbb B^n)
	\text{ in the distributional sense}
	\bigr\}
	=
	H^2(\mathbb B^n).
	\]
	Moreover, the graph norm
	\[
	\|u\|_{L^2(\mathbb B^n)}
	+
	\|\Delta_{\mathbb B^n}u\|_{L^2(\mathbb B^n)}
	\]
	is equivalent to the \(H^2(\mathbb B^n)\)-norm; see
	\cite[Proposition 4.1 and Appendix 1, Lemma 1.4]{shubin1992spectral}.
	We  set
	\[	\lambda(\beta):=\beta^2+\rho^2,\:\:\beta\in \mathbb{R}.\]
	For the basic facts about the Fourier transform on hyperbolic
	space in the Poincar\'e ball model, see 
	\cite{LiuPeng2009,GelfandGindikinGraev2003,Helgason2024}.
	
	For $(\beta,\theta)\in\mathbb R\times\mathbb S^{n-1}$, set
	\begin{equation}\label{def-h}	
		h_{\beta,\theta}(x)
		:=\Bigl(\frac{\sqrt{1-|x|^2}}{|x-\theta|}\Bigr)^{\,n-1-2\beta \bi}, 
		\quad\forall\, x\in\mathbb B^n,
	\end{equation}
	then it satisfies 
	\begin{equation}\label{eq:Helgason-eig}
		-\Delta_{\mathbb{B}^n} h_{\beta,\theta}
		= \bigl(\beta^2+\rho^2\bigr)\,h_{\beta,\theta}.
	\end{equation}
	For $f\in C_c^\infty(\mathbb B^n)$, the Helgason--Fourier transform is defined by
	\begin{equation}\label{eq:Helgason-transform}
		\hat f(\beta,\theta)
		:=
		\int_{\mathbb B^n}
		f(x)\,h_{\beta,\theta}(x)\,dV_{\mathbb B^n}(x),
		\qquad
		(\beta,\theta)\in\mathbb R\times\mathbb S^{n-1}.
	\end{equation}
	This transform extends uniquely to a unitary operator onto a closed
	subspace $\mathscr R$ of the Plancherel space,
	\[
	\mathcal F:
	L^2(\mathbb B^n)
	\longrightarrow
	\mathscr R
	\subset
	L^2\!\Bigl(
	\mathbb R\times\mathbb S^{n-1},
	\frac{d\beta\,d\sigma(\theta)}{|c(\beta)|^2}
	\Bigr).
	\]
	The restriction to $\mathscr R$ is necessary because, when the full
	spectral line $\beta\in\mathbb R$ is used, the components corresponding
	to $\beta$ and $-\beta$ satisfy the Weyl compatibility relation and are
	therefore not independent. In particular, every even spectral
	multiplication operator preserves $\mathscr R$. The measure
	$d\sigma(\theta)$ denotes the normalized surface measure on
	$\mathbb S^{n-1}$, and the corresponding Plancherel identity reads
	\[
	\int_{\mathbb B^n}|f(x)|^2\,dV_{\mathbb B^n}(x)
	=
	\int_{\mathbb R}\int_{\mathbb S^{n-1}}
	|\hat f(\beta,\theta)|^2\,
	\frac{d\sigma(\theta)\,d\beta}{|c(\beta)|^2},
	\qquad
	f\in L^2(\mathbb B^n).
	\]
	Here $c(\beta)$ is the Harish--Chandra $c$--function, with its constant
	normalized for integration over $\mathbb R$ and normalized surface
	measure $d\sigma$ (see \cite{LiuPeng2009}):
	\[
	c(\beta):=2^{\frac{n}{2}}\,\pi^{\frac{n}{4}}
	\sqrt{\Gamma\!\left(\frac{n}{2}\right)}\,
	\frac{\Gamma(\bi\beta)}{\Gamma\!\left(\frac{n-1}{2}+\bi\beta\right)}
	\]
	satisfying
	\[\left|c(\beta)\right|^{-2}
	=\frac{2^{-n}}{\Gamma\left(\frac{n}{2}\right) \pi^{\frac{n}{2}}}
	\frac{\left|\Gamma\!\bigl(\bi\beta+\frac{n-1}{2}\bigr)\right|^2}{\left|\Gamma(\bi\beta)\right|^2}.\]

	Moreover, for $f\in C_c^\infty(\mathbb{B}^n),$ there is the inversion formula
	\begin{equation}\label{inverse}
		f(x)
		= \int_{\mathbb{R}}\int_{\mathbb{S}^{n-1}}
		\overline{h_{\beta,\theta}(x)}\,
		\hat f(\beta,\theta)\,
		\frac{d\sigma(\theta)\,d\beta}{|c(\beta)|^2}.
	\end{equation}
	
	\medskip
	
	A key feature of the Helgason transform is that it diagonalizes the Laplace--Beltrami operator. A direct computation using
	\eqref{eq:Helgason-eig} and  \eqref{eq:Helgason-transform}  shows that, 
	\[	\widehat{\Delta_{\mathbb{B}^n} f}(\beta,\theta)
	= -\bigl(\beta^2+\rho^2\bigr)\,\hat f(\beta,\theta),
	\quad f\in C_c^\infty(\mathbb{B}^n).\]
	
	We next clarify the connection between the Laplace--Beltrami operator and
	multiplication operator under the Helgason transform, which will serve as the basis
	for identifying Fourier symbols of the nonlocal operators considered later.

	\begin{lemma}\label{lem:diag-laplace}
		Let $M_{\lambda}$ be the multiplication operator on $\mathscr R$,
		with its maximal domain:
		\[
		(M_{\lambda}f)(\beta,\theta)
		:= \lambda(\beta)\,f(\beta,\theta).
		\]
		Then, $-\Delta_{\mathbb{B}^n}$ on $L^2(\mathbb{B}^n)$
		is unitarily equivalent to $M_{\lambda}$, that is,
		\begin{equation}\label{eq:diag-laplace}
			-\Delta_{\mathbb{B}^n}
			= \mathcal{F}^{-1} M_{\lambda}\mathcal{F},
		\end{equation}
		with equality of domains. Equivalently, for every
		$f\in H^2(\mathbb{B}^n)$, one has
		\[
		\widehat{-\Delta_{\mathbb{B}^n}f}(\beta,\theta)
		= \lambda(\beta)\,\hat f(\beta,\theta)
		\quad\text{for a.e. }(\beta,\theta)\in\mathbb{R}\times\mathbb{S}^{n-1}.
		\]
	\end{lemma}
	
	\begin{proof}
		Recall that the Helgason transform
		\[
		\mathcal{F}:L^2(\mathbb{B}^n)\longrightarrow
		L^2\!\Bigl(\mathbb{R}\times\mathbb{S}^{n-1},
		\frac{d\beta\,d\sigma(\theta)}{|c(\beta)|^2}\Bigr)
		\]
		is unitary onto $\mathscr R$, with inverse given by the inversion formula, and  for
		$f\in C_c^\infty(\mathbb{B}^n),$ we have
		\begin{equation}\label{eq:laplace-on-core}
			\widehat{\Delta_{\mathbb{B}^n} f}(\beta,\theta)
			= -\bigl(\beta^2+\rho^2\bigr)\,\hat f(\beta,\theta)
			= -\lambda(\beta)\,\hat f(\beta,\theta).
		\end{equation}
		
		Define the operator $\cT := \mathcal{F}^{-1}M_{\lambda}\mathcal{F}$
		with domain
		\[
		\mathcal{D}(\cT)
		:= \Bigl\{f\in L^2(\mathbb{B}^n):
		\lambda(\beta)\,\hat f(\beta,\theta)
		\in L^2\!\Bigl(\mathbb{R}\times\mathbb{S}^{n-1},
		|c(\beta)|^{-2}d\beta d\sigma(\theta)\Bigr)\Bigr\}.
		\]
		Since $M_{\lambda}$ is a self--adjoint multiplication operator on the
		Helgason side and $\mathcal{F}$ is unitary, $\cT$ is self--adjoint.
		
		On the other hand, since \(\mathbb{B}^n\) is complete, the operator
		\(-\Delta_{\mathbb{B}^n}\) with initial domain
		\(C_c^\infty(\mathbb{B}^n)\) is essentially self-adjoint and hence admits
		a unique self-adjoint extension. Since \(\cT\) is self-adjoint and
		\[
		\cT u=-\Delta_{\mathbb{B}^n}u,
		\qquad
		u\in C_c^\infty(\mathbb{B}^n),
		\]
		the uniqueness of the self-adjoint extension implies $\cT=-\Delta_{\mathbb{B}^n}$.
		This yields \eqref{eq:diag-laplace}.
	\end{proof}\medskip

	For every Borel measurable function
	$\Phi:[\rho^2,\infty)\to\mathbb{R},$ the operator $\Phi(-\Delta_{\mathbb{B}^n})$ defined by the functional calculus satisfies $\Phi(-\Delta_{\mathbb{B}^n})
	= \mathcal{F}^{-1}M_{\Phi\circ\lambda}\,\mathcal{F},$
	where $M_{\Phi\circ\lambda}$ is the multiplication operator
	\[
	(M_{\Phi\circ\lambda}F)(\beta,\theta)
	:= \Phi\bigl(\lambda(\beta)\bigr)\,F(\beta,\theta).
	\]
	Here $M_{\Phi\circ\lambda}$ acts on $\mathscr R$ with its maximal domain.
	In particular,
	\[
	\mathcal{D}\bigl(\Phi(-\Delta_{\mathbb{B}^n})\bigr)
	= \Bigl\{f\in L^2(\mathbb{B}^n):
	\Phi\bigl(\lambda(\beta)\bigr)\,\hat f(\beta,\theta)
	\in L^2\!\Bigl(\mathbb{R}\times\mathbb{S}^{n-1},
	|c(\beta)|^{-2}d\beta d\sigma(\theta)\Bigr)\Bigr\},
	\]
	and for $f\in\mathcal{D}(\Phi(-\Delta_{\mathbb{B}^n}))$ one has
	\[
	\cF\big(\Phi(-\Delta_{\mathbb{B}^n})f\big)(\beta,\theta)
	= \Phi\bigl(\lambda(\beta)\bigr)\,\hat f(\beta,\theta),
	\quad f\in  \mathcal{D}\bigl(\Phi(-\Delta_{\mathbb{B}^n})\bigr).
	\]

	\begin{rmk}\label{rem:spectrum-essential-range}
		
		The spectral theorem for multiplication operators on $L^2$ implies that
		the spectrum of a multiplication operator is the essential range of its symbol; see \cite[Appendix A]{KellerLenzWojciechowski2021}. For the even multipliers considered here, restriction to $\mathscr R$ retains the full spectral support in $|\beta|$. Since
		$\mathcal{F}$ is unitary, the spectrum is invariant under this unitary
		equivalence,
		\[
		\sigma(-\Delta_{\mathbb{B}^n})
		= \sigma(M_\lambda)
		= \operatorname{ess\,ran}\lambda
		= \operatorname{ess\,ran}\bigl(\beta^2+\rho^2\bigr).
		\]
		Similarly,
		\[
		\sigma\bigl(\Phi(-\Delta_{\mathbb{B}^n})\bigr)
		= \sigma\bigl(M_{\Phi\circ\lambda}\bigr)
		= \operatorname{ess\,ran}\bigl(\Phi\circ\lambda\bigr)
		= \operatorname{ess\,ran}\bigl(\Phi(\beta^2+\rho^2)\bigr).
		\]
		In particular, if $\Phi$ is continuous, then
		$\sigma\bigl(\Phi(-\Delta_{\mathbb{B}^n})\bigr)
		=\overline{\Phi\bigl(\sigma(-\Delta_{\mathbb{B}^n})\bigr)}$.
	\end{rmk}
	
	We now specialize to the choice $\Phi(\lambda)=\lambda^s,\ s> 0$, see \cite{ChenRLogarithmic}. For $s>0$, we define the spectral fractional Sobolev space 
	\begin{equation}\label{fracaa}
		H^{2s}(\mathbb{B}^n)
		:= \Bigl\{f\in L^2(\mathbb{B}^n):
		(\beta^2+\rho^2)^s\,\hat f(\beta,\theta)
		\in L^2\!\Bigl(\mathbb{R}\times\mathbb{S}^{n-1},
		|c(\beta)|^{-2}d\beta d\sigma(\theta)\Bigr)\Bigr\},
	\end{equation}
	endowed with the norm
	\[
	\|f\|_{H^{2s}(\mathbb{B}^n)}^2
	:= \int_{\mathbb{R}}\int_{\mathbb{S}^{n-1}}
	\bigl(1+(\beta^2+\rho^2)^{2s}\bigr)\,
	|\hat f(\beta,\theta)|^2\,
	\frac{d\sigma(\theta)\,d\beta}{|c(\beta)|^2}.
	\]
	
	For each integer $k\ge1$, the spectral domains agree with the
	classical Sobolev spaces, with equivalent norms:
	\[
	\mathcal D\bigl((-\Delta_{\mathbb B^n})^{k/2}\bigr)=H^k(\mathbb B^n),
	\qquad
	\mathcal D\bigl((-\Delta_{\mathbb B^n})^k\bigr)=H^{2k}(\mathbb B^n).
	\]
	Thus the spectral fractional Laplacian satisfies
	\[\cF\big((-\Delta_{\mathbb{B}^n})^k f\big)(\beta,\theta)
	= \bigl(\beta^2+\rho^2\bigr)^k\,
	\hat f(\beta,\theta)
	\quad{\rm for}\ \, f\in H^{2k}(\mathbb{B}^n).\]
	
	\subsection{Fractional GJMS Operators and Inequality}
	\label{subsec:conf-frac-lap}
	For $\gamma\in \bigl(0,\frac{n}{2}\bigr)\setminus\mathbb N$, recall  the conformal fractional GJMS operator \(\cP_\gamma\) on \(\mathbb H^n\) has the explicit spectral representation (\ref{all pass}). Equivalently, under the
	Helgason--Fourier transform, $\cP_\gamma$ acts as a spectral multiplier:
	\[	\widehat{\cP_\gamma f}(\beta,\theta)
	= m_\gamma(\beta)\,\hat f(\beta,\theta),
	\qquad
	m_\gamma(\beta)
	:= 2^{2\gamma}\,
	\frac{\bigl|\Gamma\!\bigl(\frac{3+2\gamma}{4}+\frac{\bi}{2}\beta\bigr)\bigr|^2}
	{\bigl|\Gamma\!\bigl(\frac{3-2\gamma}{4}+\frac{\bi}{2}\beta\bigr)\bigr|^2}.\]
	
	However, $\cP_\gamma$ does not satisfy the following intertwining relations (\ref{eq:tildePs-halfspace-intertwine11}) and (\ref{eq:tildePs-ball-intertwine11}). To recover a direct intertwining with the Euclidean fractional Laplacian under the
	conformal identifications of $\mathbb H^n$ with $\mathbb R^n_+$ and $\mathbb B^n$,
	Lu and Yang introduced an auxiliary family of fractional operators
	$\widetilde \cP_\gamma$ on $\mathbb H^n$ in \cite[Theorem 1.7]{lu2023explicit}.  Equivalently, under
	the Helgason--Fourier transform one has the multiplier representation
	\begin{equation}\label{mstl}
		\widehat{\widetilde \cP_\gamma f}(\beta,\theta)
		=\widetilde m_\gamma(\beta)\,\hat f(\beta,\theta),
		\qquad
		\widetilde m_\gamma(\beta)
		:=\frac{\bigl|\Gamma\!\bigl(\gamma+\frac12+\bi\beta\bigr)\bigr|^2}
		{\bigl|\Gamma\!\bigl(\frac12+\bi\beta\bigr)\bigr|^2}.
	\end{equation}
	
	We denote the bottoms of the $L^2$ spectra by
	\[
	\lambda_{0,s}^{\mathrm{conf}}
	:=\inf\sigma(\cP_s),
	\qquad
	\widetilde\lambda_{0,s}^{\mathrm{conf}}
	:=\inf\sigma(\widetilde \cP_s).
	\]
	Since $\cA$ has spectrum $\sigma(\cA)=[0,\infty)$ and both $\cP_s$ and
	$\widetilde \cP_s$ are spectral multipliers of $\cA$, we have
	\begin{equation}\label{eq:bottom-spectrum-by-symbol}
		\lambda_{0,s}^{\mathrm{conf}}
		=\inf_{\beta\ge0} m_s(\beta),
		\qquad
		\widetilde\lambda_{0,s}^{\mathrm{conf}}
		=\inf_{\beta\ge0} \widetilde m_s(\beta),
	\end{equation}
	By the classical Gamma-function product formula
	(see \cite[Eq.~(2.14)]{lu2023explicit}),
	\[\bigl|\Gamma(a+\bi\lambda)\bigr|^{2}
	=
	\bigl|\Gamma(a)\bigr|^{2}
	\prod_{k=0}^{\infty}\frac{1}{1+\dfrac{\lambda^{2}}{(a+k)^{2}}},
	\qquad a\in\mathbb R,\ a\neq 0,-1,-2,\ldots\]
	we can obtain
	\[\frac{\bigl|\Gamma(a+\bi\lambda)\bigr|}{\bigl|\Gamma(b+\bi\lambda)\bigr|}
	\ge
	\frac{\bigl|\Gamma(a)\bigr|}{\bigl|\Gamma(b)\bigr|}
	\quad \text{provided that }\ (a+k)^2\ge (b+k)^2,\ \forall k\in\mathbb N_0\]
	and
	\[	\lambda_{0,s}^{\mathrm{conf}}
	= m_s(0)=2^{2s}\,
	\frac{\Gamma\!\bigl(\frac{3+2s}{4}\bigr)^2}
	{\Gamma\!\bigl(\frac{3-2s}{4}\bigr)^2},
	\qquad
	\widetilde\lambda_{0,s}^{\mathrm{conf}}
	=\widetilde m_s(0)=\frac{\Gamma\!\bigl(s+\frac12\bigr)^2}{\Gamma\!\bigl(\frac12\bigr)^2}.\]
	At $s=\frac32+2j$, $j\in\mathbb N_0$, the reciprocal Gamma
	factor is interpreted by continuity and $m_s(0)=0$.
	By \cite[Lemma 5.2]{lu2023explicit}, we have
	\begin{equation}
		\lambda_{0,s}^{\mathrm{conf}}
		-
		\widetilde\lambda_{0,s}^{\mathrm{conf}}
		=\frac{\sin(\pi s)}{\pi}\,\Gamma\bigl(s+\tfrac12\bigr)^2.
	\end{equation}

	Since the Helgason--Fourier
	transform diagonalizes $\cP_k$:
	\begin{equation}\label{eq:Pk-factor}
		\cP_k
		= \cP_1\,(\cP_1+2)\cdots\bigl(\cP_1+k(k-1)\bigr)
		=\prod_{j=1}^k\bigl(\cA^2+(j-\frac12)^2\bigr).
	\end{equation}
	that is, for every $f\in C_c^\infty(\mathbb H^n)$,
	\begin{equation}\label{eq:Pk-multiplier}
		\widehat{\cP_k f}(\beta,\theta)
		= m_k(\beta)\,\hat f(\beta,\theta),
		\quad\forall (\beta,\theta)\in\mathbb R\times\mathbb S^{n-1},
	\end{equation}
	with radial multiplier
	\begin{equation}\label{eq:mk-product}
		m_k(\beta)
		=\prod_{j=1}^k\bigl(\beta^2+ (j-\frac12)^2\bigr).
	\end{equation}
	Using
	$\beta^2+\bigl(j-\tfrac12\bigr)^2=\bigl(j-\tfrac12+\bi\beta\bigr)\bigl(j-\tfrac12-\bi\beta\bigr)$
	and the identity $\prod_{j=1}^k(z+j-\frac12)=\Gamma(z+k+\frac12)/\Gamma(z+\frac12)$,
	we obtain the equivalent Gamma--function closed form
	\begin{equation}\label{eq:mk-gamma}
		m_k(\beta)
		=
		\frac{\Gamma\!\bigl(k+\frac12+\bi\beta\bigr)}{\Gamma\!\bigl(\frac12+\bi\beta\bigr)}
		\frac{\Gamma\!\bigl(k+\frac12-\bi\beta\bigr)}{\Gamma\!\bigl(\frac12-\bi\beta\bigr)}
		=
		\frac{\bigl|\Gamma\!\bigl(k+\frac12+\bi\beta\bigr)\bigr|^2}
		{\bigl|\Gamma\!\bigl(\frac12+\bi\beta\bigr)\bigr|^2}.
	\end{equation}
	In particular, this coincides with the multiplier $\widetilde m_\gamma(\beta)$ of
	$\widetilde \cP_\gamma$ evaluated at $\gamma=k$.
	
	Moreover, when $\gamma=k\in\mathbb N$, the bottom of the spectrum satisfies
	\[
	\widetilde\lambda_{0,k}^{\mathrm{conf}}
	= \frac{\Gamma\!\bigl(k+\tfrac12\bigr)^2}{\Gamma\!\bigl(\tfrac12\bigr)^2}
	= \prod_{j=1}^{k}\Bigl(j-\tfrac12\Bigr)^{\!2}
	= \prod_{j=1}^{k}\frac{(2j-1)^{2}}{4},
	\]
	which equals the bottom of the spectrum of $\cP_k$.

	It is well known that the sharp fractional Sobolev constant in the Euclidean
	setting is independent of the underlying domain. In particular, applying
	\eqref{eq:upper-bound} to the unit ball
	\(\mathbb B^n\), we obtain
	\begin{equation}\label{eq:Eucl-frac-Sobolev}
		S_{n,s}
		\Bigl(\int_{\mathbb B^n}|v|^{2^*_s}\,dx\Bigr)^{\!\frac{2}{2^*_s}}
		\le \int_{\mathbb B^n} v\,(-\Delta)^s v\,dx,
		\quad\forall v\in C_c^\infty(\mathbb B^n).
	\end{equation}
	
	By conformal invariance, the Euclidean inequality \eqref{eq:Eucl-frac-Sobolev} is
	equivalent to a hyperbolic inequality involving $\widetilde \cP_s$. For the integer-order cases
	$s=1$ and, more generally, $s=k\ge 2$, we refer to alternative proofs in
	\cite{hebey1999nonlinear,liu2013sharp}. In the fractional setting, Lu and Yang \cite{lu2023explicit} established the result for
	$n\ge3$ and $\frac{n-1}{2}\le s<\tfrac n2$. Below we present a unified proof valid for
	all $s\in(0,\tfrac n2)$.

	\begin{proposition}\label{prop:conf-frac-Sob-Hn}
		Let $0<s<\tfrac n2$, then the following Sobolev inequality holds:
		\begin{equation}\label{eq:conf-frac-Sob-Hn}
			S_{n,s}
			\Bigl(\int_{\mathbb H^n}|u|^{2^*_s}\,dV_{\mathbb H^n}\Bigr)^{\!\frac{2}{2^*_s}}
			\;\le\;
			\int_{\mathbb H^n} u\,\widetilde\cP_s u\,dV_{\mathbb H^n},
			\qquad u\in C_c^\infty(\mathbb H^n),
		\end{equation}
		where $2_s^*=\frac{2n}{n-2s}$ and $S_{n,s}$ is defined in \eqref{eq:upper-bound}. Moreover, the constant $S_{n,s}$
		in \eqref{eq:conf-frac-Sob-Hn} is optimal.
	\end{proposition}
	
	\begin{proof}
		Without loss of generality, we work in the Poincar\'e ball model
		\((\mathbb B^n,g_{\mathbb B^n})\) of the hyperbolic space, where
		\[
		g_{\mathbb B^n}=\phi(x)^2\,dx^2,
		\qquad 
		\phi(x):=\frac{2}{1-|x|^2}.
		\]
		For \(u\in C_c^\infty(\mathbb B^n)\), we introduce the critical conformal
		transformation $v:=\phi^{\frac{n-2s}{2}}u\in C_c^{\infty}(\mathbb B^n).$
		
		By the exact intertwining identity (\ref{eq:tildePs-ball-intertwine11}) for the operator $\widetilde\cP_s$, we have
		\[
		\widetilde\cP_s u
		\;=\;
		\phi^{-\frac{n+2s}{2}}\,
		(-\Delta)^s\!\bigl(\phi^{\frac{n-2s}{2}}u\bigr)
		\;=\;
		\phi^{-\frac{n+2s}{2}}\,(-\Delta)^s v,
		\]
		and since $dV_{{\mathbb B^n}}=\phi(x)^n\,dx$, we obtain the exact energy identity
		\begin{equation}\label{eq:energy-identity-tildeP}
			\int_{\mathbb B^n} u\,\widetilde\cP_s u\,dV_{\mathbb B^n}
			\;=\;
			\int_{\mathbb B^n} v\,(-\Delta)^s v\,dx.
		\end{equation}
		At the critical exponent $2_s^*=\frac{2n}{n-2s}$, one also has the conformal
		invariance of the $L^{2_s^*}$ norm:
		\begin{equation}\label{eq:Lp-invariance}
			\int_{\mathbb B^n} |u|^{2_s^*}\,dV_{\mathbb B^n}
			\;=\;
			\int_{\mathbb B^n} |v|^{2_s^*}\,dx.
		\end{equation}
		Applying the sharp Euclidean fractional Sobolev inequality to $v$ and then using
		\eqref{eq:energy-identity-tildeP}--\eqref{eq:Lp-invariance} yields
		\[
		S_{n,s}\Bigl(\int_{\mathbb B^n}|u|^{2_s^*}\,dV_{\mathbb B^n}\Bigr)^{\!2/2_s^*}
		\;\le\;
		\int_{\mathbb B^n} v\,(-\Delta)^s v\,dx
		\;=\;
		\int_{\mathbb B^n} u\,\widetilde\cP_s u\,dV_{\mathbb B^n},
		\]
		which is \eqref{eq:conf-frac-Sob-Hn}.
		
		To prove sharpness, let $\{w_\varepsilon\}\subset C_c^\infty(\mathbb B^n)\setminus \left\{0\right\}$ be a sequence satisfying
		\[
		\frac{\displaystyle\int_{\mathbb B^n} w_\varepsilon\,(-\Delta)^s w_\varepsilon\,dx}
		{\Bigl(\displaystyle\int_{\mathbb B^n}|w_\varepsilon|^{2_s^*}\,dx\Bigr)^{\!2/2_s^*}}
		\to  S_{n,s}\quad {\rm as}\ \, \varepsilon \rightarrow 0^+.
		\]
		Define $u_\varepsilon\in C_c^\infty(\mathbb B^n)$ by
		$u_\varepsilon:=\phi^{-\frac{n-2s}{2}}\,w_\varepsilon$, then by \eqref{eq:energy-identity-tildeP}--\eqref{eq:Lp-invariance},
		\[
		\frac{\displaystyle\int_{\mathbb B^n} u_\varepsilon\,\widetilde\cP_s u_\varepsilon\,dV_{\mathbb B^n}}
		{\Bigl(\displaystyle\int_{\mathbb B^n}|u_\varepsilon|^{2_s^*}\,dV_{\mathbb B^n}\Bigr)^{\!2/2_s^*}}
		=
		\frac{\displaystyle\int_{\mathbb B^n} w_\varepsilon\,(-\Delta)^s w_\varepsilon\,dx}
		{\Bigl(\displaystyle\int_{\mathbb B^n}|w_\varepsilon|^{2_s^*}\,dx\Bigr)^{\!2/2_s^*}}
		\to S_{n,s}\quad {\rm as}\ \, \varepsilon \rightarrow 0^+.
		\]
		Hence the best constant on $\mathbb B^n$ equals $S_{n,s}$, proving optimality.
	\end{proof}

	\section{Hardy Lower Bounds for Fractional Hardy--Sobolev--Maz'ya Inequalities}
	This section develops the key test-function estimates needed to analyze the fractional Poincar\'e--Sobolev levels associated with $\widetilde{\cP}_s$ and ${\cP}_s$.  We introduce the standard bubble $U_\varepsilon$ and its cut-off version $w_\varepsilon$, and quantify precisely how the truncation affects the key quantities: the critical norm, the $L^2$-mass of the corresponding test function $u_\varepsilon$, and the fractional Dirichlet energy. 
	These estimates provide the core test-function machinery for the sequel, and are the main input for the proofs of Theorem~\ref{thm:Hardy-lowerbound-Ptilde} and Theorem~\ref{prop:gap-bs-lambda011}.

	\subsection{Asymptotic Analysis}
	
	In this subsection, without loss of generality, we continue working in the ball model of $\mathbb B^n$ and recall
	\[
	\phi(x)=\frac{2}{1-|x|^2},\quad
	dV_{\mathbb B^n}=\phi(x)^n\,dx
	\quad\text{on }\mathbb B^n.
	\]

	\begin{lemma}
		\label{lem:exact-reduction}
		Let $u\in C_c^\infty(\mathbb B^n)$ and 
		\[
		w(x):=\phi(x)^{\frac n2-s}u(x)
		\quad\text{for }x\in\mathbb B^n,
		\qquad
		w(x)=0\quad\text{for }x\in\mathbb R^n\setminus\mathbb B^n.
		\]
		Then $w\in C_c^\infty(\mathbb R^n)$ and the following identities hold:
		\[	\int_{\mathbb B^n} (\widetilde \cP_s u)\,u\,dV_{{\mathbb B}^n}= \int_{\mathbb R^n} w\,(-\Delta)^s w\,dx,\quad
		\int_{\mathbb B^n}|u|^{2_s^*}\,dV_{{\mathbb B}^n}
		= \int_{\mathbb R^n}|w|^{2_s^*}\,dx.
		\]
	\end{lemma}
	
	\begin{proof}
		The argument is identical to that of Proposition \ref{prop:conf-frac-Sob-Hn}, hence we omit the details.
	\end{proof}\medskip
	
	Let
	\[
	U(x)=(1+|x|^2)^{-\frac{n-2s}{2}},\qquad
	U_\varepsilon(x)=\varepsilon^{-\frac{n-2s}{2}}U\Bigl(\frac{x}{\varepsilon}\Bigr),\quad\forall  x\in\R^n
	\]
	and fix $0<\delta<\frac14$. Choose $\eta\in C_c^\infty(\mathbb B^n)$ such that
	$0\le \eta\le 1$, $\eta\equiv 1$ on $B_\delta(0)$ and $\eta\equiv 0$ on
	$\mathbb B^n\setminus B_{2\delta}(0)$. Define
	\[
	w_\varepsilon:=\eta\,U_\varepsilon\in C_c^\infty(\mathbb R^n),
	\qquad
	u_\varepsilon:=\phi^{s-\frac n2}w_\varepsilon\in C_c^\infty(\mathbb B^n).
	\]
	
	We next quantify the effect of the cut-off on the critical $L^{2_s^*}$ norm.

	\begin{proposition}\label{lem:L2sstar-cutoff}
		Let $s\in(0,\frac n2)$, then
		\[	\int_{\mathbb R^n}|w_\varepsilon|^{2_s^*}\,dx
		=\int_{\mathbb R^n}|U|^{2_s^*}\,dx+O(\varepsilon^{n})
		\quad{\rm as}\ \,  \varepsilon\to0^+,\]
		where the implicit constant depends on $n,s$ and the fixed cutoff.
	\end{proposition}
	
	\begin{proof}
		Write $p:=2_s^*=\frac{2n}{n-2s}$. Since $\frac{n-2s}{2}\,p=n$, we have the exact identity
		\[	|U_\varepsilon(x)|^p
		=\varepsilon^{-n}\Bigl(1+\frac{|x|^2}{\varepsilon^2}\Bigr)^{-n}.\]
		By the change of variables $x=\varepsilon y$,
		\[
		\int_{\mathbb R^n}|U_\varepsilon(x)|^p\,dx
		=\int_{\mathbb R^n}\varepsilon^{-n}(1+|x|^2/\varepsilon^2)^{-n}\,dx
		=\int_{\mathbb R^n}(1+|y|^2)^{-n}\,dy
		=\int_{\mathbb R^n}|U(y)|^p\,dy.
		\]
		Since $w_\varepsilon=\eta U_\varepsilon$ and $0\le \eta\le1$,
		\[
		\int_{\mathbb R^n}|w_\varepsilon|^p\,dx
		=\int_{\mathbb R^n}\eta(x)^p|U_\varepsilon(x)|^p\,dx
		=\int_{\mathbb R^n}|U_\varepsilon|^p\,dx
		-\int_{\mathbb R^n}(1-\eta^p)|U_\varepsilon|^p\,dx.
		\]
		Moreover, $1-\eta^p$ is supported in $\mathbb R^n\setminus B_\delta(0)$, hence
		\[	\int_{\mathbb R^n}(1-\eta^p)|U_\varepsilon|^p\,dx
		\le
		\int_{|x|\ge\delta}|U_\varepsilon(x)|^p\,dx.\]
		Using again $x=\varepsilon y$,
		\[
		\int_{|x|\ge\delta}|U_\varepsilon(x)|^p\,dx
		=\int_{|y|\ge \delta/\varepsilon}(1+|y|^2)^{-n}\,dy.
		\]
		For $|y|\ge1$ we have $(1+|y|^2)^{-n}\le |y|^{-2n}$, and thus by polar coordinates,
		for $R\ge1$,
		\[
		\int_{|y|\ge R}(1+|y|^2)^{-n}\,dy
		\le |\mathbb S^{n-1}|\int_R^\infty r^{n-1}r^{-2n}\,dr
		= \frac{|\mathbb S^{n-1}|}{n}\,R^{-n},
		\]
		Taking $R=\delta/\varepsilon$  yields $	\int_{|x|\ge\delta}|U_\varepsilon(x)|^p\,dx=O(\varepsilon^n),$
		which gives
		\[
		\int_{\mathbb R^n}|w_\varepsilon|^p\,dx
		=\int_{\mathbb R^n}|U_\varepsilon|^p\,dx + O(\varepsilon^n)
		=\int_{\mathbb R^n}|U|^p\,dx + O(\varepsilon^n)\quad {\rm as}\ \, \varepsilon \rightarrow 0^+,
		\]
		which completes the proof.
	\end{proof}
	
	We next derive the $L^2$-mass estimate in the hyperbolic space. 
	In contrast to the Euclidean case, the asymptotic behavior is influenced by the conformal weight and must be analyzed with respect to the hyperbolic volume element.

	\begin{proposition}\label{lem:L2-cutoff}
		Let $0<s<\frac n2$. With the fixed cutoff $\eta$ chosen above, as
		$\varepsilon\to0^+$,
		\begin{equation}\label{eq:L2-asymptotics-tilde}
			\int_{\mathbb B^n}|u_\varepsilon|^2\,dV_{\mathbb B^n}
			=
			\begin{cases}
				c_{1,s}\,\varepsilon^{2s}+o(\varepsilon^{2s}),
				& n>4s,\\[1mm]
				c_{2,s}\,\varepsilon^{2s}|\log\varepsilon|
				+O(\varepsilon^{2s}),
				& n=4s,\\[1mm]
				c_{3,s}\,\varepsilon^{n-2s}
				+o(\varepsilon^{n-2s}),
				& 2s<n<4s.
			\end{cases}
		\end{equation}
		Here
		\[
		c_{1,s}
		=
		2^{2s}
		\int_{\mathbb R^n}(1+|y|^2)^{-(n-2s)}\,dy
		\quad\text{when }n>4s,
		\]
		$c_{2,s}
		=
		2^{2s}|\mathbb S^{n-1}|,$
		and, when $2s<n<4s$,
		\[
		c_{3,s}
		=
		\int_{\mathbb B^n}
		\eta(x)^2\phi(x)^{2s}
		|x|^{-2(n-2s)}\,dx.
		\]
	\end{proposition}
	
	\begin{proof}
		Since
		\[
		u_\varepsilon
		=
		\phi^{s-\frac n2}w_\varepsilon,
		\qquad
		w_\varepsilon=\eta U_\varepsilon,
		\qquad
		dV_{\mathbb B^n}=\phi^n\,dx,
		\]
		we have $|u_\varepsilon|^2\,dV_{\mathbb B^n}
		=
		\phi^{2s}|w_\varepsilon|^2\,dx.$
		Moreover,
		\[
		U_\varepsilon(x)
		=
		\varepsilon^{-\frac{n-2s}{2}}
		\left(1+\frac{|x|^2}{\varepsilon^2}\right)^{-\frac{n-2s}{2}}
		=
		\frac{\varepsilon^{\frac{n-2s}{2}}}
		{(\varepsilon^2+|x|^2)^{\frac{n-2s}{2}}}.
		\]
		Hence, setting $F(x):=\eta(x)^2\phi(x)^{2s}$
		on $\mathbb B^n$ and extending $F$ by zero outside $\mathbb B^n$, we obtain
		\begin{equation}\label{eq:L2-F-representation}
			\int_{\mathbb B^n}|u_\varepsilon|^2\,dV_{\mathbb B^n}
			=
			\varepsilon^{n-2s}
			\int_{\mathbb R^n}
			\frac{F(x)}
			{(\varepsilon^2+|x|^2)^{n-2s}}\,dx.
		\end{equation}
		Since $\eta$ is compactly supported in $\mathbb B^n$, we have
		$F\in C_c^\infty(\mathbb R^n)$. Furthermore, $F(0)=\phi(0)^{2s}=2^{2s}.$
		
		Assume first that $n>4s$. Making the change of variables
		$x=\varepsilon y$ in \eqref{eq:L2-F-representation}, we obtain
		\[
		\varepsilon^{-2s}
		\int_{\mathbb B^n}|u_\varepsilon|^2\,dV_{\mathbb B^n}
		=
		\int_{\mathbb R^n}
		F(\varepsilon y)(1+|y|^2)^{-(n-2s)}\,dy.
		\]
		Since $n>4s,$
		the function $(1+|y|^2)^{-(n-2s)}$
		belongs to $L^1(\mathbb R^n)$. Therefore, by dominated convergence,
		\[
		\varepsilon^{-2s}
		\int_{\mathbb B^n}|u_\varepsilon|^2\,dV_{\mathbb B^n}
		\longrightarrow
		F(0)
		\int_{\mathbb R^n}
		(1+|y|^2)^{-(n-2s)}\,dy
		=
		c_{1,s}.
		\]
		This proves
		\[
		\int_{\mathbb B^n}|u_\varepsilon|^2\,dV_{\mathbb B^n}
		=
		c_{1,s}\varepsilon^{2s}+o(\varepsilon^{2s}).
		\]
		
		Suppose next that $n=4s$. Since $\eta\equiv1$ on $B_\delta(0)$ and
		\[
		\phi(x)^{2s}
		=
		\left(\frac{2}{1-|x|^2}\right)^{2s}
		=
		2^{2s}+O(|x|^2)
		\qquad\text{as }x\to0,
		\]
		we have $F(x)=2^{2s}+O(|x|^2)$ on $B_\delta(0).$
		Since $n-2s=n/2$, the contribution of the constant part is
		\[
		2^{2s}\varepsilon^{2s}
		\int_{B_\delta}
		\frac{dx}{(\varepsilon^2+|x|^2)^{n/2}}.
		\]
		Using polar coordinates and then $r=\varepsilon t$,
		\begin{align*}
			\int_{B_\delta}
			\frac{dx}{(\varepsilon^2+|x|^2)^{n/2}}
			&=
			|\mathbb S^{n-1}|
			\int_0^\delta
			\frac{r^{n-1}}{(\varepsilon^2+r^2)^{n/2}}\,dr=
			|\mathbb S^{n-1}|
			|\log\varepsilon|+O(1),
		\end{align*}
		because
		\[
		\frac{t^{n-1}}{(1+t^2)^{n/2}}
		=
		\frac1t+O(t^{-3})
		\qquad\text{as }t\to\infty.
		\]
		Thus the constant part contributes $2^{2s}|\mathbb S^{n-1}|
		\varepsilon^{2s}|\log\varepsilon|
		+
		O(\varepsilon^{2s}).$
		
		For the remainder on $B_\delta$, using
		$|F(x)-F(0)|\le C|x|^2$, we obtain
		\[
		\varepsilon^{2s}
		\int_{B_\delta}
		\frac{|F(x)-F(0)|}
		{(\varepsilon^2+|x|^2)^{n/2}}\,dx
		\le
		C\varepsilon^{2s}
		\int_0^\delta
		\frac{r^{n+1}}
		{(\varepsilon^2+r^2)^{n/2}}\,dr
		=
		O(\varepsilon^{2s}).
		\]
		On $\operatorname{supp} F\setminus B_\delta$, one has $|x|\ge\delta$, and therefore
		\[
		\varepsilon^{2s}
		\int_{\mathbb R^n\setminus B_\delta}
		\frac{|F(x)|}
		{(\varepsilon^2+|x|^2)^{n/2}}\,dx
		=
		O(\varepsilon^{2s}).
		\]
		Consequently,
		\[
		\int_{\mathbb B^n}|u_\varepsilon|^2\,dV_{\mathbb B^n}
		=
		c_{2,s}\varepsilon^{2s}|\log\varepsilon|
		+
		O(\varepsilon^{2s}),
		\]
		with $c_{2,s}=2^{2s}|\mathbb S^{n-1}|.$
		
		Finally, assume that $2s<n<4s$. From
		\eqref{eq:L2-F-representation},
		\[
		\varepsilon^{-(n-2s)}
		\int_{\mathbb B^n}|u_\varepsilon|^2\,dV_{\mathbb B^n}
		=
		\int_{\mathbb R^n}
		\frac{F(x)}
		{(\varepsilon^2+|x|^2)^{n-2s}}\,dx.
		\]
		For every $x\neq0$,
		\[
		\frac{F(x)}
		{(\varepsilon^2+|x|^2)^{n-2s}}
		\longrightarrow
		F(x)|x|^{-2(n-2s)}.
		\]
		Moreover,
		\[
		0\le
		\frac{F(x)}
		{(\varepsilon^2+|x|^2)^{n-2s}}
		\le
		F(x)|x|^{-2(n-2s)}.
		\]
		Since $F$ is compactly supported, it only remains to check integrability
		near the origin. This follows from
		\[
		2(n-2s)<n
		\quad\Longleftrightarrow\quad
		n<4s.
		\]
		Hence $F(x)|x|^{-2(n-2s)}\in L^1(\mathbb R^n),$
		and dominated convergence yields
		\[
		\varepsilon^{-(n-2s)}
		\int_{\mathbb B^n}|u_\varepsilon|^2\,dV_{\mathbb B^n}
		\longrightarrow
		\int_{\mathbb R^n}
		F(x)|x|^{-2(n-2s)}\,dx
		=
		c_{3,s}.
		\]
		Therefore
		\[
		\int_{\mathbb B^n}|u_\varepsilon|^2\,dV_{\mathbb B^n}
		=
		c_{3,s}\varepsilon^{n-2s}
		+
		o(\varepsilon^{n-2s}).
		\]
		In this regime both the conformal weight and the cutoff region enter the
		leading-order coefficient, and hence $c_{3,s}$ depends on the fixed
		cutoff $\eta$.
	\end{proof}
	
	The next derivative estimate controls cutoff errors away from the concentration point in the proof of Proposition~\ref{lem:energy-asymptotics-tilde}.
	
	\begin{lemma}\label{lem:derivative-bound-Ueps}
		Let $0<s<\frac n2$ and set $a=\frac{n-2s}{2}$. For every multi-index $\alpha\in\mathbb N_0^n$ and every $\delta>0$ there
		exists $C_{\alpha,\delta}>0$ such that for all $\varepsilon\in(0,1)$ and all
		$x\in\mathbb R^n$ with $|x|\ge\delta$,
		\begin{equation}\label{eq:derivative-bound-Ueps}
			|\partial^\alpha U_\varepsilon(x)|
			\le C_{\alpha,\delta}\,\varepsilon^{\frac{n-2s}{2}}\,|x|^{-(n-2s+|\alpha|)}.
		\end{equation}
	\end{lemma}
	
	\begin{proof}
		By the chain rule, for any multi-index $\alpha$,
		\begin{equation}\label{eq:scale-derivative}
			\partial^\alpha U_\varepsilon(x)
			=\varepsilon^{-a-|\alpha|}\,(\partial^\alpha U)\!\left(\frac{x}{\varepsilon}\right).
		\end{equation}
		
		We claim that there exists $C_\alpha>0$ such that
		\begin{equation}\label{eq:decay-derivative-U}
			|\partial^\alpha U(y)|
			\le C_\alpha\,(1+|y|)^{-(n-2s+|\alpha|)}
			\qquad\forall\,y\in\mathbb R^n.
		\end{equation}
		Indeed, there exist a polynomial $P_\alpha$ on
		$\mathbb R^n$ and an integer $m=m(\alpha)\in\mathbb N_0$ such that
		\[\partial^\alpha(1+|y|^2)^{-a}
		=\frac{P_\alpha(y)}{(1+|y|^2)^{a+m}},\]
		with $m=|\alpha|$ and $\deg P_\alpha\le |\alpha|$, as follows by induction on $|\alpha|$.
		Consequently, for $|y|\ge1$,
		\[
		|\partial^\alpha U(y)|
		\le C_\alpha\,\frac{|y|^{|\alpha|}}{|y|^{2(a+m)}}
		\le C_\alpha\,|y|^{-2a-2|\alpha|+|\alpha|}
		\le C_\alpha\,|y|^{-(2a+|\alpha|)}
		= C_\alpha\,|y|^{-(n-2s+|\alpha|)},
		\]
		and \eqref{eq:decay-derivative-U} follows. Combining \eqref{eq:scale-derivative} and \eqref{eq:decay-derivative-U} yields
		\[
		|\partial^\alpha U_\varepsilon(x)|
		\le \varepsilon^{-a-|\alpha|}\,
		C_\alpha \Bigl(1+\Bigl|\frac{x}{\varepsilon}\Bigr|\Bigr)^{-(n-2s+|\alpha|)}
		\le C_\alpha\,\varepsilon^{-a-|\alpha|}
		\Bigl|\frac{x}{\varepsilon}\Bigr|^{-(n-2s+|\alpha|)}= C_\alpha\,|x|^{-(n-2s+|\alpha|)}\,\varepsilon^{a}
		\]
		for every $\varepsilon\in(0,1)$ and $x\ne0$, proving \eqref{eq:derivative-bound-Ueps}.
		In particular, since $|x|\ge\delta$, we have $|x|^{-(n-2s+|\alpha|)}\le \delta^{-(n-2s+|\alpha|)}$,
		and therefore
		\[
		|\partial^\alpha U_\varepsilon(x)|
		\le C_{\alpha,\delta}\,\varepsilon^{a}
		=C_{\alpha,\delta}\,\varepsilon^{\frac{n-2s}{2}},
		\]
		which completes the proof.
	\end{proof}
	
	\smallskip
	
	We next establish a key energy asymptotic expansion for the truncated bubbles.
	The first-difference singular-integral formula used for $s\in(0,1)$
	in~\cite{servadei2015brezis} does not apply throughout the present range
	of $s$. We use the Fourier definition of $(-\Delta)^s$ to derive the
	desired asymptotics.

	\begin{proposition}\label{lem:energy-asymptotics-tilde}
		Let $0<s<\frac n2$.
		Then, as $\varepsilon\to0$,
		\begin{equation}\label{eq:energy-asymptotics-tilde-lemma}
			\int_{\mathbb R^n} w_\varepsilon\,(-\Delta)^s w_\varepsilon\,dx
			=
			\bigl\|(-\Delta)^{\frac s2}U\bigr\|_{L^2(\mathbb R^n)}^2
			\;+\;O(\varepsilon^{\,n-2s}).
		\end{equation}
		In particular,
		\begin{equation}\label{eq:energy-asymptotics-tilde}
			\int_{\mathbb R^n} w_\varepsilon\,(-\Delta)^s w_\varepsilon\,dx
			=
			S_{n,s}\Bigl(\int_{\mathbb R^n}|U|^{2_s^*}\,dx\Bigr)^{\!\frac{2}{2_s^*}}
			\;+\;O(\varepsilon^{\,n-2s}).
		\end{equation}
	\end{proposition}
	
	\begin{proof}
		Throughout the proof, we use the Fourier definition
		\[
		\widehat{(-\Delta)^s f}(\xi)=|\xi|^{2s}\hat f(\xi),\qquad
		\int_{\mathbb R^n} f\,(-\Delta)^s f\,dx
		=\int_{\mathbb R^n}|\xi|^{2s}|\hat f(\xi)|^2\,d\xi
		=\bigl\|(-\Delta)^{\frac s2}f\bigr\|_{L^2(\mathbb R^n)}^2.
		\]
		Set $z_\varepsilon:=w_\varepsilon-U_\varepsilon=(\eta-1)U_\varepsilon,$
		so that $z_\varepsilon\equiv0$ on $B_\delta(0)$ and
		$\operatorname{supp}(z_\varepsilon)\subset \mathbb{R}^n\setminus B_\delta(0)$.
		Then, since $w_{\varepsilon}\in C_c^{\infty}(\mathbb{R}^n),$
		\begin{equation}\label{eq:energy-expand}
			\int w_\varepsilon(-\Delta)^s w_\varepsilon
			=\bigl\|(-\Delta)^{\frac s2}U_\varepsilon\bigr\|_{L^2(\mathbb R^n)}^2
			+2\int (-\Delta)^{\frac{s}{2}} z_\varepsilon \cdot (-\Delta)^{\frac{s}{2}} U_\varepsilon
			+\bigl\|(-\Delta)^{\frac s2}z_\varepsilon\bigr\|_{L^2(\mathbb R^n)}^2.
		\end{equation}
		
		\smallskip\noindent
		\textbf{Step 1: Scaling invariance of the main term.}
		We claim that
		\begin{equation}\label{eq:Ueps-energy-invariant-fourier}
			\bigl\|(-\Delta)^{\frac s2}U_\varepsilon\bigr\|_{L^2(\mathbb R^n)}^2
			=\bigl\|(-\Delta)^{\frac s2}U\bigr\|_{L^2(\mathbb R^n)}^2 .
		\end{equation}
		Indeed, using the Fourier definition
		\[
		\widehat{(-\Delta)^{\frac s2}f}(\xi)=|\xi|^{s}\,\hat f(\xi),
		\qquad
		\|(-\Delta)^{\frac s2}f\|_{2}^{2}
		=\int_{\mathbb R^n}|\xi|^{2s}\,|\hat f(\xi)|^{2}\,d\xi,
		\]
		it suffices to compute the scaling of $\widehat{U_\varepsilon}$. Since $U$ need not belong to $L^1$ or $L^2$, the Fourier transforms in this calculation are understood in the sense of tempered distributions.
		The scaling identity in $\mathcal S'(\mathbb R^n)$ gives
		\[
		\widehat{U_\varepsilon}(\xi)
		=\varepsilon^{\frac{n+2s}{2}}\widehat U(\varepsilon\xi).
		\]
		Hence
		\begin{align*}
			\bigl\|(-\Delta)^{\frac s2}U_\varepsilon\bigr\|_{2}^{2}
			&=\int_{\mathbb R^n}|\xi|^{2s}\,|\widehat{U_\varepsilon}(\xi)|^{2}\,d\xi
			=\int_{\mathbb R^n}|\xi|^{2s}\,
			\varepsilon^{n+2s}\,|\widehat U(\varepsilon\xi)|^{2}\,d\xi.
		\end{align*}
		Now set $\zeta=\varepsilon\xi$. Then
		\begin{align*}
			\bigl\|(-\Delta)^{\frac s2}U_\varepsilon\bigr\|_{2}^{2}
			&=\int_{\mathbb R^n}\Bigl|\frac{\zeta}{\varepsilon}\Bigr|^{2s}\,
			\varepsilon^{n+2s}\,|\widehat U(\zeta)|^{2}\,\varepsilon^{-n}\,d\zeta
			=\int_{\mathbb R^n}|\zeta|^{2s}\,|\widehat U(\zeta)|^{2}\,d\zeta=\bigl\|(-\Delta)^{\frac s2}U\bigr\|_{2}^{2}.
		\end{align*}
		
		\smallskip\noindent
		\textbf{Step 2: The cross term is $O(\varepsilon^{n})$.}
		We stress that for $s>1$, it is not convenient to interpret $(-\Delta)^sU$
		pointwise. Instead we use the Fourier
		definition.
		Let $\dot H^s(\mathbb R^n)$ be the completion of $C_c^\infty(\mathbb R^n)$
		under $\|(-\Delta)^{s/2}\cdot\|_2$. The bubble $U$ is an extremal
		in this space; see \cite{cotsiolis2004best}. Since $U$ is an extremal, it is a critical point of the Sobolev quotient
		\[
		\mathcal J(f):=\frac{\|(-\Delta)^{\frac{s}{2}}f\|_{2}^{2}}
		{\|f\|_{2_s^*}^{2}},
		\]
		hence there exists $\kappa_{n,s}>0$ such that the following
		Euler--Lagrange equation holds in weak form:
		for every $\varphi\in \dot H^s(\mathbb R^n)$,
		\[	\int_{\mathbb R^n} (-\Delta)^{\frac{s}{2}}U\;
		(-\Delta)^{\frac{s}{2}}\varphi\,dx
		= \kappa_{n,s}\int_{\mathbb R^n} U^{2_s^*-1}\,\varphi\,dx.\]
		By scaling invariance, the same identity holds for $U_\varepsilon$:
		\begin{equation}\label{eq:EL-weak-Ueps}
			\int_{\mathbb R^n} (-\Delta)^{\frac{s}{2}}U_\varepsilon\;
			(-\Delta)^{\frac{s}{2}}\varphi\,dx
			= \kappa_{n,s}\int_{\mathbb R^n} U_\varepsilon^{2_s^*-1}\,\varphi\,dx,
			\quad \forall\,\varphi\in \dot H^s(\mathbb R^n).
		\end{equation}
		Now let $z_\varepsilon:=w_\varepsilon-U_\varepsilon=(\eta-1)U_\varepsilon$. The dyadic estimate in Step~3 below shows that $z_\varepsilon\in\dot H^s$. Applying \eqref{eq:EL-weak-Ueps} with $\varphi=z_\varepsilon$ gives
		\begin{equation}\label{eq:cross-term-weak}
			\int_{\mathbb R^n} (-\Delta)^{\frac{s}{2}}U_\varepsilon\;
			(-\Delta)^{\frac{s}{2}}z_\varepsilon\,dx
			=\kappa_{n,s}\int_{\mathbb R^n}U_\varepsilon^{2_s^*-1}\,z_\varepsilon\,dx
			=\kappa_{n,s}\int_{\mathbb R^n}(\eta-1)\,U_\varepsilon^{2_s^*}\,dx.
		\end{equation}
		Since $\eta\equiv1$ on $B_\delta(0)$, the integrand is supported in $\{|x|\ge\delta\}$, we estimate (for $\varepsilon$ small so that $\delta/\varepsilon\ge1$)
		\begin{align*}
			\left|\int_{\mathbb R^n}(\eta-1)\,U_\varepsilon^{2_s^*}\,dx\right|
			&\le \int_{|x|\ge\delta}|U_\varepsilon(x)|^{2_s^*}\,dx
			=\int_{|y|\ge\delta/\varepsilon}(1+|y|^2)^{-n}\,dy
			=O(\varepsilon^n).
		\end{align*}
		Combining this with \eqref{eq:cross-term-weak} yields
		\begin{equation}\label{eq:cross-term-On-revised}
			\left|	\int_{\mathbb R^n} (-\Delta)^{\frac{s}{2}}U_\varepsilon\;
			(-\Delta)^{\frac{s}{2}}z_\varepsilon\,dx\right|
			\le O(\varepsilon^n).
		\end{equation}
		
		\smallskip\noindent
		\textbf{Step 3: The remainder energy is $O(\varepsilon^{n-2s})$.}
		Set $a:=\frac{n-2s}{2}>0$
		and choose an integer $m>s$. Let $\{\chi_j\}_{j\ge0}$ be a smooth dyadic
		partition of unity on $\{|x|\ge\delta\}$ such that
		\[
		\operatorname{supp}\chi_j\subset\{x:|x|\asymp R_j\},
		\qquad
		R_j:=2^j\delta,
		\]
		and $|\partial^\gamma\chi_j(x)|
		\le C_\gamma R_j^{-|\gamma|}.$
		Since $z_\varepsilon=(\eta-1)U_\varepsilon,$
		Lemma~\ref{lem:derivative-bound-Ueps} and Leibniz' rule give, for
		$|\alpha|\le m$,
		\[
		|\partial^\alpha z_\varepsilon(x)|
		\le
		C_{\alpha,\delta}\varepsilon^a|x|^{-2a-|\alpha|},
		\qquad |x|\ge\delta.
		\]
		
		Define $g_j(y):=(\chi_jz_\varepsilon)(R_jy).$
		Then $g_j$ is supported in a fixed annulus and $\|g_j\|_{H^m(\mathbb R^n)}
		\le
		C\varepsilon^aR_j^{-2a}.$
		Since $m>s$, $\|g_j\|_{\dot H^s(\mathbb R^n)}
		\le
		C\|g_j\|_{H^m(\mathbb R^n)}.$
		Scaling back,
		\[
		\|\chi_jz_\varepsilon\|_{\dot H^s}
		=
		R_j^{\frac n2-s}\|g_j\|_{\dot H^s}
		\le
		C\varepsilon^aR_j^{-2a+\frac n2-s}
		=
		C\varepsilon^aR_j^{-a}.
		\]
		Because $a>0$, $\sum_{j=0}^\infty R_j^{-a}<\infty.$
		Hence the partial sums of $\sum_j\chi_jz_\varepsilon$ converge in
		$\dot H^s(\mathbb R^n)$ and in distributions to $z_\varepsilon$, and therefore $\|z_\varepsilon\|_{\dot H^s}
		\le
		C\varepsilon^a.$
		Thus
		\begin{equation}\label{eq:zeps-energy-Ons}
			\bigl\|(-\Delta)^{s/2}z_\varepsilon\bigr\|_{L^2(\mathbb R^n)}^2
			\le
			C_{n,s,\eta,\delta}\,\varepsilon^{n-2s}.
		\end{equation}
		
		Inserting \eqref{eq:Ueps-energy-invariant-fourier},
		\eqref{eq:cross-term-On-revised}, and
		\eqref{eq:zeps-energy-Ons} into \eqref{eq:energy-expand}, we obtain
		\[
		\int_{\mathbb R^n}
		w_\varepsilon(-\Delta)^s w_\varepsilon\,dx
		=
		\bigl\|(-\Delta)^{s/2}U\bigr\|_{L^2(\mathbb R^n)}^2
		+
		O(\varepsilon^{n-2s}),
		\]
		which is \eqref{eq:energy-asymptotics-tilde-lemma}.
	\end{proof}
	
	\subsection{Proof of Theorem \ref{thm:Hardy-lowerbound-Ptilde}, \ref{prop:gap-bs-lambda011}}
	
	With the asymptotic estimates established above for the cut-off bubble, we are now in a position to prove Theorem~\ref{thm:Hardy-lowerbound-Ptilde} and Theorem~\ref{prop:gap-bs-lambda011}.

	\noindent \textbf{Proof of Theorem \ref{thm:Hardy-lowerbound-Ptilde}:}
	Recall
	\[
	U(x)=(1+|x|^2)^{-\frac{n-2s}{2}},\qquad
	U_\varepsilon(x)=\varepsilon^{-\frac{n-2s}{2}}U\Bigl(\frac{x}{\varepsilon}\Bigr),
	\]
	and fix $0<\delta<\frac14$. Choose $\eta\in C_c^\infty(\mathbb B^n)$ such that
	$0\le \eta\le 1$, $\eta\equiv 1$ on $B_\delta(0)$ and $\eta\equiv 0$ on
	$\mathbb B^n\setminus B_{2\delta}(0)$. Define
	\[
	w_\varepsilon:=\eta\,U_\varepsilon\in C_c^\infty(\mathbb R^n),
	\qquad
	u_\varepsilon:=\phi^{s-\frac n2}w_\varepsilon\in C_c^\infty(\mathbb B^n).
	\]
	By Lemma~\ref{lem:exact-reduction},
	\begin{equation}	\label{eq:reduce-energy}
		\int_{\mathbb B^n}(\widetilde \cP_s u_\varepsilon)u_\varepsilon\,dV_{{\mathbb B}^n}=\int_{\mathbb R^n} w_\varepsilon(-\Delta)^s w_\varepsilon\,dx,\quad
		\int_{\mathbb B^n}|u_\varepsilon|^{2_s^*}\,dV_{{\mathbb B}^n}=\int_{\mathbb R^n}|w_\varepsilon|^{2_s^*}\,dx.
	\end{equation}
	In particular, setting
	\[
	L_\varepsilon:=\Bigl(\int_{\mathbb B^n}|u_\varepsilon|^{2_s^*}\,dV_{{\mathbb B}^n}\Bigr)^{\!\frac{2}{2_s^*}}
	=\Bigl(\int_{\mathbb R^n}|w_\varepsilon|^{2_s^*}\,dx\Bigr)^{\!\frac{2}{2_s^*}},
	\]
	by Proposition \ref{lem:L2sstar-cutoff}, we have $L_\varepsilon = L_0 + O(\varepsilon^n)$ with $L_0:=\Bigl(\int_{\mathbb R^n}|U|^{2_s^*}\,dx\Bigr)^{\!\frac{2}{2_s^*}}>0,$
	and by \eqref{eq:energy-asymptotics-tilde-lemma} and  \eqref{eq:reduce-energy},
	\begin{equation}\label{eq:energy-over-B}
		\frac{\displaystyle \int_{\mathbb B^n}(\widetilde \cP_s u_\varepsilon)u_\varepsilon\,dV_{{\mathbb B}^n}}{L_\varepsilon}
		\;=\;
		S_{n,s} \;+\; O(\varepsilon^{\,n-2s}).
	\end{equation}
	By the condition that
	\begin{equation}\label{eq:main-ineq-after-test}
		S_{n,s}
		\;\le\;
		\frac{\displaystyle \int_{\mathbb B^n}(\widetilde \cP_s u_\varepsilon)u_\varepsilon\,dV_{{\mathbb B}^n}}{L_\varepsilon}
		\;+\;
		\lambda\,\frac{\displaystyle \int_{\mathbb B^n}|u_\varepsilon|^2\,dV_{{\mathbb B}^n}}{L_\varepsilon}.
	\end{equation}
	By \eqref{eq:energy-over-B}, the first fraction is $S_{n,s}+O(\varepsilon^{n-2s})$.
	The second fraction has the same order as
	$\int_{\mathbb B^n}|u_\varepsilon|^2\,dV_{{\mathbb B}^n}$.
	
	\smallskip
	\textbf{ (i) Case $n>4s$.}
	Then \eqref{eq:L2-asymptotics-tilde} gives
	\[
	\frac{\int_{\mathbb B^n}|u_\varepsilon|^2\,dV_{{\mathbb B}^n}}{L_\varepsilon}
	=
	\widetilde c_{n,s}\,\varepsilon^{2s}+o(\varepsilon^{2s})
	\qquad\text{with }\widetilde c_{n,s}>0,
	\]
	whereas $O(\varepsilon^{n-2s})=o(\varepsilon^{2s})$ because $n-2s>2s$.
	Thus \eqref{eq:main-ineq-after-test} becomes $S_{n,s}
	\le
	S_{n,s}
	+\lambda\,\widetilde c_{n,s}\,\varepsilon^{2s}
	+o(\varepsilon^{2s}).$
	If $\lambda<0$, the right-hand side is $<S_{n,s}$ for $\varepsilon$ small,
	a contradiction. Hence $\lambda\ge 0$.
	
	\smallskip
	\textbf{ (ii) Case $n=4s$.}
	Then
	\[
	\frac{\int_{\mathbb B^n}|u_\varepsilon|^2\,dV_{{\mathbb B}^n}}{L_\varepsilon}
	=
	\widetilde c_{n,s}\,\varepsilon^{2s}|\log\varepsilon|+O(\varepsilon^{2s}),
	\qquad \widetilde c_{n,s}>0,
	\]
	while the energy error is $O(\varepsilon^{n-2s})=O(\varepsilon^{2s})$.
	Therefore the logarithmic factor dominates, and the same contradiction argument
	shows $\lambda\ge 0$.
	
	\textbf{(iii) Case $2s<n<4s$.}
	In this regime, the cutoff error in the energy and the $L^2$-term are both of order
	$\varepsilon^{n-2s}$. By Proposition~\ref{lem:L2-cutoff},
	\[
	\int_{\mathbb B^n}|u_\varepsilon|^2\,dV_{\mathbb B^n}
	=
	c_{3,s}\varepsilon^{n-2s}
	+o(\varepsilon^{n-2s}),
	\qquad c_{3,s}>0.
	\]
	Hence \eqref{eq:main-ineq-after-test} gives
	\[
	S_{n,s}
	\le
	S_{n,s}
	+
	O(\varepsilon^{n-2s})
	+
	\lambda\frac{c_{3,s}}{L_0}\varepsilon^{n-2s}
	+
	o(\varepsilon^{n-2s}).
	\]
	Since the energy remainder is only controlled at the same order, this expansion
	does not yield a sharper bound on $\lambda$. On the other hand, the assumed inequality implies
	\[
	\bigl\langle \widetilde{\cP}_s u,u\bigr\rangle
	+\lambda\|u\|_{L^2(\mathbb B^n)}^2\ge0
	\]
	for every $u\in C_c^\infty(\mathbb B^n)$. Therefore, $\lambda\ge -\widetilde{\lambda}_{0,s}^{\mathrm{conf}}.$
	Consequently, $\widetilde\Lambda_{n,s}^{\mathrm{HS}}
	\ge
	-\widetilde{\lambda}_{0,s}^{\mathrm{conf}}.$
	\hfill$\Box$\medskip

	\noindent \textbf{Proof of Theorem \ref{prop:gap-bs-lambda011}:}
	
	(i) Applying \cite[Lemma~5.2]{lu2023explicit}  with $\gamma=s$ and $\lambda=0$, we obtain
	\begin{equation}\label{eq:lemma52-at-zero}
		2^{2s}\,\frac{\Gamma\bigl(\frac{3+2s}{4}\bigr)^2}{\Gamma\bigl(\frac{3-2s}{4}\bigr)^2}
		=
		\frac{\Gamma\bigl(s+\frac12\bigr)^2}{\Gamma\bigl(\frac12\bigr)^2}
		+\frac{\sin(\pi s)}{\pi}\,\Gamma\bigl(s+\tfrac12\bigr)^2.
	\end{equation}
	Since $\Gamma(\frac12)=\sqrt{\pi}$, identity \eqref{eq:lemma52-at-zero} becomes
	\begin{equation}\label{eq:lambda0-in-Cs}
		\lambda_{0,s}^{\mathrm{conf}}
		=\frac{1+\sin(\pi s)}{\pi}\,\Gamma\bigl(s+\tfrac12\bigr)^2\ge0.
	\end{equation}
	
	We now compare $\lambda_{0,s}^{\mathrm{conf}}$ with $b_s$.
	
	\smallskip
	\noindent\textbf{Case 1: $\sin(\pi s)\le 0$.}
	Then $b_s=0$, hence $b_s\le\lambda_{0,s}^{\mathrm{conf}}$.
	Moreover, \eqref{eq:lambda0-in-Cs} yields the second line of \eqref{eq:gap-explicit}.
	
	\smallskip
	\noindent\textbf{Case 2: $\sin(\pi s)>0$.}
	Then $b_s=\frac{\sin(\pi s)}{\pi}\Gamma\bigl(s+\tfrac12\bigr)^2$. Subtracting from \eqref{eq:lambda0-in-Cs} gives
	\[
	\lambda_{0,s}^{\mathrm{conf}}-b_s
	=\Bigl(\frac{1+\sin(\pi s)}{\pi}-\frac{\sin(\pi s)}{\pi}\Bigr)\Gamma\bigl(s+\tfrac12\bigr)^2
	=\frac{\Gamma\bigl(s+\tfrac12\bigr)^2}{\pi}>0,
	\]
	which proves both $b_s<\lambda_{0,s}^{\mathrm{conf}}$ and \eqref{eq:gap-explicit}.
	
	(ii) We define the spectral multiplier
	\[	M_s(\beta):=\Bigl|\Gamma\Bigl(s+\tfrac12+\bi\beta\Bigr)\Bigr|^2,\quad \beta\ge 0.\]
	Using the classical asymptotic formula $\lim_{|\lambda|\to\infty}\, \bigl|\Gamma(a+\bi\lambda)\bigr|\,
	e^{\frac{\pi}{2}|\lambda|}\,|\lambda|^{\frac12-a}
	=\sqrt{2\pi},$
	we know that $M_s(\beta)\to0$ as $\beta\to\infty$. By the fact that $|\Gamma(a+i\lambda)|\le |\Gamma(a)|$,
	\begin{equation}\label{rema}
		\sup_{\beta\ge 0} M_s(\beta)=M_s(0)\in(0,\infty).
	\end{equation}
	Thus, by the spectral calculus, $\|\,M_s(\cA)\,\|_{L^2\to L^2}\le M_s(0).$
	By identity \eqref{eq:P-vs-tildeP},
	\begin{equation}\label{eq:P-vs-Ptilde-116}
		\cP_s
		=\widetilde \cP_s
		+\frac{\sin(s\pi)}{\pi}\,M_s(\cA),\  s\in \left(0,\frac{n}{2}\right)\setminus \mathbb{N},
	\end{equation}
	where $M_s(\cA)$ is defined by functional calculus and is a bounded, self-adjoint,
	nonnegative operator.
	
	\smallskip
	\noindent\textbf{Case 1: $\sin(s\pi)\le 0$.}
	Then the last term in \eqref{eq:P-vs-Ptilde-116} is nonpositive, hence
	\[
	\langle \widetilde \cP_s u,u\rangle_{L^2}
	\;\ge\;
	\langle \cP_s u,u\rangle_{L^2}
	\qquad\text{for all }u\in C_c^\infty(\mathbb B^n).
	\]
	Now apply Theorem~\ref{thm:Hardy-lowerbound-Ptilde} to conclude the corresponding lower
	bound for $\lambda$.
	
	\smallskip
	\noindent\textbf{Case 2: $\sin(s\pi )> 0$.}
	Since $M_s(\cA)\le M_s(0)\,\mathrm{Id}$ on $L^2$, we have from
	\eqref{eq:P-vs-Ptilde-116}:
	\[
	\langle \cP_s u,u\rangle_{L^2}
	\le
	\langle \widetilde \cP_s u,u\rangle_{L^2}
	+\frac{\sin(s\pi)}{\pi}\,M_s(0)\,\|u\|_{L^2}^2
	=
	\langle \widetilde \cP_s u,u\rangle_{L^2}
	+b_s\,\|u\|_{L^2}^2.
	\]
	Equivalently, $\langle \widetilde \cP_s u,u\rangle_{L^2}
	\ge
	\langle \cP_s u,u\rangle_{L^2}
	-b_s\,\|u\|_{L^2}^2.$
	Thus Theorem~\ref{thm:Hardy-lowerbound-Ptilde} applies with the shifted coefficient
	$\lambda':=\lambda+b_s$. Reading off the conclusion gives exactly (b), i.e.
	\[
	\lambda \ge -b_s \quad\text{when } n\ge 4s,
	\qquad
	\lambda \ge -b_s-\widetilde \lambda_s^{\mathrm{conf}} \quad\text{when } 2s<n<4s.
	\]
	Combining this with Lemma~\ref{prop:bottom-spectrum-quadratic-form}, we obtain the desired conclusion.\hfill$\Box$

	\section{Fractional Brezis--Nirenberg Problems}
	
	In this section, we study the attainability of the fractional Poincar\'e--Sobolev levels associated with $\cP_s$ and $\widetilde{\cP}_s$, and, as an application, we establish existence results for the corresponding fractional Brezis--Nirenberg problems.
	
	We first define the energy spaces and establish their local Sobolev
	regularity, including at the spectral bottom. We also record the
	Sobolev mapping properties of the operators. For attainment, a bounded
	spectral cutoff detects a nonzero weak limit after hyperbolic
	isometries, and the Brezis--Lieb lemma yields strong convergence.
	
	We then prove that the levels equal $-\infty$ above the respective
	spectral bottoms. The proof uses the exponential decay of the
	off-diagonal kernels from
	Proposition~\ref{lem:offdiag-exp-decay-Ptilde}, obtained from conformal
	intertwining and an explicit spherical transform.

	\subsection{Energy Spaces}
	
	Now we can study fractional Brezis--Nirenberg type equations (\ref{eq:BN-Hn-frac}) and (\ref{eq:BN-Hn-frac2}) on the whole hyperbolic space
	$\mathbb H^n$ driven by $\cP_s$ and $\widetilde\cP_s$, namely
	\[\cP_s u=\lambda u+|u|^{p-1}u
	\quad\text{in }\mathbb H^n,\quad 	\widetilde\cP_s u=\lambda u+|u|^{p-1}u
	\quad\text{in }\mathbb H^n\]
	for  
	$1<p\le 2_s^*-1.$
	From the bottom spectrum
	\[
	\lambda_{0,s}^{\mathrm{conf}}
	:=\inf\sigma(\cP_s)=m_s(0)\ge0,\quad \widetilde \lambda_{0,s}^{\mathrm{conf}}
	:=\inf\sigma(\widetilde\cP_s)=\widetilde m_s(0)>0
	\]
	and by \cite[Theorem E.8]{KellerLenzWojciechowski2021}, we obtain  the sharp fractional Poincar\'e
	inequality
	\begin{equation}\label{eq:conf-Poincare}
		\lambda_{0,s}^{\mathrm{conf}}\,\|u\|_{L^2(\mathbb H^n)}^2
		\le \langle u,\cP_s u\rangle_{L^2(\mathbb H^n)},\ \ \widetilde \lambda_{0,s}^{\mathrm{conf}}\|u\|_{L^2(\mathbb H^n)}^2
		\le \langle u,\widetilde\cP_s u\rangle_{L^2(\mathbb H^n)}
		\qquad\text{for all }u\in C_c^\infty(\mathbb H^n),
	\end{equation}
	and hence quadratic forms
	\[
	\mathcal E_{\lambda,s}(u)
	:=\langle (\cP_s-\lambda)u,u\rangle_{L^2(\mathbb H^n)}
	=\langle \cP_s u,u\rangle_{L^2(\mathbb H^n)}-\lambda\|u\|_2^2
	\]
	and
	\[
	{\widetilde{\mathcal E}}_{\lambda,s}(u)
	:=\langle (\widetilde\cP_s-\lambda)u,u\rangle_{L^2(\mathbb H^n)}
	=\langle \widetilde\cP_s u,u\rangle_{L^2(\mathbb H^n)}-\lambda\|u\|_2^2
	\]
	are nonnegative whenever $\lambda\le \lambda_{0,s}^{\mathrm{conf}}$ and
	$\lambda\le \widetilde\lambda_{0,s}^{\mathrm{conf}}$, respectively.
	
	For $u,v\in C_c^\infty(\mathbb H^n)$, we set
	\[
	\langle u,v\rangle_{\lambda,s}
	:=\langle (\cP_s-\lambda)u,v\rangle_{L^2(\mathbb H^n)}
	=\int_{\mathbb H^n}v\ (\cP_s-\lambda)u\,dV_{\mathbb H^n},
	\qquad
	\|u\|_{\lambda,s}^2:=\langle u,u\rangle_{\lambda,s}=\mathcal E_{\lambda,s}(u),
	\]
	and
	\[
	\langle u,v\rangle_{\lambda,s,\sim}
	:=\langle (\widetilde\cP_s-\lambda)u,v\rangle_{L^2(\mathbb H^n)}
	=\int_{\mathbb H^n} v\ (\widetilde\cP_s-\lambda)u\,dV_{\mathbb H^n},
	\qquad
	\|u\|_{\lambda,s,\sim}^2:=\langle u,u\rangle_{\lambda,s,\sim}
	={\widetilde{\mathcal E}}_{\lambda,s}(u).
	\]
	
	By \cite[Theorem 1.3 and Theorem 1.8]{lu2023explicit}, for every $\lambda\le \lambda_{0,s}^{\mathrm{conf}}$, the map
	$u\mapsto\|u\|_{\lambda,s}$ defines a norm on $C_c^\infty(\mathbb H^n)$, and  for every $\lambda\le \widetilde \lambda_{0,s}^{\mathrm{conf}}$, the map
	$u\mapsto\|u\|_{\lambda,s,\sim}$ also defines a norm on $C_c^\infty(\mathbb H^n)$
	
	The next two lemmas show that below the respective spectral bottoms these norms are equivalent to the standard spectral fractional Sobolev norm.

	\begin{lemma}\label{lem:domain-tildePs}
		Let $s\in (0,\frac{n}{2})$ and $\widetilde m_s(\beta)$ be defined in (\ref{mstl}), then there exist constants $c_1,c_2>0$ (depending only on $s$) such that
		for all $\beta\in\mathbb R$,
		\[	c_1\,(\beta^2+1)^{s}\ \le\ \widetilde m_s(\beta)\ \le\ c_2\,(\beta^2+1)^{s}.\]
	\end{lemma}
	
	\begin{proof}
		First note that $\widetilde m_s$ is even and strictly positive on $\mathbb R$, since
		$\Gamma$ has no zeros and $\Gamma(\tfrac12+\bi\beta)\neq 0$ for all $\beta\in\mathbb R$. We use the uniform Stirling estimate in vertical strips: for each fixed $\sigma>0$,
		there exist $T\ge 1$ and constants $C_\pm=C_\pm(\sigma)>0$ such that for all
		$\sigma'\in[\sigma,\sigma+1]$ and all $|t|\ge T$,
		\begin{equation}\label{eq:stirling-strip-tilde}
			C_-\,|t|^{\sigma'-\frac12}e^{-\frac{\pi}{2}|t|}
			\ \le\ |\Gamma(\sigma'+\bi t)|\
			\le\ C_+\,|t|^{\sigma'-\frac12}e^{-\frac{\pi}{2}|t|}.
		\end{equation}
		Apply \eqref{eq:stirling-strip-tilde} with $\sigma'=s+\frac12$ and $\sigma'=\frac12$.
		For $|\beta|\ge T$ this gives $	\frac{\bigl|\Gamma(s+\frac12+\bi\beta)\bigr|}{\bigl|\Gamma(\frac12+\bi\beta)\bigr|}
		\sim |\beta|^{s},$
		hence
		\[
		\widetilde m_s(\beta)
		=\Bigl(\frac{\bigl|\Gamma(s+\frac12+\bi\beta)\bigr|}{\bigl|\Gamma(\frac12+\bi\beta)\bigr|}\Bigr)^{\!2}
		\sim|\beta|^{2s},
		\qquad |\beta|\ge T.
		\]
		Since $(1+\beta^2)^s\sim |\beta|^{2s}$ for $|\beta|\gg 1$, we obtain constants
		$C_1,C_2>0$ such that
		\begin{equation}\label{eq:mtilde-1beta2-large}
			C_1(1+\beta^2)^s\le \widetilde m_s(\beta)\le C_2(1+\beta^2)^s,
			\qquad |\beta|\ge T.
		\end{equation}
		
		On the compact set $\{|\beta|\le T\}$, the function $\widetilde m_s$ is continuous and strictly
		positive, and $(1+\beta^2)^s$ is also bounded above and below by positive constants on $\{|\beta|\le T\}$,
		combining with \eqref{eq:mtilde-1beta2-large} yields global constants (still denoted $C_1,C_2$) such that
		\[	C_1(1+\beta^2)^s\le \widetilde m_s(\beta)\le C_2(1+\beta^2)^s,
		\qquad \forall\,\beta\in\mathbb R,\]
		which is the desired estimate.
	\end{proof}\medskip
	
	By Lemma~\ref{lem:domain-tildePs} we have $\mathcal D(\widetilde{\cP}_s)=H^{2s}(\mathbb H^n),$
	where $H^{2s}(\mathbb H^n)$ is given in (\ref{fracaa}). Consequently, for every $\lambda<\widetilde\lambda_{0,s}^{\mathrm{conf}}$, the quadratic form
	\[
	\|u\|_{\lambda,s,\sim}^2:=\bigl\langle (\widetilde{\cP}_s-\lambda)u,u\bigr\rangle_{L^2(\mathbb H^n)}
	\]
	defines a norm on $H^{s}(\mathbb H^n)$ which is equivalent to the standard $H^{s}$-norm.
	
	Similarly, using the high--frequency asymptotics of $m_s$ (via Stirling's formula)
	together with the fact that $m_s$ is bounded on compact $\beta$-intervals, one
	obtains the global two--sided comparison
	\begin{equation}\label{eq:1+ms-comparable}
		1+m_s(\beta)^2\ \sim\ 1+(\beta^2+1)^{2s}
		\quad{\rm for}\ \,  \beta\in\mathbb R,
	\end{equation}
	with implicit constants depending only on $s$ and $\rho$. In particular, $	\mathcal D(\cP_s)=H^{2s}(\mathbb H^n).$
	We emphasize that, in general, one cannot upgrade \eqref{eq:1+ms-comparable} to
	$m_s(\beta)\sim(\beta^2+\rho^2)^s$ for all $\beta\in\mathbb R$, for instance, this
	fails in the exceptional cases $s=\frac32+2k$, where $m_s(\beta)\sim\beta^2$ as
	$\beta\to 0$.

	\begin{lemma}\label{lem:equiv-norm-Ps}
		Let $s\in (0,\frac{n}{2})\setminus \mathbb{N},$ then for every $\lambda<\lambda_{0,s}^{\mathrm{conf}}$, the quantity
		\[
		\|u\|_{\lambda,s}=\sqrt{\langle (\cP_s-\lambda)u,u\rangle_{L^2(\mathbb H^n)}},\quad \forall u\in C_c^{\infty}(\mathbb{H}^n)
		\]
		extends to a norm on $H^s(\mathbb H^n)$ and it is equivalent to the standard $H^s$-norm.
	\end{lemma}
	
	\begin{proof}
		By Plancherel formula, one has 
		\begin{equation}\label{eq:form-spectral}
			\|u\|_{\lambda,s}^2
			=\langle (\cP_s-\lambda)u,u\rangle_2
			=\int_{\mathbb R}\!\!\int_{\mathbb S^{n-1}}
			\bigl(m_s(\beta)-\lambda\bigr)\,|\hat u(\beta,\theta)|^2\,
			\frac{d\sigma(\theta)\,d\beta}{|c(\beta)|^2}.
		\end{equation}
		
		Set $\lambda_0:=\lambda_{0,s}^{\mathrm{conf}}$ and $\delta:=\lambda_0-\lambda>0$.
		Since $m_s(\beta)\ge \lambda_0\ge0$ for all $\beta$, we have $m_s(\beta)-\lambda\ge \delta$.
		Moreover,
		\[	m_s(\beta)-\lambda \le m_s(\beta)+|\lambda|
		\le (1+|\lambda|)\,(1+m_s(\beta))\qquad\forall\,\beta\in \mathbb{R}.\]	
		To get the bound $1+m_s(\beta)$ from above, set
		\[
		k:=\min\left\{\frac{\lambda_0-\lambda}{1+\lambda_0},\frac{1}{2}\right\}>0.
		\]
		Since $m_s(\beta)\ge \lambda_0$ for all $\beta$, we have
		\[
		m_s(\beta)-\lambda - k\bigl(1+m_s(\beta)\bigr)\ge 0\ \, \Leftrightarrow\ \,  \lambda_0\ge \frac{\lambda+k}{1-k}\ \, \Leftrightarrow\ \, k\le \frac{\lambda_0-\lambda}{1+\lambda_0},
		\]
		and hence
		\[
		m_s(\beta)-\lambda \ge k\bigl(1+m_s(\beta)\bigr),\qquad \forall\,\beta\in\mathbb R.
		\]
		Therefore there exists $C_\lambda>0$ such that
		\[
		C_\lambda^{-1}\,(1+m_s(\beta))\le m_s(\beta)-\lambda \le C_\lambda\,(1+m_s(\beta)),
		\qquad \forall\,\beta\in\mathbb R.
		\]
		Plugging the pointwise comparison into \eqref{eq:form-spectral} and using the
		definition \eqref{fracaa} of $H^s(\mathbb H^n)$ gives
		\[
		\|u\|_{\lambda,s}^2\ \sim\ \|u\|_{H^s(\mathbb H^n)}^2,\qquad u\in C_c^\infty(\mathbb H^n).
		\]
		In particular, $\|\cdot\|_{\lambda,s}$ is positive definite on $C_c^\infty$.
		Since $C_c^\infty(\mathbb H^n)$ is dense in $H^s(\mathbb H^n)$, the norm extends uniquely to
		$H^s(\mathbb H^n)$ and remains equivalent to the standard $H^s$-norm.
	\end{proof}\medskip

	Hence, throughout the paper, we will simply write $\|\cdot\|_{\lambda,s}$
	to denote either of the equivalent norms associated with $\cP_s$ or $\widetilde{\cP}_s$,
	whenever the distinction is immaterial.

	\begin{definition}\label{def:Hs-lambda-conf}
		For $0<s<n/2$ and $\lambda$ not exceeding the spectral bottom of the respective operator, we define the energy spaces
		$\mathcal H^{s}_{\lambda}(\mathbb H^n)$ and $\mathcal H^{s}_{\lambda,\sim}(\mathbb H^n)$
		as the completions of $C_c^\infty(\mathbb H^n)$ with respect to the norms
		$\|\cdot\|_{\lambda,s}$ and $\|\cdot\|_{\lambda,s,\sim}$, respectively, i.e.
		\[
		\mathcal H^{s}_{\lambda}(\mathbb H^n)
		:=\overline{C_c^\infty(\mathbb H^n)}^{\,\|\cdot\|_{\lambda,s}},
		\qquad
		\mathcal H^{s}_{\lambda,\sim}(\mathbb H^n)
		:=\overline{C_c^\infty(\mathbb H^n)}^{\,\|\cdot\|_{\lambda,s,\sim}}.
		\]
	\end{definition}
	
	\begin{remark}\label{rem:Hs-lambda-equals-Hs}
		Let $ \lambda_{0,s}^{\mathrm{conf}}$ and $\widetilde \lambda_{0,s}^{\mathrm{conf}}$ denote the bottom spectral values associated with
		$\cP_s$ and $\widetilde\cP_s$, respectively. By Lemma \ref{lem:equiv-norm-Ps}, if $\lambda<\lambda_{0,s}^{\mathrm{conf}}$, then $\mathcal H^{s}_{\lambda}(\mathbb H^n)=H^s(\mathbb H^n).$
		By Lemma \ref{lem:domain-tildePs}, if $\lambda<\widetilde\lambda_{0,s}^{\mathrm{conf}}$, then $\mathcal H^{s}_{\lambda,\sim}(\mathbb H^n)=H^s(\mathbb H^n).$
		In both cases, the corresponding energy norm is equivalent to the standard $H^s(\mathbb H^n)$ norm.
	\end{remark}
	
	\begin{lemma}\label{lem:critical-compact-support-and-weak-Hs}
		Let
		\[
		\lambda=\lambda_{0,s}^{\mathrm{conf}}
		\quad\text{for }\mathcal H^s_\lambda(\mathbb H^n),
		\qquad
		\lambda=\widetilde\lambda_{0,s}^{\mathrm{conf}}
		\quad\text{for }\mathcal H^s_{\lambda,\sim}(\mathbb H^n).
		\]
		Then the following hold.
		
		\smallskip
		\noindent\textnormal{(i)}
		If \(u\) belongs to the corresponding space and
		\(\operatorname{supp}u\Subset\mathbb H^n\), then
		\(u\in H^s(\mathbb H^n)\).
		
		\smallskip
		\noindent\textnormal{(ii)}
		Let \(\{v_j\}\) be a sequence in the corresponding space such that
		\(v_j\rightharpoonup0\) weakly there. If, for some compact set
		\(K\Subset\mathbb H^n\), $\operatorname{supp}v_j\subset K$ for all $j$,
		then
		\[
		v_j\rightharpoonup0
		\qquad\text{weakly in }H^s(\mathbb H^n).
		\]
	\end{lemma}
	
	\begin{proof}
		We treat the two spaces simultaneously. Write $T=\cP_s$ or $T=\widetilde\cP_s,$
		let \(m(\beta)\) be the corresponding spectral multiplier, and set $q(\beta):=m(\beta)-m(0).$
		By the multiplier estimates established above,
		\begin{equation}\label{eq:endpoint-q-asymptotics}
			q(\beta)\asymp\beta^2\quad |\beta|\le1,
			\qquad
			q(\beta)\asymp|\beta|^{2s}\quad |\beta|\ge1.
		\end{equation}
		Indeed, when $m(0)>0$, the Gamma-function product formula above gives a
		strictly positive coefficient of $\beta^2$ in $m(\beta)-m(0)$.
		For $T=\cP_s$ at $s=\frac32+2j$, $j\in\mathbb N_0$, the reciprocal
		Gamma factor has a simple zero, so $m(\beta)\asymp\beta^2$ as well.
		The high-frequency estimate follows from Stirling's formula.
		Moreover, $|c(\beta)|^{-2}\asymp\beta^2$ as $\beta\to0.$ Choose an even function \(\chi\in C_c^\infty(\mathbb R)\) such that
		\(\chi\equiv1\) near \(0\). For \(u\in C_c^\infty(\mathbb H^n)\), define
		\[
		\widehat\ell(\beta,\theta)
		=
		\chi(\beta)\widehat u(\beta,\theta),
		\qquad
		\widehat h(\beta,\theta)
		=
		(1-\chi(\beta))\widehat u(\beta,\theta),
		\]
		so that \(u=\ell+h\). Since \(\chi\) is even, these multipliers preserve
		the Helgason--Fourier range.
		
		By Plancherel and \eqref{eq:endpoint-q-asymptotics},
		\begin{equation}\label{eq:endpoint-high-frequency}
			\|h\|_{H^s(\mathbb H^n)}
			\le C\|u\|_{\lambda,s},
		\end{equation}
		where, in the tilde case, \(\|\cdot\|_{\lambda,s}\) denotes the
		corresponding tilde energy norm.
		
		For the low-frequency part, let \(K\Subset\mathbb H^n\). The derivatives
		of the Helgason waves are uniformly bounded on
		\(K\times\operatorname{supp}\chi\times\mathbb S^{n-1}\). Hence the inversion formula
		and Cauchy--Schwarz yield, for every integer \(k\ge0\),
		\[
		\|\ell\|_{C^k(K)}
		\le
		C_{K,k}\|u\|_{\lambda,s}
		\left(
		\int_{\operatorname{supp}\chi}
		\frac{|c(\beta)|^{-2}}{q(\beta)}\,d\beta
		\right)^{1/2}.
		\]
		The last integral is finite because $\frac{|c(\beta)|^{-2}}{q(\beta)}
		\asymp1$ as $\beta\to0.$
		Thus
		\begin{equation}\label{eq:endpoint-low-frequency}
			\|\ell\|_{C^k(K)}
			\le C_{K,k}\|u\|_{\lambda,s}.
		\end{equation}
		
		Let now \(\eta\in C_c^\infty(\mathbb H^n)\), and choose an integer
		\(k>s\) and a compact set \(K\) containing \(\operatorname{supp}\eta\). By
		\eqref{eq:endpoint-high-frequency}, the boundedness of multiplication by
		\(\eta\) on \(H^s\), and \eqref{eq:endpoint-low-frequency},
		\[
		\begin{aligned}
			\|\eta u\|_{H^s}
			&\le
			\|\eta h\|_{H^s}+\|\eta\ell\|_{H^s}\le
			C_\eta\|h\|_{H^s}
			+
			C_{\eta,k}\|\ell\|_{C^k(K)}
			\le
			C_\eta\|u\|_{\lambda,s}.
		\end{aligned}
		\]
		Therefore multiplication by \(\eta\) extends continuously from the
		corresponding endpoint energy space into \(H^s(\mathbb H^n)\).
		By the \(L^q\)-embeddings of
		\cite[Theorems~1.3 and~1.8]{lu2023explicit}, convergence in the energy
		space implies convergence in distributions, so this extension agrees
		with the usual product \(\eta u\).
		
		For \textnormal{(i)}, choose \(\eta\in C_c^\infty(\mathbb H^n)\) such
		that \(\eta\equiv1\) on a neighborhood of \(\operatorname{supp}u\).
		Then $u=\eta u\in H^s(\mathbb H^n).$
		
		For \textnormal{(ii)}, choose \(\eta\in C_c^\infty(\mathbb H^n)\) with
		\(\eta\equiv1\) on a neighborhood of \(K\). Since $u\mapsto\eta u$
		is a bounded linear map from the endpoint energy space into
		\(H^s(\mathbb H^n)\), the weak convergence
		\(v_j\rightharpoonup0\) in the energy space implies
		\[
		\eta v_j\rightharpoonup0
		\qquad\text{weakly in }H^s(\mathbb H^n).
		\]
		Since \(\operatorname{supp} v_j\subset K\), we have \(\eta v_j=v_j\), and hence $v_j\rightharpoonup0$ weakly in $H^s(\mathbb H^n).$
	\end{proof}
	
	\smallskip
	
	Next, we begin to study the attainability of fractional Sobolev levels associated with $\cP_s-\lambda$ and $\widetilde \cP_s-\lambda$ 
	\begin{equation}\label{hnsp}
		H_{n,s,p}(\lambda)
		:=\inf_{u\in \mathcal H^{s}_{\lambda}(\mathbb H^n)\setminus\{0\}}
		\frac{\mathcal E_{\lambda,s}(u)}{\|u\|_{L^{p+1}(\mathbb H^n)}^2}
		\in(0,\infty),\quad \lambda\le\lambda_{0,s}^{\mathrm{conf}},
	\end{equation}
	and
	\[	\widetilde H_{n,s,p}(\lambda)
	:=\inf_{u\in \mathcal H^{s}_{\lambda,\sim}(\mathbb H^n)\setminus\{0\}}
	\frac{{\widetilde{\mathcal E}}_{\lambda,s}(u)}{\|u\|_{L^{p+1}(\mathbb H^n)}^2}
	\in(0,\infty),\quad \lambda\le\widetilde \lambda_{0,s}^{\mathrm{conf}}.\]
	In the critical case \(p+1=2_s^*:=\frac{2n}{n-2s}\), we note that
	\[
	H_{n,s,p}(\lambda)=H_{n,s}(\lambda)
	\qquad\text{and}\qquad
	\widetilde H_{n,s,p}(\lambda)=\widetilde H_{n,s}(\lambda),
	\]
	where $H_{n,s}(\lambda)$ and $\widetilde H_{n,s}(\lambda)$ are defined in \eqref{eq:Sns-Hn-lambda} and \eqref{eq:Sns-Hn-lambda111}. For convenience, we set
	\[	I_{\lambda,s}(u):=\frac{\mathcal E_{\lambda,s}(u)}{\|u\|_{L^{p+1}(\mathbb H^n)}^2},\quad 	\widetilde I_{\lambda,s}(u):=\frac{{\widetilde{\mathcal E}}_{\lambda,s}(u)}{\|u\|_{L^{p+1}(\mathbb H^n)}^2}
	\qquad u\in C_c^{\infty}(\mathbb H^n)\setminus\{0\}.\]
	
	These quotients extend continuously to the corresponding completed energy spaces away from zero.
	Throughout, we write $\|u\|_{q}:=\|u\|_{L^{q}(\mathbb H^n)}$ for brevity. It is standard to work on
	the Nehari manifold
	\[	\mathcal N_{\lambda,s}
	:=\Bigl\{u\in \mathcal H^s_\lambda(\mathbb H^n)\setminus\{0\}:
	\mathcal E_{\lambda,s}(u)=\|u\|_{p+1}^{p+1}\Bigr\}\]
	and
	\[	\widetilde{\mathcal N}_{\lambda,s}
	:=\Bigl\{u\in \mathcal H^s_{\lambda,\sim}(\mathbb H^n)\setminus\{0\}:
	\widetilde{	\mathcal E}_{\lambda,s}(u)=\|u\|_{p+1}^{p+1}\Bigr\}.\]
	Note that for every $u\neq0$, there exists a unique $t(u)>0$ such that $t(u)u\in\mathcal N_{\lambda,s}$,
	namely $$t(u)^{p-1}=\mathcal E_{\lambda,s}(u)/\|u\|_{p+1}^{p+1}.$$
	Moreover, for $u\in\mathcal N_{\lambda,s}$ we have
	\begin{equation}\label{eq:I-on-Nehari}
		I_{\lambda,s}(u)=\|u\|_{p+1}^{p-1}
		=\mathcal E_{\lambda,s}(u)^{\frac{p-1}{p+1}}.
	\end{equation}
	Hence
	\[	H_{n,s,p}(\lambda)=\inf_{u\in\mathcal N_{\lambda,s}} I_{\lambda,s}(u)
	=\inf_{u\in\mathcal N_{\lambda,s}} \|u\|_{p+1}^{p-1}.\]
	
	The next lemma records the isometry invariance of $\cP_s$ and $\widetilde{\cP}_s$, and in particular the invariance of their associated quadratic forms.

	\begin{lemma}\label{lem:isom-inv-Ps}
		Let $\mathcal T$ be an isometry on $\mathbb{H}^n$ and define $(U_{\mathcal T}u)(x):=u(\mathcal T x).$
		Then, for all
		$u\in C_c^\infty(\mathbb H^n)$,
		\[
		\cP_s(U_{\mathcal T}u)=U_{\mathcal T}(\cP_su),\quad \widetilde\cP_s(U_{\mathcal T}u)=U_{\mathcal T}(\widetilde\cP_su).
		\]
		Consequently,
		\[
		\langle U_{\mathcal T}u,\cP_s(U_{\mathcal T}u)\rangle_{2}
		=\langle u,\cP_su\rangle_{2},\quad \langle U_{\mathcal T}u,\widetilde\cP_s(U_{\mathcal T}u)\rangle_{2}
		=\langle u,\widetilde\cP_su\rangle_{2}.
		\]
	\end{lemma}
	
	\begin{proof}
		Since $\mathcal T$ is an isometry and $dV_{\mathbb H^n}$ is invariant under $\mathcal T$, $U_{\mathcal T}$
		is unitary on $L^2(\mathbb H^n)$ and preserves $L^p$ norms by change of variables.
		Moreover, the Laplace--Beltrami operator is invariant under isometries, i.e.
		$\Delta_{\mathbb H^n}(u\circ\mathcal T)=(\Delta_{\mathbb H^n}u)\circ\mathcal T$.
		Therefore $U_{\mathcal T}^{-1}(-\Delta_{\mathbb H^n})U_{\mathcal T}=-\Delta_{\mathbb H^n}$.
		Since on $\mathbb H^n$,  $\cP_s$ and  $\widetilde \cP_s$ are given by
		spectral calculus as a Borel function of $	\cA^2:=-\Delta_{\mathbb H^n}-\rho^2,$
		the functional calculus implies
		$U_{\mathcal T}^{-1}\cP_sU_{\mathcal T}=\cP_s$ and $U_{\mathcal T}^{-1}\widetilde\cP_sU_{\mathcal T}=\widetilde\cP_s$,
		which is equivalent to
		\[
		\cP_s(u\circ\mathcal T)=(\cP_su)\circ\mathcal T,\quad \widetilde\cP_s(u\circ\mathcal T)=(\widetilde\cP_su)\circ\mathcal T.
		\]
		Finally, the energy identity follows from unitarity:
		\[
		\langle U_{\mathcal T}u,\cP_s(U_{\mathcal T}u)\rangle_{2}
		=\langle U_{\mathcal T}u,U_{\mathcal T}(\cP_su)\rangle_{2}
		=\langle u,\cP_su\rangle_{2}
		\]
		and
		\[
		\langle U_{\mathcal T}u,\widetilde\cP_s(U_{\mathcal T}u)\rangle_{2}
		=\langle U_{\mathcal T}u,U_{\mathcal T}(\widetilde\cP_su)\rangle_{2}
		=\langle u,\widetilde\cP_su\rangle_{2},
		\]
		which completes the proof.
	\end{proof}
	
	\subsection{Spectral Multipliers and Sobolev Mapping Properties}\label{pseduo}
	
	We establish holomorphy and derivative estimates for the spectral
	multipliers, and derive the Sobolev mapping properties used below to
	define their distribution kernels.
	
	We start from recalling the fractional Bessel potential spaces, see \cite{TaylorPDO}. Fix $s\in\mathbb R$ and $\mu>4\rho^2$ sufficiently large. Define the fractional Bessel potential space by
	\[
	\tilde{H}^{s}(\mathbb H^n):=\overline{C_c^\infty(\mathbb H^n)}^{\,\|(\mu I-\Delta_{\mathbb H^n})^{s/2}\cdot\|_2},
	\qquad 
	\|u\|_{\tilde{H}^{s}(\mathbb H^n)}:=\bigl\|(\mu I-\Delta_{\mathbb H^n})^{s/2}u\bigr\|_{L^2(\mathbb H^n)}.
	\]
	It is standard (and easy to check by the spectral theorem) that for any $\mu_1,\mu_2>4\rho^2$
	the norms $\|(\mu_1 I-\Delta_{\mathbb H^n})^{s/2}u\|_{L^2(\mathbb H^n)}$ and
	$\|(\mu_2 I-\Delta_{\mathbb H^n})^{s/2}u\|_{L^2(\mathbb H^n)}$ are equivalent; hence the space
	$\tilde H^{s}(\mathbb H^n)$ is independent of the particular choice of $\mu$, and we fix $\mu$
	once for all.
	Endow $\tilde{H}^{s}(\mathbb H^n)$ with the inner product
	\[	\langle u,v\rangle_{\tilde{H}^{s}}
	:=\Big\langle (\mu I-\Delta_{\mathbb H^n})^{s/2}u,\,
	(\mu I-\Delta_{\mathbb H^n})^{s/2}v
	\Big\rangle_{L^2(\mathbb H^n)}.\]
	Equivalently, by Plancherel and the spectral resolution of $-\Delta_{\mathbb H^n}$,
	\[
	\langle u,v\rangle_{\tilde{H}^{s}}
	=\int_{\mathbb R}\!\int_{\mathbb S^{n-1}}
	(\mu+\beta^2+\rho^2)^{s}\,
	\hat u(\beta,\theta)\,\overline{\hat v(\beta,\theta)}\,
	\frac{d\sigma(\theta)\,d\beta}{|c(\beta)|^{2}}.
	\]
	Thus, by the definition of (\ref{fracaa}), 
	\[\tilde{H}^{s}(\mathbb H^n)={H}^{s}(\mathbb H^n),\quad s\ge 0.\]
	
	We recall that $\cP_s$ with $s\in (0,\frac{n}{2})\setminus {\mathbb N}$ on $\mathbb H^n$
	defined by spectral calculus as
	\begin{equation}\label{eq:Pgamma_multiplier_def}
		\cP_s = m_s\!\bigl(\cA\bigr)\quad {\rm with}\ \, \cA=\sqrt{-\Delta_{\mathbb H^n}-\rho^2}  
	\end{equation}
	where $\operatorname{Spec}(\cA^2)=[0,\infty)$ and, in the Helgason--Fourier representation,
	\begin{equation}\label{eq:m_gamma}
		m_s(\beta)
		=2^{2s}\,
		\frac{\bigl|\Gamma\bigl(\frac{3+2s}{4}+\frac{\bi}{2}\beta\bigr)\bigr|^2}
		{\bigl|\Gamma\bigl(\frac{3-2s}{4}+\frac{\bi}{2}\beta\bigr)\bigr|^2},
		\qquad \beta\in\mathbb R.
	\end{equation}
	
	\begin{lemma}
		\label{lem:mgamma-holo-symbol}
		Let  
		\[
		\gamma>0,\qquad a=\frac{3+2\gamma}{4}>0,\qquad b=\frac{3-2\gamma}{4}
		\]
		and 
		\begin{equation}\label{eq:mgamma-holo-def}
			m_\gamma(\zeta)
			:=2^{2\gamma}\,
			\frac{\Gamma\!\bigl(a+\frac{\bi}{2}\zeta\bigr)\,\Gamma\!\bigl(a-\frac{\bi}{2}\zeta\bigr)}
			{\Gamma\!\bigl(b+\frac{\bi}{2}\zeta\bigr)\,\Gamma\!\bigl(b-\frac{\bi}{2}\zeta\bigr)}\quad \text{for $\zeta\in\mathbb C$}.
		\end{equation}
		Then $m_\gamma$ is holomorphic in every strip $|{\rm Im}(\zeta)|<\sigma$ with $\sigma<\epsilon$,
		where one may take $\epsilon:=2a=\frac{3+2\gamma}{2}.$
		Moreover, for each $\sigma<\epsilon$ and each $k\in\mathbb N_0$ there exists a constant
		$C_{k,\gamma,\sigma}>0$ such that
		\begin{equation}\label{eq:mgamma-symbol}
			|\partial_\zeta^k m_\gamma(\zeta)|
			\le C_{k,\gamma,\sigma}\,(1+|\zeta|)^{2\gamma-k},
			\qquad |{\rm Im}(\zeta)|<\sigma,
		\end{equation}
		where ${\rm Im}(\zeta)=\zeta_2$ if $\zeta=\zeta_1+\bi\zeta_2$, $\zeta_1,\zeta_2\in\R$. 
	\end{lemma}
	
	\begin{proof}
		The only possible singularities of \eqref{eq:mgamma-holo-def} come from the poles of
		$\Gamma\!\bigl(a\pm \frac{\bi}{2}\zeta\bigr)$, which occur precisely when
		$a\pm \frac{\bi}{2}\zeta\in\{0,-1,-2,\dots\}$, i.e. $\zeta=\pm 2\bi(a+k),\,k\in\mathbb N_0.$
		Hence $m_\gamma$ is holomorphic in the strip $|{\rm Im}(\zeta)|<2a=:\epsilon$.
		Poles of $\Gamma\!\bigl(b\pm \frac{\bi}{2}\zeta\bigr)$, if any, produce zeros of
		$m_\gamma$ and therefore do not affect holomorphy.
		
		Fix $\sigma<\epsilon$ and write
		\[
		m_\gamma(\zeta)=2^{2\gamma}\,R_+(\zeta)\,R_-(\zeta),
		\qquad
		R_\pm(\zeta):=\frac{\Gamma\!\bigl(a\pm \frac{\bi}{2}\zeta\bigr)}
		{\Gamma\!\bigl(b\pm \frac{\bi}{2}\zeta\bigr)}.
		\]
		Note that $a=b+\gamma$. For $|{\rm Im}(\zeta)|<\sigma$ and $\pm$ fixed, set
		\[
		z_\pm:=b\pm \frac{\bi}{2}\zeta,
		\qquad\text{so that}\qquad
		R_\pm(\zeta)=\frac{\Gamma(z_\pm+\gamma)}{\Gamma(z_\pm)}.
		\]
		
		Uniform Stirling estimates in vertical strips yield
		\[
		|R_\pm(\zeta)|\le C_{\gamma,\sigma}\,(1+|\zeta|)^{\gamma},
		\qquad |{\rm Im}(\zeta)|<\sigma,
		\]
		and therefore
		\[
		|m_\gamma(\zeta)|\le C_{\gamma,\sigma}\,(1+|\zeta|)^{2\gamma},
		\qquad |{\rm Im}(\zeta)|<\sigma.
		\]
		
		We next estimate derivatives for $|\zeta|\to\infty$ inside the strip.
		Choose $R_0\ge 1$ so large that for all $\zeta$ with $|{\rm Im}(\zeta)|<\sigma$ and
		$|\zeta|\ge R_0$, the points $z_\pm$ stay at a positive distance from the poles
		$\{0,-1,-2,\dots\}$ of the digamma and polygamma functions.
		On the region
		\[
		\Omega_{R_0,\sigma}:=\{\zeta\in\mathbb C:\ |{\rm Im}(\zeta)|<\sigma,\ |\zeta|\ge R_0\},
		\]
		we may choose a holomorphic branch of $\log R_\pm$ and use
		$\frac{d}{dz}\log\Gamma(z)=\psi(z)$ to compute
		\[
		\partial_\zeta\log R_\pm(\zeta)
		=\pm \frac{\bi}{2}\Bigl(\psi(z_\pm+\gamma)-\psi(z_\pm)\Bigr).
		\]
		By the uniform asymptotic expansions in vertical strips,
		\[
		\psi(z)=\log z+O\!\left(\frac1{z}\right),
		\qquad
		\psi^{(j)}(z)=O\!\left(\frac1{|z|^{j}}\right)\quad (j\ge 1),
		\qquad |z|\to\infty,
		\]
		and $\log(z+\gamma)-\log z = O(1/z)$, we obtain
		\[
		\psi(z_\pm+\gamma)-\psi(z_\pm)=O\!\left(\frac1{|z_\pm|}\right)
		=O\!\left(\frac1{|\zeta|}\right)
		\quad{\rm for}\ \,  |\zeta|\to\infty,\quad \zeta\in\Omega_{R_0,\sigma}.
		\]
		Moreover, for every $j\ge 1$, the polygamma bounds yield
		\[
		\partial_\zeta^{\,j}\Bigl(\psi(z_\pm+\gamma)-\psi(z_\pm)\Bigr)
		=O\!\left(\frac1{|\zeta|^{j+1}}\right),
		\qquad |\zeta|\to\infty,\quad \zeta\in\Omega_{R_0,\sigma}.
		\]
		Consequently, for all $j\in\mathbb N_0$,
		\[
		\Bigl|\partial_\zeta^{\,j}\bigl(\partial_\zeta\log R_\pm(\zeta)\bigr)\Bigr|
		\le C_{j,\gamma,\sigma}\,(1+|\zeta|)^{-j-1},
		\qquad \zeta\in\Omega_{R_0,\sigma}.
		\]
		Since $R_\pm=\exp(\log R_\pm)$, Fa\`a di Bruno's formula (Bell polynomials)
		implies that for every $k\in\mathbb N_0$,
		\[
		|\partial_\zeta^{\,k}R_\pm(\zeta)|
		\le C_{k,\gamma,\sigma}\,|R_\pm(\zeta)|\,(1+|\zeta|)^{-k}
		\le C_{k,\gamma,\sigma}\,(1+|\zeta|)^{\gamma-k},
		\qquad \zeta\in\Omega_{R_0,\sigma}.
		\]
		
		Consider the compact set
		\[
		K_{R_0,\sigma}:=\big\{\zeta\in\mathbb C:\ |{\rm Im}(\zeta)|\le \sigma,\ |\zeta|\le R_0\big\}.
		\]
		Since $m_\gamma$ (hence $R_\pm$) is holomorphic in a neighbourhood of
		$K_{R_0,\sigma}$, each derivative $\partial_\zeta^{\,k}R_\pm$ is continuous and
		bounded on $K_{R_0,\sigma}$. Thus there exists $C_{k,\gamma,\sigma}>0$ such that
		\[
		|\partial_\zeta^{\,k}R_\pm(\zeta)|
		\le C_{k,\gamma,\sigma}\,(1+|\zeta|)^{\gamma-k},
		\qquad \zeta\in K_{R_0,\sigma}.
		\]
		Combining this with the estimate on $\Omega_{R_0,\sigma}$ yields, for all
		$\zeta$ with $|{\rm Im}(\zeta)|<\sigma$,
		\[
		|\partial_\zeta^{\,k}R_\pm(\zeta)|
		\le C_{k,\gamma,\sigma}\,(1+|\zeta|)^{\gamma-k}.
		\]
		Finally, by Leibniz' rule and the previous bounds,
		\[
		|\partial_\zeta^{\,k}m_\gamma(\zeta)|
		\le C_{k,\gamma,\sigma}\sum_{j=0}^k
		(1+|\zeta|)^{\gamma-j}(1+|\zeta|)^{\gamma-(k-j)}
		\le C_{k,\gamma,\sigma}\,(1+|\zeta|)^{2\gamma-k},
		\qquad |{\rm Im}(\zeta)|<\sigma,
		\]
		which proves \eqref{eq:mgamma-symbol}.
	\end{proof}

	\begin{remark}
		The symbol estimate also holds for $\widetilde m_{\gamma}$ (see \eqref{mstl}), with the strip restricted to $|\operatorname{Im}\zeta|<\sigma<\gamma+\frac12$. The proof is analogous.
	\end{remark}
	
	The symbol estimates obtained above immediately yield the Sobolev mapping properties of the fractional operators. In particular, both $\cP_\gamma$ and $\widetilde{\cP}_\gamma$ act as operators of order \(2\gamma\) on the hyperbolic Bessel potential scale.

	\begin{lemma}\label{lem:Ps_mapping_Hsp}
		Let $\gamma\in\bigl(0,\frac{n}{2}\bigr)$. Then both operators
		$\cP_\gamma$ and $\widetilde\cP_\gamma$ extend by continuity to bounded
		linear maps
		\begin{equation}\label{eq:Pgamma_mapping}
			\widetilde H^{\,s}(\mathbb H^n)
			\longrightarrow
			\widetilde H^{\,s-2\gamma}(\mathbb H^n),
			\qquad s\in\mathbb R.
		\end{equation}
	\end{lemma}
	
	\begin{proof}
		By the real-axis multiplier estimates established above,
		\[
		|m_\gamma(\beta)|+|\widetilde m_\gamma(\beta)|
		\le C(1+\beta^2)^\gamma
		\le C(\mu+\rho^2+\beta^2)^\gamma.
		\]
		Hence, by Plancherel,
		\[
		\begin{aligned}
			\|\cP_\gamma u\|_{\widetilde H^{s-2\gamma}}^2
			&=
			\int_{\mathbb R}\int_{\mathbb S^{n-1}}
			(\mu+\rho^2+\beta^2)^{s-2\gamma}
			|m_\gamma(\beta)|^2
			|\widehat u(\beta,\theta)|^2
			\frac{d\sigma(\theta)\,d\beta}{|c(\beta)|^2}\le
			C\|u\|_{\widetilde H^s}^2.
		\end{aligned}
		\]
		The same argument applies to $\widetilde\cP_\gamma$. The asserted
		extensions follow by density of $C_c^\infty(\mathbb H^n)$ in
		$\widetilde H^s(\mathbb H^n)$.
	\end{proof}
	
	\subsection{Proof of Theorems~\ref{mainthe2} and~\ref{mainthe}}
	
	In this subsection, we prove nonvanishing by separating bounded and high
	spectral frequencies, and then obtain attainment by the Brezis--Lieb lemma.
	
	\begin{lemma}\label{lem:spectral-nonvanishing}
		Let $n\ge2$, $0<s<\frac n2$, and $q_*:=\frac{2n}{n-2s}.$
		Let $T=m(\cA)$ be either $\cP_s$ or $\widetilde\cP_s$, set
		$b:=m(0)$, and let $\lambda\le b$. Denote the corresponding completed
		energy space by $\mathcal X_\lambda$, with quadratic form $\mathcal Q_\lambda(u):=\|u\|_{\mathcal X_\lambda}^2.$
		For $2<q\le q_*$, define
		\[
		S_{\lambda,q}
		:=
		\inf_{u\in\mathcal X_\lambda\setminus\{0\}}
		\frac{\mathcal Q_\lambda(u)}{\|u\|_q^2},
		\qquad
		C_0
		:=
		\inf_{f\in C_c^\infty(\mathbb H^n)\setminus\{0\}}
		\frac{\langle Tf,f\rangle_2}{\|f\|_{q_*}^2}.
		\]
		Assume either $q<q_*$, or $q=q_*$ and $S_{\lambda,q_*}<C_0.$
		If $\|u_j\|_q=1$ and $\mathcal Q_\lambda(u_j)\rightarrow S_{\lambda,q},$
		then there exist hyperbolic isometries $\mathcal T_j$ such that, after
		passing to a subsequence,
		\[
		u_j\circ\mathcal T_j\rightharpoonup v
		\qquad\text{weakly in }\mathcal X_\lambda
		\]
		for some $v\ne0$.
	\end{lemma}
	
	\begin{proof}
		By \cite[Theorems~1.3 and~1.8]{lu2023explicit},
		\begin{equation}\label{eq:energy-all-exponents}
			\|u\|_r
			\le
			C_r\,\mathcal Q_\lambda(u)^{1/2},
			\qquad
			2<r\le q_*.
		\end{equation}
		In particular, $S_{\lambda,q}>0$. Choose an even $\chi\in C_c^\infty(\mathbb R)$ such that
		$0\le\chi\le1$, $\chi\equiv1$ on $[-1,1]$, and
		$\operatorname{supp}\chi\subset[-2,2]$. For $R\ge1$, set
		\[
		L_R:=\chi(\cA/R),
		\qquad
		J_R:=I-L_R,
		\qquad
		\ell_j:=L_Ru_j,
		\qquad
		h_j:=J_Ru_j.
		\]
		Both $L_R$ and $J_R$ are contractions on $\mathcal X_\lambda$.
		Since $m(\beta)-\lambda\asymp|\beta|^{2s}$ as $|\beta|\to\infty,$
		for all sufficiently large $R$,
		\begin{equation}\label{eq:high-frequency-small-L2}
			\|J_Ru\|_2
			\le
			CR^{-s}\mathcal Q_\lambda(u)^{1/2},
			\qquad
			J_Ru\in H^s(\mathbb H^n).
		\end{equation}
		
		Suppose first that $q<q_*$. Choose $\theta\in(0,1)$ such that
		\[
		\frac1q=\frac{\theta}{2}+\frac{1-\theta}{q_*}.
		\]
		By interpolation, \eqref{eq:energy-all-exponents}, and
		\eqref{eq:high-frequency-small-L2},
		\[
		\|h_j\|_q
		\le
		\|h_j\|_2^\theta\|h_j\|_{q_*}^{1-\theta}
		\le
		CR^{-s\theta}\mathcal Q_\lambda(u_j)^{1/2}.
		\]
		Hence $R$ can be fixed so large that $\limsup_{j\to\infty}\|h_j\|_q<1.$
		
		Now let $q=q_*$. Since $h_j\in H^s$, the defining inequality for $C_0$
		extends to $h_j$ by density, and therefore $C_0\|h_j\|_{q_*}^2
		\le
		\langle Th_j,h_j\rangle_2.$
		Set
		\[
		\varepsilon_R
		:=
		\sup_{|\beta|\ge R}
		\frac{\max\{\lambda,0\}}{m(\beta)-\lambda}.
		\]
		Then $\varepsilon_R\to0$ as $R\to\infty$, and
		\begin{equation}\label{eq:critical-high-frequency-gap}
			C_0\|h_j\|_{q_*}^2
			\le
			\langle Th_j,h_j\rangle_2
			\le
			(1+\varepsilon_R)\mathcal Q_\lambda(h_j)
			\le
			(1+\varepsilon_R)\mathcal Q_\lambda(u_j).
		\end{equation}
		Since $S_{\lambda,q_*}<C_0$, we may again choose $R$ so large that $\limsup_{j\to\infty}\|h_j\|_{q_*}<1.$
		
		Point evaluation after $L_R$ is continuous on
		$\mathcal X_\lambda$. Let $o$ be the origin in the ball model. By the
		inversion formula and Cauchy--Schwarz,
		\begin{equation}\label{eq:bounded-frequency-evaluation}
			|(L_Ru)(o)|^2
			\le
			\mathcal Q_\lambda(u)
			\int_{\mathbb R}
			\frac{|\chi(\beta/R)|^2}{m(\beta)-\lambda}
			\frac{d\beta}{|c(\beta)|^2}
			\le
			C_R\mathcal Q_\lambda(u).
		\end{equation}
		If $\lambda<b$, the integral is immediately finite near $\beta=0$.
		If $\lambda=b$, the estimates established above give
		\[
		m(\beta)-b\asymp\beta^2,
		\qquad
		|c(\beta)|^{-2}\asymp\beta^2
		\qquad \beta\to0,
		\]
		so the integrand remains locally integrable. By isometry invariance,
		\eqref{eq:bounded-frequency-evaluation} holds uniformly at every point.
		The estimate extends by density to $\mathcal X_\lambda$; applying the
		same argument to derivatives on compact sets gives a smooth representative
		of $L_Ru$.
		
		Thus, in either case, $\liminf_{j\to\infty}\|\ell_j\|_q>0.$
		Fix $r\in(2,q)$. By \eqref{eq:energy-all-exponents},
		$\{\ell_j\}$ is bounded in $L^r$, and hence $\|\ell_j\|_q^q
		\le
		\|\ell_j\|_\infty^{q-r}\|\ell_j\|_r^r.$
		Consequently, $\liminf_{j\to\infty}\|\ell_j\|_\infty>0.$
		
		By $\liminf_{j\to\infty}\|\ell_j\|_\infty>0$, choose $x_j\in\mathbb H^n$ such that $|\ell_j(x_j)|\ge c>0.$
		Let $\mathcal T_j$ be hyperbolic isometries satisfying
		$\mathcal T_j(o)=x_j$, and set $v_j:=u_j\circ\mathcal T_j.$
		The isometries act unitarily on $\mathcal X_\lambda$ and commute with
		$L_R$. Passing to a subsequence,
		\[
		v_j\rightharpoonup v
		\qquad\text{weakly in }\mathcal X_\lambda.
		\]
		Since $u\mapsto(L_Ru)(o)$ is a continuous linear functional on
		$\mathcal X_\lambda$, we have
		\[
		(L_Rv_j)(o)\rightarrow(L_Rv)(o).
		\]
		Moreover, $|(L_Rv_j)(o)|
		=
		|\ell_j(x_j)|
		\ge c.$
		Hence $(L_Rv)(o)\ne0$, and therefore $v\ne0$.
	\end{proof}
	
	\medskip
	\noindent\textbf{Proof of Theorem~\ref{mainthe}.}
	Set
	\[
	q:=p+1,
	\qquad
	H:=H_{n,s,p}(\lambda),
	\]
	and choose a minimizing sequence
	$\{u_j\}\subset\mathcal H^s_\lambda(\mathbb H^n)$ such that
	\[
	\|u_j\|_q=1,
	\qquad
	\mathcal E_{\lambda,s}(u_j)\longrightarrow H.
	\]
	By \eqref{eq:energy-all-exponents}, $H>0$. Apply Lemma~\ref{lem:spectral-nonvanishing} with $T=\cP_s$. In the
	critical case $q=q_*$,
	\[
	S_{\lambda,q_*}=H_{n,s}(\lambda),
	\qquad
	C_0=H_{n,s}(0),
	\]
	so the hypothesis of the lemma is precisely the assumed strict gap $H_{n,s}(\lambda)<H_{n,s}(0).$
	After composing with suitable hyperbolic isometries and passing to a
	subsequence, we obtain
	\[
	v_j\rightharpoonup v\ne0
	\qquad\text{weakly in }\mathcal H^s_\lambda(\mathbb H^n).
	\]
	
	If $\lambda<\lambda_{0,s}^{\mathrm{conf}}$, the energy norm is
	equivalent to the $H^s$-norm. At
	$\lambda=\lambda_{0,s}^{\mathrm{conf}}$, the local $H^s$ estimate proved
	in Lemma~\ref{lem:critical-compact-support-and-weak-Hs} applies after
	multiplication by a compactly supported cutoff. Thus local Rellich
	compactness gives, after passing to a further subsequence, $v_j\rightarrow v$ a.e. in $\mathbb H^n.$
	
	Set $a:=\|v\|_q^q\in(0,1].$
	By the Brezis--Lieb lemma \cite[Theorem~1]{BrezisLieb1983},
	\begin{equation}\label{eq:BL-main-proof}
		\|v_j-v\|_q^q
		=
		1-a+o(1).
	\end{equation}
	Since $\mathcal H^s_\lambda$ is a Hilbert space and
	$v_j\rightharpoonup v$,
	\[
	\mathcal E_{\lambda,s}(v_j)
	=
	\mathcal E_{\lambda,s}(v)
	+
	\mathcal E_{\lambda,s}(v_j-v)
	+
	o(1).
	\]
	By the definition of $H$ and \eqref{eq:BL-main-proof}, $H
	\ge
	H\Bigl(a^{2/q}+(1-a)^{2/q}\Bigr).$
	Since $2/q<1$ and $a>0$, this is possible only if $a=1$. Hence $\|v\|_q=1,\,
	\mathcal E_{\lambda,s}(v)=H,$
	and
	\[
	v_j\longrightarrow v
	\qquad\text{strongly in }\mathcal H^s_\lambda(\mathbb H^n)
	\]
	and in $L^q(\mathbb H^n)$. Thus $v$ attains the infimum. The Euler--Lagrange equation is
	\[
	\langle v,\varphi\rangle_{\lambda,s}
	=
	H\int_{\mathbb H^n}
	|v|^{p-1}v\,\varphi\,dV_{\mathbb H^n},
	\qquad
	\varphi\in\mathcal H^s_\lambda(\mathbb H^n).
	\]
	Therefore, $H^{1/(p-1)}v$
	is a nontrivial weak solution of \eqref{eq:BN-Hn-frac}.
	\hfill$\Box$\medskip
	
	\noindent\textbf{Proof of Theorem~\ref{mainthe2}.}
	Apply Lemma~\ref{lem:spectral-nonvanishing} with $T=\widetilde\cP_s,\,
	\mathcal X_\lambda=\mathcal H^s_{\lambda,\sim}(\mathbb H^n).$
	In the critical case,
	\[
	S_{\lambda,q_*}
	=
	\widetilde H_{n,s}(\lambda),
	\qquad
	C_0
	=
	\widetilde H_{n,s}(0),
	\]
	so the required hypothesis is exactly $\widetilde H_{n,s}(\lambda)
	<
	\widetilde H_{n,s}(0).$
	The preceding nonvanishing, Brezis--Lieb splitting, and
	Euler--Lagrange argument apply verbatim with the tilde quadratic form.
	Hence the infimum is attained, and after the corresponding rescaling
	one obtains a nontrivial weak solution of \eqref{eq:BN-Hn-frac2}.
	\hfill$\Box$\medskip
	
	\subsection{Proof of Theorem \ref{prop:strict-gap-fractional}, \ref{thm:strict-gap-Bplus-Bzero}}
	
	In this subsection, we analyze the monotonicity, attainability, and strict-gap regimes of the Poincar\'e--Sobolev levels associated with \(\cP_s\) and \(\widetilde{\cP}_s\), and prove Theorem~\ref{prop:strict-gap-fractional} and Theorem~\ref{thm:strict-gap-Bplus-Bzero}. 
	The key step is to show the threshold phenomenon that, once \(\lambda\) is above the bottom of the corresponding spectrum, the level drops instantly to \(-\infty\). 
	For nonlocal operators, this requires delicate control of far-interaction terms. To obtain this off-diagonal decay, we use the conformal identity and an explicit spherical transform.
	
	Define the operator $\bA_s:\ C_c^\infty(\mathbb B^n)\longrightarrow \mathcal D'(\mathbb B^n)$
	by the duality pairing
	\[
	\langle \bA_sf,\varphi\rangle
	:=\int_{\mathbb B^n} (\widetilde{\cP}_s f)(x)\,\varphi(x)\,dV_{\mathbb B^n}(x)
	\quad{\rm for}\ \,  f,\varphi\in C_c^\infty(\mathbb B^n).
	\]
	In particular, $\bA_s$ is linear by construction.
	
	\begin{lemma}
		\label{lem:seq-continuity-Ptilde}
		The operator $\bA_s$ is  continuous in the following sense: if $f_j\to f$ in
		$C_c^\infty(\mathbb B^n)$, then for every $\varphi\in C_c^\infty(\mathbb B^n)$,
		\[
		\langle \bA_sf_j,\varphi\rangle\longrightarrow \langle \bA_sf,\varphi\rangle \quad {\rm as}\ j\to+\infty.
		\]
	\end{lemma}
	
	\begin{proof}
		Let $f_j\to f$  in $C_c^\infty(\mathbb B^n)$ as $j\to+\infty$, then there exists a compact set $K\Subset \mathbb B^n$ such that
		$\operatorname{supp} f_j\subset K$ for all $j$, and for every integer $\ell\ge0$,
		\[
		\max_{|\alpha|\le \ell}\ \sup_{x\in K}\bigl|\partial^\alpha(f_j-f)(x)\bigr|
		\longrightarrow 0 .
		\]
		In particular, for fixed integer $m>2s$ we have
		\[
		\|f_j-f\|_{H^{m}(\mathbb B^n)}\longrightarrow 0\quad {\rm as}\ j\to+\infty.
		\]
		By Lemma \ref{lem:Ps_mapping_Hsp},
		\[
		\|\widetilde{\cP}_s(f_j-f)\|_{H^{m-2s}(\mathbb B^n)}
		\le C_m\,\|f_j-f\|_{H^{m}(\mathbb B^n)} \longrightarrow 0.
		\]
		Since $m-2s>0$, we have the continuous embedding
		$H^{m-2s}(\mathbb B^n)\hookrightarrow L^2(\mathbb B^n)$, hence
		\[
		\|\widetilde{\cP}_s(f_j-f)\|_{L^2(\mathbb B^n)}\longrightarrow 0.
		\]
		Fix any $\varphi\in C_c^\infty(\mathbb B^n)\subset L^2(\mathbb B^n)$. Then by
		Cauchy--Schwarz,
		\[
		\bigl|\langle \bA_s f_j-\bA_sf,\varphi\rangle\bigr|
		=\left|\int_{\mathbb B^n} \bigl(\widetilde{\cP}_s(f_j-f)\bigr)(x)\,\varphi(x)\,dV_{\mathbb B^n}(x)\right|
		\le \|\widetilde{\cP}_s(f_j-f)\|_{L^2}\,\|\varphi\|_{L^2}
		\longrightarrow 0.
		\]
		This proves $\langle \bA_sf_j,\varphi\rangle\to \langle \bA_sf,\varphi\rangle$ for every
		$\varphi\in C_c^\infty(\mathbb B^n)$, i.e. $\bA_sf_j\rightharpoonup \bA_s f$ in
		$\mathcal D'(\mathbb B^n)$.
	\end{proof}\medskip
	
	Now  all assumptions of the Schwartz kernel theorem corresponding to smooth manifolds are satisfied, and there exists
	a unique distribution $K_s\in \mathcal D'(\mathbb B^n\times \mathbb B^n)$ such that
	\[
	\langle \widetilde{\cP}_s f,\varphi\rangle=\langle K_s,\,\varphi\otimes f\rangle,
	\qquad f,\varphi\in C_c^\infty(\mathbb B^n),
	\]
	where $(\varphi\otimes f)(x,y):=\varphi(x)f(y)$. 
	Equivalently,
	\[
	\int_{\mathbb B^n} (\widetilde{\cP}_s f)(x)\,\varphi(x)\,dV_{\mathbb B^n}(x)
	=
	\iint_{\mathbb B^n\times\mathbb B^n} K_s(x,y)\,\varphi(x)\,f(y)\,dV_{\mathbb B^n}(x)\,dV_{\mathbb B^n}(y),
	\]
	with the right-hand side interpreted as the pairing of  $K_s$
	with the test function $(x,y)\mapsto \varphi(x)f(y)$.
	
	Let $G=\mathrm{Isom}(\mathbb B^n)$ act on functions by $(L_g f)(x):=f(g^{-1}x),\,g\in G.$
	By Lemma \ref{lem:isom-inv-Ps},
	\[
	\widetilde{\cP}_s L_g = L_g \widetilde{\cP}_s \qquad \text{for all } g\in G .
	\]
	The commutation relation implies that $K_s$ is $G$-invariant in the distributional sense:
	\[
	(g\times g)^*K_s=K_s \quad \text{in } \mathcal D'(\mathbb B^n\times\mathbb B^n),
	\]
	i.e. $\langle K_s,\Phi\rangle=\langle K_s,\Phi\circ(g\times g)\rangle$ for all
	$\Phi\in C_c^\infty(\mathbb B^n\times\mathbb B^n)$.
	Consequently, $K_s$ is radial: it depends only on the geodesic distance
	$d(x,y)$. In particular, there exists a radial distribution
	$k_s$ on $[0,\infty)$ such that, for every $f,\varphi\in C_c^\infty(\mathbb B^n)$,
	\[
	\int_{\mathbb B^n} (\widetilde{\cP}_s f)(x)\,\varphi(x)\,dV_{\mathbb B^n}(x)
	=\iint_{\mathbb B^n\times\mathbb B^n} k_s\!\bigl(d(x,y)\bigr)\,\varphi(x)\,f(y)\,
	dV_{\mathbb B^n}(x)\,dV_{\mathbb B^n}(y).
	\]
	
	Recall that the M\"obius transformation $T_x$ is an isometry of $(\mathbb B^n,g_{\mathbb B})$
	and satisfies $T_x(x)=0$. Hence,
	\[
	d(x,y)=d\bigl(T_x(x),T_x(y)\bigr)=d\bigl(0,T_x(y)\bigr).
	\]
	Write the corresponding radial distribution on $\mathbb B^n$ as $\kappa_s(z):=k_s\!\bigl(d(0,z)\bigr),\,z\in\mathbb B^n$, with the kernel identities understood distributionally.
	Then, using the hyperbolic convolution (\ref{sjak}), we obtain that	for every $\varphi\in C_c^\infty(\mathbb B^n)$
	\begin{align*}
		\int_{\mathbb B^n} (\widetilde{\cP}_s f)(x)\,\varphi(x)\,dV_{\mathbb B^n}(x)
		=\iint_{\mathbb B^n\times\mathbb B^n} \kappa_s\!\bigl(T_x(y)\bigr)\,\varphi(x)\,f(y)\,
		dV_{\mathbb B^n}(x)\,dV_{\mathbb B^n}(y).
	\end{align*}
	By the Helgason--Fourier transform, we can get
	\begin{equation}\label{eq:spherical-transform-kappa}
		\widehat{\kappa_s}(\beta)=\widetilde m_s(\beta)\quad{\rm for}\ \,  \beta\in\R.
	\end{equation}
	in the following sense: for every test function $\psi\in C_c^\infty(\mathbb B^n)$,
	\begin{equation}\label{eq:kappa-distributional-pairing}
		\langle \kappa_s,\psi\rangle
		=
		\int_{\R}\widetilde m_s(\beta)\,
		\Bigl(\int_{\mathbb B^n}\psi(x)\,\Phi_\beta(x)\,dV_{\mathbb B^n}(x)\Bigr)\,
		\frac{d\beta}{|c(\beta)|^2}.
	\end{equation}
	
	Define spherical function (see \cite{LiuPeng2009,lu2023explicit}) $\Phi_\beta(x):=\int_{\mathbb{S}^{n-1}}h_{-\beta,\theta}(x) d\sigma(\theta),$
	where $h_{-\beta,\theta}$ is defined in (\ref{def-h}). 
	Then $\Phi_\beta(0)=1,$ $\Phi_\beta(x)$ is radial and has the explicit Legendre representation
	\[	\Phi_\beta(x)
	:=
	2^{\frac{n-2}{2}}\,\Gamma\!\Bigl(\frac n2\Bigr)\,
	(\sinh d(x,0))^{\frac{2-n}{2}}\,
	P_{-\frac12+\bi\beta}^{\,\frac{2-n}{2}}(\cosh d(x,0)),\]
	where $P_\nu^\mu$ is the associated Legendre function of the first kind. It is known that $P_\nu^\mu$  can be defined via hypergeometric function:
	\[	P_\nu^{\mu}(z)
	=
	\frac{1}{\Gamma(1-\mu)}
	\left(\frac{z+1}{z-1}\right)^{\frac{\mu}{2}}
	\,{}_2F_1\!\left(-\nu,\nu+1;1-\mu;\frac{1-z}{2}\right),\quad \mu\ne 1,2,3\cdots\]
	where ${}_2F_1(a,b;c;z)$ is a hypergeometric function
	\[	{}_2F_1(a,b;c;z)
	=
	\sum_{k=0}^{\infty}
	\frac{(a)_k(b)_k}{(c)_k}\,\frac{z^k}{k!},
	\qquad |z|<1,\quad c\neq 0,-1,-2,\dots.\]
	Outside the disk of convergence, the hypergeometric function is understood by analytic continuation.
	
	The next proposition provides the crucial off-diagonal exponential decay estimate for \(\cP_s\) and \(\widetilde{\cP}_s\), which controls long-range interactions between separated supports.

	\begin{proposition}\label{lem:offdiag-exp-decay-Ptilde}
		Let $n\ge2$, $s\in\bigl(0,\frac n2\bigr)$, and $\rho:=\frac{n-1}{2}.$
		Then there exists $C_{n,s}>0$ such that for every
		$f,g\in C_c^\infty(\mathbb B^n)$ satisfying
		\[
		\operatorname{dist}
		\bigl(\operatorname{supp}f,\operatorname{supp}g\bigr)
		\ge R>1,
		\]
		one has
		\[
		\bigl|\langle \widetilde{\cP}_s f,g\rangle_{L^2(\mathbb B^n)}\bigr|
		+
		\bigl|\langle \cP_s f,g\rangle_{L^2(\mathbb B^n)}\bigr|
		\le
		C_{n,s}e^{-\rho R}
		\|f\|_{L^1(\mathbb B^n)}
		\|g\|_{L^1(\mathbb B^n)}.
		\]
	\end{proposition}
	
	\begin{proof}
		If $s\in\mathbb N$, then both $\widetilde{\cP}_s$ and $\cP_s$ are
		local differential operators. Hence, under the support separation
		assumption, $\langle \widetilde{\cP}_s f,g\rangle
		=
		\langle \cP_s f,g\rangle
		=
		0.$
		
		Assume henceforth that $s\notin\mathbb N$. Set
		\[
		a_{n,s}
		:=
		\frac{4^s\Gamma(\frac n2+s)}
		{\pi^{n/2}\Gamma(-s)}.
		\]
		For $x\neq y$, the distribution kernel of $(-\Delta)^s$ is $a_{n,s}|x-y|^{-n-2s}.$
		Using the conformal covariance of $\widetilde{\cP}_s$ and
		\[
		\phi(x)\phi(y)|x-y|^2
		=
		4\sinh^2\!\left(\frac{d(x,y)}2\right),
		\]
		we obtain the off-diagonal kernel
		\[
		\widetilde k_s(r)
		=
		a_{n,s}
		\left(2\sinh\frac r2\right)^{-n-2s},
		\qquad r>0.
		\]
		
		To fix the normalization of the second kernel, set
		$F_s(r):=(2\cosh(r/2))^{-n-2s}$.
		Using the normalized spherical function $\Phi_\beta$ above and polar
		coordinates, the radial transform in our convention is
		\[
		\widehat F_s(\beta)
		=(2\pi)^{n/2}\int_0^\infty
		F_s(r)(\sinh r)^{n/2}
		P_{-\frac12+i\beta}^{\,\frac{2-n}{2}}(\cosh r)\,dr.
		\]
		Applying \cite[Eqs.~(2.24) and~(3.7)]{lu2023explicit} with exponent
		$n+2s>n-1$ gives
		\[
		\widehat F_s(\beta)
		=\frac{\pi^{n/2}}{4^s\Gamma(\frac n2+s)\Gamma(s+1)}
		\left|\Gamma\!\left(s+\frac12+i\beta\right)\right|^2.
		\]
		Hence the kernel of
		\[
		M_s(\cA)
		:=
		\left|\Gamma\!\left(s+\frac12+i\cA\right)\right|^2
		\]
		is
		\[
		K_{M_s}(r)
		=
		\frac{
			4^s\Gamma(\frac n2+s)\Gamma(s+1)
		}{
			\pi^{n/2}
		}
		\left(2\cosh\frac r2\right)^{-n-2s}.
		\]
		Combining this with \eqref{eq:P-vs-tildeP} and
		\[
		\Gamma(-s)\Gamma(s+1)
		=
		-\frac{\pi}{\sin(\pi s)},
		\]
		we obtain the off-diagonal kernel of $\cP_s$:
		\[
		k_s(r)
		=
		a_{n,s}
		\left[
		\left(2\sinh\frac r2\right)^{-n-2s}
		-
		\left(2\cosh\frac r2\right)^{-n-2s}
		\right].
		\]
		
		For $r\ge1$, $2\sinh(r/2)\asymp e^{r/2},\,
		2\cosh(r/2)\asymp e^{r/2},$
		and therefore
		\[
		|k_s(r)|+|\widetilde k_s(r)|
		\le
		C_{n,s}e^{-(\frac n2+s)r}.
		\]
		Since $\frac n2+s>\rho,$
		it follows that
		\begin{equation}\label{eq:offdiag-kernel-decay}
			|k_s(r)|+|\widetilde k_s(r)|
			\le
			C_{n,s}e^{-\rho r},
			\qquad r\ge1.
		\end{equation}
		
		Because $d(x,y)\ge R$
		for $x\in\operatorname{supp}g,\,
		y\in\operatorname{supp}f,$
		the two kernels are smooth on the region of integration. Hence
		\[
		\begin{aligned}
			\bigl|\langle \widetilde{\cP}_s f,g\rangle\bigr|
			&\le
			\iint
			|\widetilde k_s(d(x,y))|
			\,|f(y)|\,|g(x)|
			\,dV_{\mathbb B^n}(y)\,dV_{\mathbb B^n}(x) \\
			&\le
			C_{n,s}e^{-\rho R}
			\|f\|_{L^1(\mathbb B^n)}
			\|g\|_{L^1(\mathbb B^n)}.
		\end{aligned}
		\]
		The same estimate follows from \eqref{eq:offdiag-kernel-decay} for
		$\cP_s$. This proves the result.
	\end{proof}\medskip
	
	As a direct consequence of the off-diagonal decay estimate, we obtain the threshold blow-down phenomenon: once \(\lambda\) exceeds the spectral bottom, the corresponding Poincar\'e--Sobolev level equals \(-\infty\).

	\begin{proposition}\label{prop:Hnslambda-minus-infty}
		Let $n\ge2$ and  $s\in\bigl(0,\frac n2\bigr)$. Then,
		
		(i) for $\lambda>\widetilde\lambda_{0,s}^{\mathrm{conf}}$, $\widetilde H_{n,s}(\lambda)=-\infty;$
		
		(ii) for $\lambda>\lambda_{0,s}^{\mathrm{conf}}$, $H_{n,s}(\lambda)=-\infty.$
	\end{proposition}
	
	\begin{proof}
		(i) Fix $\lambda>\widetilde \lambda_{0,s}^{\mathrm{conf}}$. By the definition of
		$\widetilde\lambda_{0,s}^{\mathrm{conf}}$, there exists $\varphi\in C_c^\infty(\mathbb H^n)\setminus\{0\}$
		such that
		\[
		Q_\lambda(\varphi):=\int_{\mathbb H^n} (\widetilde\cP_s \varphi)\,\varphi\,dV_{\mathbb H^n}
		-\lambda\int_{\mathbb H^n}|\varphi|^2\,dV_{\mathbb H^n}<0.
		\]
		Set $q:=-Q_\lambda(\varphi)>0$ and $p:=2_s^*$.
		
		\noindent\emph{Step 1: many far-apart copies.}
		Let $K:=\operatorname{supp}\varphi$, which is compact. For each $N\in\mathbb N$ we choose
		isometries $\{\tau_j\}_{j=1}^N\subset \mathrm{Isom}(\mathbb H^n)$ such that the sets
		\[
		K_j:=\tau_j(K)
		\quad\text{are pairwise disjoint and satisfy}\quad
		\operatorname{dist}(K_i,K_j)\ge R_N\ \ \text{for all }i\neq j,
		\]
		where $R_N>0$ will be chosen later. Define $\varphi_j:=\varphi\circ\tau_j^{-1}$ and $u_N:=\sum_{j=1}^N \varphi_j\in C_c^\infty(\mathbb H^n).$
		Since isometries preserve $dV_{\mathbb H^n}$ and the supports are disjoint, we have $	\int_{\mathbb H^n}|u_N|^2\,dV_{\mathbb H^n}
		=\sum_{j=1}^N \int_{\mathbb H^n}|\varphi_j|^2\,dV_{\mathbb H^n}
		= N\int_{\mathbb H^n}|\varphi|^2\,dV_{\mathbb H^n},$
		and $\int_{\mathbb H^n}|u_N|^{p}\,dV_{\mathbb H^n}
		=\sum_{j=1}^N \int_{\mathbb H^n}|\varphi_j|^{p}\,dV_{\mathbb H^n}
		= N\int_{\mathbb H^n}|\varphi|^{p}\,dV_{\mathbb H^n}.$
		Hence the denominator satisfies
		\begin{equation}\label{eq:denom-growth}
			\Bigl(\int_{\mathbb H^n}|u_N|^p\,dV_{\mathbb H^n}\Bigr)^{\!2/p}
			= N^{2/p}\Bigl(\int_{\mathbb H^n}|\varphi|^p\,dV_{\mathbb H^n}\Bigr)^{\!2/p}
			= N^{\frac{n-2s}{n}}\Bigl(\int_{\mathbb H^n}|\varphi|^p\,dV_{\mathbb H^n}\Bigr)^{\!2/p}.
		\end{equation}
		
		\medskip
		\noindent\emph{Step 2: estimates of the quadratic form.}
		Write $u_N=\sum_{j=1}^N \varphi_j$. Expanding the quadratic form gives
		\begin{align*}
			Q_\lambda(u_N)
			&=\sum_{j=1}^N Q_\lambda(\varphi_j)
			\;+\;2\sum_{1\le i<j\le N}\int_{\mathbb H^n} (\widetilde\cP_s\varphi_i)\,\varphi_j\,dV_{\mathbb H^n}.
		\end{align*}
		By Lemma \ref{lem:isom-inv-Ps} and $dV_{\mathbb H^n}$, $Q_\lambda(\varphi_j)=Q_\lambda(\varphi)=-q$,
		thus $\sum_{j=1}^N Q_\lambda(\varphi_j)=-Nq$.
		
		By Proposition \ref{lem:offdiag-exp-decay-Ptilde}, there exist constants $\alpha>0$ and $C>0$ (depending on $n,s$ and $\varphi$) such that
		whenever $\operatorname{dist}(\operatorname{supp} f,\operatorname{supp}g)\ge R$,
		\begin{equation}\label{eq:interaction-bound}
			\biggl|\int_{\mathbb H^n} (\widetilde \cP_s f)\,g\,dV_{\mathbb H^n}\biggr|
			\le C\,e^{-\alpha R}.
		\end{equation}

		Applying \eqref{eq:interaction-bound} with $f=\varphi_i$, $g=\varphi_j$ and $R=R_N$ yields
		\[
		\biggl|\sum_{1\le i<j\le N}\int_{\mathbb H^n} (\widetilde\cP_s\varphi_i)\,\varphi_j\,dV_{\mathbb H^n}\biggr|
		\le C \binom{N}{2} e^{-\alpha R_N}\le C N^2 e^{-\alpha R_N}.
		\]
		Therefore
		\begin{equation}\label{eq:Qlambda-upper}
			Q_\lambda(u_N)\le -Nq + 2C N^2 e^{-\alpha R_N}.
		\end{equation}
		Choose $R_N:=\frac{2}{\alpha}\log N + R_0$ with $R_0>0$ so large that $2C e^{-\alpha R_0}\le \frac q4$.
		Then $2C N^2 e^{-\alpha R_N}\le \frac q4\,N$, and \eqref{eq:Qlambda-upper} becomes
		\begin{equation}\label{eq:Qlambda-negative-linear}
			Q_\lambda(u_N)\le -\frac{3q}{4}\,N\le -\frac q2\,N
			\qquad\text{for all sufficiently large }N.
		\end{equation}
		Combining \eqref{eq:denom-growth} and \eqref{eq:Qlambda-negative-linear}, for large $N$ we obtain
		\[\frac{Q_\lambda(u_N)}{\bigl(\int_{\mathbb H^n}|u_N|^p\,dV_{\mathbb H^n}\bigr)^{2/p}}
		\le
		-\frac{q}{2}\,
		\frac{N}{N^{2/p}}\,
		\Bigl(\int_{\mathbb H^n}|\varphi|^p\,dV_{\mathbb H^n}\Bigr)^{-2/p} \\
		=
		- kN^{1-\frac{2}{p}}
		=
		-kN^{\frac{2s}{n}}
		\to -\infty,\]
		where $k:=\dfrac{q}{2}\Bigl(\int_{\mathbb H^n}|\varphi|^p\,dV_{\mathbb H^n}\Bigr)^{-2/p}>0$.
		Taking the infimum over $u$ gives $\widetilde H_{n,s}(\lambda)=-\infty$.
		
		(ii) The same argument applies to $\cP_s$ as well, and we omit the proof.
	\end{proof}\bigskip
	
	We next record the basic comparison properties of the hyperbolic fractional Poincar\'e--Sobolev levels with respect to the Euclidean sharp constant \(S_{n,s}\).

	\begin{proposition}\label{prop:basic-props-Hns}
		Let $n\ge2$ and $s\in(0,\tfrac n2)$. Then for all $\lambda\in \mathbb{R}$,
		\[
		H_{n,s}(\lambda)\le S_{n,s} ;\quad  \widetilde H_{n,s}(\lambda)\le S_{n,s}
		\]
		where $S_{n,s},H_{n,s}(\lambda), \widetilde H_{n,s}(\lambda)$ are defined in  (\ref{eq:upper-bound}) (\ref{eq:Sns-Hn-lambda})(\ref{eq:Sns-Hn-lambda111}). In particular, for $\lambda\le 0,$ we have
		\[
		\widetilde H_{n,s}(\lambda)= \widetilde H_{n,s}(0)= S_{n,s}.
		\]
	\end{proposition}

	\begin{proof}
		Use the test functions $u_\varepsilon$ constructed in Section~3. By
		Lemma~\ref{lem:exact-reduction} and
		Propositions~\ref{lem:L2sstar-cutoff}, \ref{lem:L2-cutoff}, and
		\ref{lem:energy-asymptotics-tilde},
		\[
		\frac{\langle\widetilde\cP_su_\varepsilon,u_\varepsilon\rangle}
		{\|u_\varepsilon\|_{2_s^*}^2}
		\longrightarrow S_{n,s},
		\qquad
		\frac{\|u_\varepsilon\|_2^2}
		{\|u_\varepsilon\|_{2_s^*}^2}
		\longrightarrow0.
		\]
		Moreover, by \eqref{eq:P-vs-tildeP} and \eqref{rema},
		$\cP_s-\widetilde\cP_s$ is bounded on $L^2(\mathbb H^n)$. Hence
		\[
		\frac{
			\langle(\cP_s-\widetilde\cP_s)u_\varepsilon,u_\varepsilon\rangle
		}{
			\|u_\varepsilon\|_{2_s^*}^2
		}
		\longrightarrow0.
		\]
		Since $\lambda$ is fixed, $\lambda\,
		\frac{\|u_\varepsilon\|_2^2}
		{\|u_\varepsilon\|_{2_s^*}^2}
		\rightarrow0.$
		Therefore $H_{n,s}(\lambda)\le S_{n,s},\,
		\widetilde H_{n,s}(\lambda)\le S_{n,s}.$
		
		If $\lambda\le0$, then
		\[
		\langle\widetilde\cP_su,u\rangle
		-\lambda\|u\|_2^2
		\ge
		\langle\widetilde\cP_su,u\rangle,
		\]
		and Proposition~\ref{prop:conf-frac-Sob-Hn} gives $\widetilde H_{n,s}(\lambda)\ge S_{n,s}.$
		Thus
		\[
		\widetilde H_{n,s}(\lambda)
		=
		\widetilde H_{n,s}(0)
		=
		S_{n,s},
		\qquad \lambda\le0,
		\]
		which completes the proof.
	\end{proof}\medskip
	
	Finally, we prove Theorem~\ref{prop:strict-gap-fractional} and Theorem \ref{thm:strict-gap-Bplus-Bzero} by combining the strict attainment mechanism established above with the threshold behavior at the spectral bottom.\medskip

	\noindent \textbf{Proof of Theorem \ref{prop:strict-gap-fractional}:} 
	(i) By Theorem \ref{thm:Hardy-lowerbound-Ptilde} (i) and Proposition \ref{prop:basic-props-Hns}, we obtain $\mathcal{G}_{n,s}[\widetilde H_{n,s}]=(0,+\infty)$. By the proof of Theorem \ref{mainthe2}, \(\widetilde H_{n,s}(\lambda)\) is achieved for every
	\(\lambda\in(0,\widetilde\lambda_{0,s}^{\mathrm{conf}}]\), thus by Lemma \ref{prop:comparison-attainment1} (ii),  \(\widetilde H_{n,s}(\lambda)\) is strictly decreasing in \((0,\widetilde\lambda_{0,s}^{\mathrm{conf}}]\). The remaining conclusions follow from Proposition \ref{prop:Hnslambda-minus-infty}, Proposition \ref{prop:basic-props-Hns} and Lemma \ref{prop:bottom-spectrum-quadratic-form1}.
	
	(ii) By Theorem \ref{thm:Hardy-lowerbound-Ptilde} (ii), there exists $\widetilde\lambda_{s}^{\mathrm{conf}}\in(0,\widetilde{\lambda}_{0,s}^{\mathrm{conf}}]$ such that $\bigl(\widetilde\lambda_{s}^{\mathrm{conf}},\infty\bigr)
	\subset
	\mathcal{G}_{n,s}\bigl[\widetilde H_{n,s}\bigr]$, the remaining conclusions follow from Proposition \ref{prop:Hnslambda-minus-infty}, Proposition \ref{prop:basic-props-Hns} and Lemma \ref{prop:bottom-spectrum-quadratic-form1}.
	
	(iii) By \cite[Theorem 1.9]{lu2023explicit}, we obtain $\widetilde H_{n,s}\!\bigl(\widetilde{\lambda}_{0,s}^{\mathrm{conf}}\bigr)
	=S_{n,s}.$ By Proposition \ref{prop:basic-props-Hns}, $\widetilde H_{n,s}\!\bigl(0\bigr)
	=S_{n,s},$ which completes the proof.\hfill$\Box$\bigskip

	\noindent \textbf{Proof of Theorem \ref{thm:strict-gap-Bplus-Bzero}:} 
	(i) For $s\in \left(0,\frac{n}{4}\right]\cap\Bzero$, we have $b_s=0.$ By Theorem \ref{prop:gap-bs-lambda011} (b), we obtain $H_{n,s}(\lambda)<S_{n,s} \ \, {\rm for}\ \lambda>0,$ the remaining conclusions follow from Proposition \ref{prop:basic-props-Hns}.
	
	(ii)  For $s\in \left(0,\frac{n}{4}\right]\cap\Bplus$, we have $\lambda_{0,s}^{\mathrm{conf}}>b_s>0$. By Proposition \ref{prop:basic-props-Hns} and (\ref{eq:P-vs-tildeP}) we know $S_{n,s}=\widetilde H_{n,s}(0)\le H_{n,s}(0)\le S_{n,s},$ thus $H_{n,s}(0)= S_{n,s}.$ Then by Theorem \ref{prop:gap-bs-lambda011} (b), we have $(b_s,\infty)\subset\mathcal{G}_{n,s}\bigl[ H_{n,s}\bigr].$
	By the proof of Theorem \ref{mainthe} and Lemma \ref{prop:comparison-attainment1} (ii), \(H_{n,s}(\lambda)\) is achieved and strictly decreasing for every
	\(\lambda\in(b_s,\lambda_{0,s}^{\mathrm{conf}}]\). 
	
	(iii) follows immediately by combining Theorem~\ref{prop:gap-bs-lambda011} (b) with Proposition~\ref{prop:basic-props-Hns}.
	
	(iv) By \cite[Theorem 1.4]{lu2023explicit} and Proposition~\ref{prop:basic-props-Hns}, for $\lambda\le{\lambda}_{0,s}^{\mathrm{conf}}$, we obtain $H_{n,s}\!\bigl({\lambda}_{0,s}^{\mathrm{conf}}\bigr)
	=S_{n,s}= H_{n,s}(\lambda).$ The remaining conclusions follow from Proposition \ref{prop:Hnslambda-minus-infty}.\hfill$\Box$\bigskip

	\section{Appendix}
	\label{appendix}
	
	In this appendix, we collect several properties of Sobolev levels in Euclidean  settings. Some of these results are classical, while others are reproved here from a perspective motivated by our observations. Building on these properties, we provide the proofs of Propositions~\ref{prop:Lp-perturbation-frac-prop} and~\ref{prop:Lp-perturbation-frac-prophn11}.

	\subsection{Stability and Attainability of Sobolev Levels}
	
	We first show that adding a positive lower-order $L^p$ perturbation term does not change the optimal Sobolev level.

	\begin{proposition}\label{lem:Lp-perturbation}
		Let $\Omega\subset\mathbb{R}^n$ be a nonempty open set, $n\ge3$ and $2^*=\frac{2n}{n-2}$.
		For $p\in(0,\infty)$ and $\mu>0$ define
		\[
		S^{(p)}_{\Omega}
		:= \inf_{v\in C_c^\infty(\Omega)\setminus\{0\}}
		\frac{\displaystyle\int_{\mathbb{R}^n}\bigl(|\nabla v|^2+\mu |v|^p\bigr)\,dx}
		{\Bigl(\displaystyle\int_{\mathbb{R}^n}|v|^{2^*}\,dx\Bigr)^{2/2^*}},\quad 	S_{\Omega}
		:= \inf_{v\in C_c^\infty(\Omega)\setminus\{0\}}
		\frac{\displaystyle\int_{\mathbb{R}^n}|\nabla v|^2\,dx}
		{\Bigl(\displaystyle\int_{\mathbb{R}^n}|v|^{2^*}\,dx\Bigr)^{2/2^*}}.
		\]
		Then $S^{(p)}_{\Omega} = S_{\Omega}=S_{n,1},$
		where $S_{n,1}$ is defined in (\ref{eq:upper-bound}).
	\end{proposition}
	
	\begin{proof}
		For $v\in C_c^\infty(\Omega)\setminus\{0\},$ set
		\[
		S_{\Omega}(v)
		:= \frac{\displaystyle\int_{\mathbb{R}^n}|\nabla v|^2\,dx}
		{\Bigl(\displaystyle\int_{\mathbb{R}^n}|v|^{2^*}\,dx\Bigr)^{2/2^*}},
		\qquad
		S^{(p)}_{\Omega}(v)
		:= \frac{\displaystyle\int_{\mathbb{R}^n}\bigl(|\nabla v|^2+\mu |v|^p\bigr)\,dx}
		{\Bigl(\displaystyle\int_{\mathbb{R}^n}|v|^{2^*}\,dx\Bigr)^{2/2^*}}.
		\]
		It is obvious that $S^{(p)}_{\Omega}=\inf S^{(p)}_{\Omega}(v)
		\;\ge\;\inf S_{\Omega}(v)=S_{\Omega}.$
		Next, we show that $S^{(p)}_{\Omega}\le S_{\Omega}$
		We split the argument into the cases $p\neq2$ and $p=2$.
		
		\smallskip\noindent
		\textbf{Case A: $p\neq2$.}
		Fix $u\in C_c^\infty(\Omega)\setminus\{0\}$ and set
		\[
		A:=\int_{\mathbb{R}^n}|\nabla u|^2\,dx,\quad
		B:=\int_{\mathbb{R}^n}|u|^p\,dx,\quad
		C:=\Bigl(\int_{\mathbb{R}^n}|u|^{2^*}\,dx\Bigr)^{2/2^*}>0.
		\]
		For $t>0$ consider the amplitude scaling $u_t:=t\,u$. Then $u_t\in C_c^\infty(\Omega)$ and
		\[
		\int_{\mathbb{R}^n}|\nabla u_t|^2\,dx = t^2A,\quad
		\int_{\mathbb{R}^n}|u_t|^p\,dx = t^pB,\quad
		\Bigl(\int_{\mathbb{R}^n}|u_t|^{2^*}\,dx\Bigr)^{2/2^*} = t^2C.
		\]
		Therefore
		\[
		S^{(p)}_{\Omega}(u_t)
		= \frac{t^2A+\mu t^pB}{t^2C}
		= S_{\Omega}(u) + \mu\,t^{p-2}\,\frac{B}{C}.
		\]
		If $p>2$, let $t\to0^+$; then $t^{p-2}\to0$ and $\lim_{t\to0^+}S^{(p)}_{\Omega}(u_t)=S_{\Omega}(u).$
		If $0<p<2$, let $t\to+\infty$; then $t^{p-2}\to0$ and the same limit holds. In both cases we obtain $S^{(p)}_{\Omega}\le \inf_{t>0}S^{(p)}_{\Omega}(u_t)\le S_{\Omega}(u).$
		Since $u$ is arbitrary, taking the infimum over $u$ yields $	S^{(p)}_{\Omega}\le S_{\Omega}.$
		
		\smallskip\noindent
		\textbf{Case B: $p=2$.}
		It is well known that $S_{\Omega}$ coincides
		with the Sobolev constant on any ball contained in $\Omega$. Choose $x_0\in\Omega$
		and $r>0$ such that $B_r(x_0)\subset\Omega$, and let
		\[
		S_{\Omega}=S_{B_r(x_0)}=\inf_{v\in C_c^\infty(B_r(x_0))\setminus\{0\}}S_{B_r(x_0)}(v).
		\]
		Hence, for any given $\varepsilon>0$, there exists $u\in C_c^\infty(B_r(x_0))\setminus\{0\}$ such that $	S_{B_r(x_0)}(u)\le S_{\Omega}+\varepsilon.$
		For $\lambda\ge1$, define the critical Sobolev scaling around $x_0$,
		\[
		u_\lambda(x):=\lambda^{\frac{n-2}{2}}u\bigl(x_0+\lambda(x-x_0)\bigr).
		\]
		Then $u_\lambda\in C_c^\infty(B_{r/\lambda}(x_0))$, with $B_{r/\lambda}(x_0)\subset B_r(x_0)\subset\Omega$, and a direct
		change-of-variables computation gives
		\begin{align*}
			\int_{\mathbb{R}^n}|\nabla u_\lambda|^2\,dx
			= \int_{\mathbb{R}^n}|\nabla u|^2\,dx,\quad
			\int_{\mathbb{R}^n}|u_\lambda|^{2^*}\,dx= \int_{\mathbb{R}^n}|u|^{2^*}\,dx
		\end{align*}
		and $\int_{\mathbb{R}^n}|u_\lambda|^{2}\,dx=\lambda^{-2}\int_{\mathbb{R}^n}|u|^{2}\,dx.$
		Therefore,
		\[
		S^{(2)}_{\Omega}(u_\lambda)
		= \frac{\displaystyle\int_{\mathbb{R}^n}|\nabla u|^2\,dx
			+ \mu\lambda^{-2}\displaystyle\int_{\mathbb{R}^n}|u|^2\,dx}
		{\Bigl(\displaystyle\int_{\mathbb{R}^n}|u|^{2^*}\,dx\Bigr)^{2/2^*}}
		= S_{B_r(x_0)}(u)
		+ \mu\lambda^{-2}\,
		\frac{\displaystyle\int_{\mathbb{R}^n}|u|^2\,dx}
		{\Bigl(\displaystyle\int_{\mathbb{R}^n}|u|^{2^*}\,dx\Bigr)^{2/2^*}}.
		\]
		Letting $\lambda\to\infty$ we obtain $\lim_{\lambda\to\infty}S^{(2)}_{\Omega}(u_\lambda)=S_{B_r(x_0)}(u)\le S_{\Omega}+\varepsilon.$
		Hence, $S^{(2)}_{\Omega}\le S_{\Omega}+\varepsilon$
		for every $\varepsilon>0$, which implies $S^{(2)}_{\Omega}\le S_{\Omega}$.
	\end{proof}\medskip
	
	Analogously, in the fractional setting, a positive lower-order $L^p$ perturbation still leaves the sharp Sobolev level unchanged.

	\begin{proposition}
		\label{lem:Lp-perturbation-frac}
		Let $\Omega\subset\mathbb{R}^n$ be a nonempty open set, $n>2s,s\in (0,1)$ and $2_s^*=\frac{2n}{n-2s}$. 
		For $p\in(0,\infty)$ and $\mu>0$ define
		\[
		S^{(p)}_{\Omega,s}
		:= \inf_{u\in C_c^\infty(\Omega)\setminus\{0\}}
		\frac{\displaystyle\int_{\Omega} u\,(-\Delta)^s u\,dx
			+\mu\displaystyle\int_{\mathbb{R}^n}|u|^p\,dx}
		{\bigl(\displaystyle\int_{\mathbb{R}^n}|u|^{2_s^*}\,dx\bigr)^{2/2_s^*}}, \quad	S_{\Omega,s}
		:= \inf_{u\in C_c^\infty(\Omega)\setminus\{0\}}
		\frac{\displaystyle\int_{\Omega} u\,(-\Delta)^s u\,dx}
		{\bigl(\displaystyle\int_{\mathbb{R}^n}|u|^{2_s^*}\,dx\bigr)^{2/2_s^*}}.
		\]
		Then $S^{(p)}_{\Omega,s} = S_{\Omega,s}	=S_{n,s}$ for all $p>0,$
		where $S_{n,s}$ is defined in (\ref{eq:upper-bound}).
	\end{proposition}
	
	\begin{proof}
		The argument is entirely analogous to the local case in Proposition~\ref{lem:Lp-perturbation}:
		for $p\ne2$, use the amplitude scaling $u\mapsto tu$ as in Case A. For $p=2$,
		one employs the critical dilation $u_t(x):=t^{\frac{n-2s}{2}}\,u\bigl(x_0+t(x-x_0)\bigr),\,t\ge 1,$
		which preserves both the quadratic energy and the critical normalization,
		\[
		\int_{\R^n} u_t\,(-\Delta)^s u_t\,dx=\int_{\R^n} u\,(-\Delta)^s u\,dx,
		\qquad
		\Bigl(\int_{\R^n}|u_t|^{2_s^*}\,dx\Bigr)^{\!\frac{2}{2_s^*}}
		=\Bigl(\int_{\R^n}|u|^{2_s^*}\,dx\Bigr)^{\!\frac{2}{2_s^*}}.
		\]
		Thus, we omit the details.
	\end{proof}\medskip

	Since $\lambda_{0,s}^{\mathrm{conf}}$ and $\widetilde{\lambda}_{0,s}^{\mathrm{conf}}$ are the spectral bottoms of $\cP_s$ and $\widetilde{\cP}_s$, respectively, the following two results are immediate. Their proofs are straightforward and therefore omitted.

	\begin{lemma}
		\label{prop:bottom-spectrum-quadratic-form}
		Let $n\ge2$ and  $s\in\bigl(0,\frac n2\bigr)$.
		Then, for $\lambda\in\mathbb R$,
		\[	\int_{\mathbb H^n} (\cP_s u)\,u\,dV_{\mathbb H^n}-\lambda\int_{\mathbb H^n}|u|^2\,dV_{\mathbb H^n}\, \ge 0
		\qquad \forall\,u\in C_c^\infty(\mathbb H^n)\]
		holds if and only if $\lambda\le \lambda_{0,s}^{\mathrm{conf}}$. In particular, $H_{n,s}(\lambda)<0$ iff $\lambda>\lambda_{0,s}^{\mathrm{conf}}.$
	\end{lemma}
	
	\begin{lemma}
		\label{prop:bottom-spectrum-quadratic-form1}
		Let $n\ge2$ and  $s\in\bigl(0,\frac n2\bigr)$.
		Then, for $\lambda\in\mathbb R$,
		\[	\int_{\mathbb H^n} (\widetilde \cP_s u)\,u\,dV_{\mathbb H^n}-\lambda\int_{\mathbb H^n}|u|^2\,dV_{\mathbb H^n}\, \ge 0
		\qquad \forall\,u\in C_c^\infty(\mathbb H^n)\]
		holds if and only if $\lambda\le \widetilde\lambda_{0,s}^{\mathrm{conf}}$. In particular, $\widetilde H_{n,s}(\lambda)<0$ iff $\lambda>\widetilde \lambda_{0,s}^{\mathrm{conf}}.$
	\end{lemma}
	
	We next present a very useful abstract lemma, which links attainment of Sobolev-type levels to a strict comparison (strict monotonicity) of Sobolev-type levels.

	\begin{lemma}\label{prop:comparison-attainment1}
		Let $X$ be a nontrivial function space and fix a parameter $\lambda$.
		Let $f_\lambda,\,g_\lambda : X\setminus\{0\} \to \mathbb{R}$
		be two functionals such that
		\begin{equation}\label{eq:pointwise-strict}
			f_\lambda(u) < g_\lambda(u)
			\qquad\text{for all }u\in X\setminus\{0\}.
		\end{equation}
		Define
		\[
		f(\lambda) := \inf_{u\in X\setminus\{0\}} f_\lambda(u),
		\qquad
		g(\lambda) := \inf_{u\in X\setminus\{0\}} g_\lambda(u).
		\]
		Then the following properties hold:
		
		(i) If $f(\lambda)=g(\lambda)$, then $g(\lambda)$ is not attained.
		
		(ii) Conversely, if $g(\lambda)$ is attained at some
		$u_0\in X\setminus\{0\}$, then $f(\lambda)< g(\lambda).$
	\end{lemma}
	
	\begin{proof}
		Since $f_\lambda(u) < g_\lambda(u)$ for all $u\in X\setminus\{0\}$, we clearly have
		\[
		f(\lambda)
		= \inf_{u\in X\setminus\{0\}} f_\lambda(u)
		\;\le\; \inf_{u\in X\setminus\{0\}} g_\lambda(u)
		= g(\lambda).
		\]
		
		(i) Assume by contradiction that $f(\lambda)=g(\lambda)$ and that $g(\lambda)$
		is attained, i.e., there exists $u_0\in X\setminus\{0\}$ such that $g_\lambda(u_0)=g(\lambda).$
		Then, by the strict pointwise inequality \eqref{eq:pointwise-strict}, $	f_\lambda(u_0) < g_\lambda(u_0)=g(\lambda)=f(\lambda),$
		which contradicts the definition of $f(\lambda)$ as the infimum of $f_\lambda$.
		Hence $g(\lambda)$ cannot be attained in $X\setminus\{0\}$.
		
		(ii) Conversely, assume that $g(\lambda)$ is attained at some $u_0\in X\setminus\{0\}$,
		so that $	g_\lambda(u_0)=g(\lambda).$
		By \eqref{eq:pointwise-strict} we have $f_\lambda(u_0) < g_\lambda(u_0)=g(\lambda).$
		Using the definition of $f(\lambda)$, $	f(\lambda)
		= \inf_{u\in X\setminus\{0\}} f_\lambda(u)
		\;\le\; f_\lambda(u_0)
		< g(\lambda).$
		Therefore $f(\lambda)<g(\lambda)$, which proves the second claim.
	\end{proof}

	\subsection{Proof of Propositions   \ref{prop:Lp-perturbation-frac-prop} and \ref{prop:Lp-perturbation-frac-prophn11}}
	
	Using the preceding results together with classical results on the Brezis--Nirenberg problem in Euclidean space, we now prove Propositions~\ref{prop:Lp-perturbation-frac-prop} and~\ref{prop:Lp-perturbation-frac-prophn11}.\medskip

	\noindent \textbf{Proof of Proposition \ref{prop:Lp-perturbation-frac-prop}. } 
	
	\textbf{ Part I: $s=1$.}
	
	(i) By Proposition~\ref{lem:Lp-perturbation} we have, for all $\lambda\le 0$, $S_{n,1,\Omega
	}(\lambda)=S_{n,1}.$
	Next, we show that $S_{n,1,\Omega}(\lambda_{1,1}\left(\Omega\right))=0.$ Let $\phi_1$ be the first Dirichlet eigenfunction,
	$-\Delta\phi_1=\lambda_{1,1}(\Omega)\phi_1$ in $\Omega$,
	$\phi_1\in H_0^1(\Omega)\setminus\{0\}$. Then
	\[
	0\le S_{n,1,\Omega}(\lambda)
	\;\le\;
	\frac{\displaystyle\int_\Omega|\nabla\phi_1|^2\,dx
		-\lambda\int_\Omega\phi_1^2\,dx}
	{\|\phi_1\|_{L^{2^*}(\Omega)}^2}
	=(\lambda_{1,1}(\Omega)-\lambda)\,
	\frac{\displaystyle\int_\Omega\phi_1^2\,dx}
	{\|\phi_1\|_{L^{2^*}(\Omega)}^2},
	\]
	where $\lambda\le\lambda_{1,1}(\Omega)$; setting $\lambda=\lambda_{1,1}(\Omega)$ gives the desired equality. Fix $v\in C_c^\infty(\Omega)\setminus\{0\}$, by H\"older's inequality, we get $	\int_\Omega |v|^2\,dx
	\le |\Omega|^{2/n}\Bigl(\int_\Omega |v|^{2^*}\,dx\Bigr)^{\frac{2}{2^*}}.$
	Therefore,
	\[
	\frac{\displaystyle\int_{\Omega} |\nabla v|^2\,dx-\lambda\int_{\Omega} v^2\,dx}
	{\Bigl(\displaystyle\int_{\Omega}|v|^{2^*}\,dx\Bigr)^{2/2^*}}
	\ge -\lambda\,|\Omega|^{2/n}\quad \text{for} \quad \lambda>\lambda_{1,1}(\Omega
	).
	\]
	Taking the infimum over all $v\in C_c^\infty(\Omega)\setminus\{0\}$ gives
	$S_{n,1,\Omega}(\lambda)\ge -\lambda|\Omega|^{2/n}>-\infty$.

	By \cite[Lemma~1.1 and Lemma~1.2]{brezis1983positive} together with
	Lemma~\ref{prop:comparison-attainment1} (ii), then for any $0<\mu_1<\mu_2<\lambda_{1,1}(\Omega)<\mu_3,$
	one has the strict chain
	\[
	S_{n,1,\Omega}(\mu_3)\;<\;0
	\;=\;S_{n,1,\Omega}(\lambda_{1,1}\left(\Omega\right))
	\;<\;S_{n,1,\Omega}(\mu_2)
	\;<\;S_{n,1,\Omega}(\mu_1)
	\;<\;S_{n,1}.
	\]
	Combining with $S_{n,1,\Omega
	}(\lambda)=S_{n,1}$ for $\lambda\le0$ yields $\mathcal{G}_{n,1,\Omega
	}[S_{n,1,\Omega
	}]=(0,\infty).$
	
	It remains to characterize attainability.  
	If $\lambda>0$, by \cite[Lemma~1.1 and Lemma~1.2]{brezis1983positive}, the infimum is attained. For $\lambda=0$, it is well known that $S_{n,1,\Omega}(0)=S_{n,1}$ is not attained on bounded domains.
	
	If $\lambda<0$, suppose by contradiction that $S_{n,1,\Omega
	}(\lambda)$ is attained.
	Then Lemma~\ref{prop:comparison-attainment1} (ii) implies the strict monotonicity
	\(
	S_{n,1,\Omega
	}(\lambda)>S_{n,1,\Omega}(\lambda/2),
	\)
	contradicting Proposition~\ref{lem:Lp-perturbation}, which gives
	\(S_{n,1,\Omega}(\lambda)=S_{n,1,\Omega}(\lambda/2)=S_{n,1}\) for all $\lambda\le0$.
	Therefore $S_{n,1,\Omega}(\lambda)$ is attained  if and only if
	$\lambda\in\mathcal{G}_{n,1,\Omega}[S_{n,1,\Omega}]=(0,\infty)$. 
	
	(iii) By \cite[Lemmas~1.2--1.3]{brezis1983positive}, Lemma~\ref{prop:comparison-attainment1}(ii),
	and Proposition~\ref{lem:Lp-perturbation}, for any choice of parameters
	\[
	\lambda\le 0<\frac{\lambda_{1,1}(B_1)}{4}<\mu_1<\mu_2<\lambda_{1,1}(B_1)<\mu_3,
	\]
	we have the chain of strict inequalities
	\[
	\begin{gathered}
		S_{3,1,B_1}(\mu_3)<0
		=S_{3,1,B_1}\bigl(\lambda_{1,1}(B_1)\bigr)
		<S_{3,1,B_1}(\mu_2)
		<S_{3,1,B_1}(\mu_1)
		\\\le S_{3,1,B_1}\!\Bigl(\frac{\lambda_{1,1}(B_1)}{4}\Bigr)
		\le S_{3,1}
		=S_{3,1,B_1}(\lambda).
	\end{gathered}
	\]
	We claim that $S_{3,1,B_1}\!\Bigl(\frac{\lambda_{1,1}(B_1)}{4}\Bigr)=S_{3,1}.$
	Indeed, if this were false, then \cite[Lemma~1.2]{brezis1983positive} would imply that
	$S_{3,1,B_1}\!\bigl(\frac{\lambda_{1,1}(B_1)}{4}\bigr)$ is attained, contradicting
	\cite[Lemma~1.4]{brezis1983positive}. Consequently,
	\cite[Lemma~1.3]{brezis1983positive} yields $	S_{3,1,B_1}(\mu_1)\;<\;S_{3,1,B_1}\!\Bigl(\frac{\lambda_{1,1}(B_1)}{4}\Bigr).$
	By an analogous argument in (i), we conclude that
	$S_{3,1,B_1}(\lambda)$ is attained if and only if $\lambda\in\mathcal{G}_{3,1,B_1}[S_{3,1,B_1}]$.\smallskip
	
	\textbf{ Part II: $s\in(0,1)$:}
	
	(i) Using \cite[proof of Claim~14.1]{BisciRadulescuServadei2016}, Proposition~\ref{lem:Lp-perturbation-frac}, Lemma~\ref{prop:comparison-attainment1}, the proof is completely analogous to the case $s=1$, and we therefore omit the details.
	
	(ii) By \cite[Proposition  16.4]{BisciRadulescuServadei2016}, we obtain $(\lambda_s^*,\infty)\subset\mathcal{G}_{n,s,\Omega}[S_{n,s,\Omega}]$ and thus $S_{n,s,\Omega}(\lambda)$ is attained  if $\lambda \in (\lambda_s^*,\infty)$.
	By Proposition~\ref{lem:Lp-perturbation-frac}, Lemma~\ref{prop:comparison-attainment1}, for any $\lambda\le 0<\lambda_{s}^*<\mu_1<\mu_2$,
	\[
	S_{n,s,\Omega}(\mu_2)\;<\;S_{n,s,\Omega}(\mu_1)\;\le\;S_{n,s,\Omega}(\lambda_{s}^*)
	\;\le\;S_{n,s,\Omega}(0)\;=\;S_{n,s,\Omega}(\lambda).
	\]
	The proof ends. \hfill$\Box$\bigskip

	\noindent\textbf{Proof of Proposition \ref{prop:Lp-perturbation-frac-prophn11}.}  When $k=1$, the proof is the following.  \smallskip
	
	(i) By Proposition \ref{prop:basic-props-Hns} and Proposition \ref{prop:conf-frac-Sob-Hn}, we know that for any $\lambda \le 0,$ 
	\[H_{n,1}(\lambda)=H_{n,1}(0)=S_{n,1}.\] By the proof of \cite[Theorem 1.5, Theorem 1.6]{ManciniSandeep2008}, we know that $H_{n,1}(\lambda)$ is achieved if and only if  $\lambda\in (0,\lambda_{0,1}^{\mathrm{conf}}].$ Thus, by Lemma~\ref{prop:comparison-attainment1}(ii), for any $\lambda\le 0<\mu_1<\mu_2<\frac{1}{4},$ we have
	\[H_{n,1}(\frac{1}{4})<H_{n,1}(\mu_2)<H_{n,1}(\mu_1)<H_{n,1}(\lambda)=H_{n,1}(0)=S_{n,1}.\] 
	By the Poincar\'e--Sobolev inequality (see \cite{ManciniSandeep2008}) and Lemma \ref{prop:bottom-spectrum-quadratic-form}, we obtain $H_{n,1}(\frac{1}{4})>0>H_{n,1}(\mu_3)$ for $\mu_3>\frac14$.
	Thus, by Lemma~\ref{prop:comparison-attainment1} (ii) we obtain the strict decreasing property  and $\mathcal{G}_{n,1}[	H_{n,1}]=(0,\infty).$ By Proposition \ref{prop:Hnslambda-minus-infty}, we obtain $H_{n,1}(\mu)=-\infty$ for $\mu>\lambda_{0,1}^{\mathrm{conf}}.$
	
	(ii) For dimension $n=3$, by \cite[Theorem 1.1]{BenguriaFrankLoss2007}, we have $H_{3,1}(\frac{1}{4})=S_{3,1}.$ Thus, by Proposition \ref{prop:basic-props-Hns}, we obtain for any $\mu_1\le \frac{1}{4},$ $H_{3,1}(\mu_1)=H_{3,1}(\frac{1}{4})=S_{3,1}.$
	Again by the  Poincar\'e--Sobolev inequality and Lemma \ref{prop:bottom-spectrum-quadratic-form}, we obtain for any $\mu_2>\frac{1}{4},$ $H_{3,1}(\frac{1}{4})>0>H_{3,1}(\mu_2).$ Thus, we prove the inequality and $\mathcal{G}_{3,1}[	H_{3,1}]=(\frac{1}{4},\infty).$ 
	By  \cite[Theorem 1.7]{ManciniSandeep2008}, we obtain $H_{3,1}(\lambda)$ is never achieved for any $\lambda\in \mathbb{R}$.
	
	\smallskip
	
	Now we deal with the case $k\geq 2$.  \smallskip
	
	(i) By \cite[Theorem~1.14]{li2022higher} and \cite[Theorem~1.7]{lu2022green}, we know that
	$H_{n,k}(\lambda)<S_{n,k}$ and  $H_{n,k}(\lambda)$ is achieved whenever $\lambda\in (0,\lambda_{0,k}^{\mathrm{conf}})$ when $n\ge 4k$.
	By arguments completely analogous to those used in the preceding propositions in high dimensions, we obtain
	the desired conclusion here, and therefore omit the details.
	
	(ii) By \cite[Theorem~1.14]{li2022higher} and \cite[Theorem~1.7]{lu2022green}, there exists $\lambda_k^{\mathrm{conf}}\in (0,\lambda_{0,k}^{\mathrm{conf}})$ such that
	$H_{n,k}(\lambda)<S_{n,k}$ when $\lambda>\lambda_k^{\mathrm{conf}}$ and $H_{n,k}(\lambda)$ is achieved whenever $\lambda\in (\lambda_{k}^{\mathrm{conf}},\lambda_{0,k}^{\mathrm{conf}})$.
	The remainder of the proof follows by an entirely analogous argument, combined with
	Lemma~\ref{prop:comparison-attainment1} (ii) and Proposition \ref{prop:Hnslambda-minus-infty}, and we therefore omit the details.

	(iii) When $n=2k+1$, by \cite[Theorem~1.4]{lu2023explicit} or \cite[Theorem~1.2]{lu2022green}, we obtain $H_{n,k}(\lambda_{0,k}^{\mathrm{conf}})=S_{n,k}.$ The rest of the proof is completely analogous to (ii), and we omit it.   \hfill$\Box$
	
	\bigskip\bigskip
	
	\noindent {\bf  Conflicts of interest:} The authors declare that they have no conflicts of interest regarding this work.
	\medskip
	
	\noindent {\bf  Data availability:}  This paper has no associated data.\medskip
	
	\noindent{{\bf Acknowledgements:}    
		H. Chen is supported by  NSFC, no. 12361043. \\
		R. Chen is supported by China Scholarship Council,  Liujinxuan [2025] no. 37.}

	\begingroup
	\setlength{\emergencystretch}{2em}
	\printbibliography

@misc{TaylorPDO,
author = {Taylor, M. E.},
title  = {Short Course on Pseudodifferential Operators},
year   = {2022},
note   = {Lecture notes, revised version of the 2008 MSRI course},
url    = {https://mtaylor.web.unc.edu/wp-content/uploads/sites/16915/2022/05/psidolect.pdf}
}

@article{brezis1983positive,
author  = {Br{\'e}zis, H. and Nirenberg, L.},
title   = {Positive solutions of nonlinear elliptic equations involving critical {S}obolev exponents},
journal = {Comm. Pure Appl. Math.},
volume  = {36},
number  = {4},
pages   = {437--477},
year    = {1983}
}

@article{mazzeo1987meromorphic,
author  = {Mazzeo, R. R. and Melrose, R. B.},
title   = {Meromorphic extension of the resolvent on complete spaces with asymptotically constant negative curvature},
journal = {J. Funct. Anal.},
volume  = {75},
number  = {2},
pages   = {260--310},
year    = {1987}
}

@article{graham1992conformally,
author  = {Graham, C. R. and Jenne, R. and Mason, L. J. and Sparling, G. A. J.},
title   = {Conformally invariant powers of the {L}aplacian, {I}: Existence},
journal = {J. Lond. Math. Soc. (2)},
volume  = {46},
number  = {3},
pages   = {557--565},
year    = {1992}
}

@article{shubin1992spectral,
	author  = {Shubin, M. A.},
	title   = {Spectral theory of elliptic operators on non-compact manifolds},
	journal = {Ast{\'e}risque},
	number  = {207},
	pages   = {35--108},
	year    = {1992}
}

@book{hebey1999nonlinear,
author    = {Hebey, E.},
title     = {Nonlinear Analysis on Manifolds: Sobolev Spaces and Inequalities},
series    = {Courant Lecture Notes in Mathematics},
volume    = {5},
publisher = {American Mathematical Society},
address   = {Providence, RI},
year      = {2000}
}

@article{joshi2000inverse,
author  = {Joshi, M. S. and S{\'a} Barreto, A.},
title   = {Inverse scattering on asymptotically hyperbolic manifolds},
journal = {Acta Math.},
volume  = {184},
number  = {1},
pages   = {41--86},
year    = {2000}
}

@incollection{stapelkamp2002brezis,
author    = {Stapelkamp, S.},
title     = {The {B}r{\'e}zis--{N}irenberg problem on $\mathbb{H}^n$: Existence and uniqueness of solutions},
booktitle = {Elliptic and Parabolic Problems},
pages     = {283--290},
publisher = {World Scientific},
address   = {River Edge, NJ},
year      = {2002}
}

@book{GelfandGindikinGraev2003,
author    = {Gelfand, I. M. and Gindikin, S. G. and Graev, M. I.},
title     = {Selected Topics in Integral Geometry},
series    = {Translations of Mathematical Monographs},
volume    = {220},
publisher = {American Mathematical Society},
address   = {Providence, RI},
year      = {2003}
}

@article{graham2003scattering,
author  = {Graham, C. R. and Zworski, M.},
title   = {Scattering matrix in conformal geometry},
journal = {Invent. Math.},
volume  = {152},
number  = {1},
pages   = {89--118},
year    = {2003}
}

@article{cotsiolis2004best,
author  = {Cotsiolis, A. and Tavoularis, N. K.},
title   = {Best constants for {S}obolev inequalities for higher order fractional derivatives},
journal = {J. Math. Anal. Appl.},
volume  = {295},
number  = {1},
pages   = {225--236},
year    = {2004}
}

@article{gover2006laplacian,
author  = {Gover, A. R.},
title   = {Laplacian operators and $Q$-curvature on conformally {E}instein manifolds},
journal = {Math. Ann.},
volume  = {336},
number  = {2},
pages   = {311--334},
year    = {2006}
}

@article{BenguriaFrankLoss2007,
author  = {Benguria, R. D. and Frank, R. L. and Loss, M.},
title   = {The sharp constant in the {H}ardy--{S}obolev--{M}az'ya inequality in the three-dimensional upper half-space},
journal = {Math. Res. Lett.},
volume  = {15},
number  = {4},
pages   = {613--622},
year    = {2008}
}

@article{ManciniSandeep2008,
author  = {Mancini, G. and Sandeep, K.},
title   = {On a semilinear elliptic equation in $\mathbb{H}^n$},
journal = {Ann. Sc. Norm. Super. Pisa Cl. Sci. (5)},
volume  = {7},
number  = {4},
pages   = {635--671},
year    = {2008}
}

@article{LiuPeng2009,
author  = {Liu, C. and Peng, L.},
title   = {Generalized {H}elgason--{F}ourier transforms associated to variants of the {L}aplace--{B}eltrami operators on the unit ball in $\mathbb{R}^n$},
journal = {Indiana Univ. Math. J.},
volume  = {58},
number  = {3},
pages   = {1457--1492},
year    = {2009}
}

@article{chang2011fractional,
author  = {Chang, S.-Y. A. and Gonz{\'a}lez, M. M.},
title   = {Fractional {L}aplacian in conformal geometry},
journal = {Adv. Math.},
volume  = {226},
number  = {2},
pages   = {1410--1432},
year    = {2011}
}

@article{fefferman2013juhl,
author  = {Fefferman, C. and Graham, C. R.},
title   = {Juhl's formulae for {GJMS} operators and $Q$-curvatures},
journal = {J. Amer. Math. Soc.},
volume  = {26},
number  = {4},
pages   = {1191--1207},
year    = {2013}
}

@article{juhl2013explicit,
author  = {Juhl, A.},
title   = {Explicit formulas for {GJMS}-operators and $Q$-curvatures},
journal = {Geom. Funct. Anal.},
volume  = {23},
number  = {4},
pages   = {1278--1370},
year    = {2013}
}

@article{liu2013sharp,
author  = {Liu, G.},
title   = {Sharp higher-order {S}obolev inequalities in the hyperbolic space $\mathbb{H}^n$},
journal = {Calc. Var. Partial Differential Equations},
volume  = {47},
number  = {3--4},
pages   = {567--588},
year    = {2013}
}

@article{gonzalez2013fractional,
author  = {Gonz{\'a}lez, M. M. and Qing, J.},
title   = {Fractional conformal {L}aplacians and fractional {Y}amabe problems},
journal = {Anal. PDE},
volume  = {6},
number  = {7},
pages   = {1535--1576},
year    = {2013}
}

@book{BisciRadulescuServadei2016,
author    = {Molica Bisci, G. and R{\u a}dulescu, V. D. and Servadei, R.},
title     = {Variational Methods for Nonlocal Fractional Problems},
series    = {Encyclopedia of Mathematics and its Applications},
volume    = {162},
publisher = {Cambridge University Press},
address   = {Cambridge},
year      = {2016}
}

@article{benguria2016solution,
author  = {Benguria, S.},
title   = {The solution gap of the {B}rezis--{N}irenberg problem on the hyperbolic space},
journal = {Monatsh. Math.},
volume  = {181},
number  = {3},
pages   = {537--559},
year    = {2016}
}

@book{KellerLenzWojciechowski2021,
author    = {Keller, M. and Lenz, D. and Wojciechowski, R. K.},
title     = {Graphs and Discrete Dirichlet Spaces},
series    = {Grundlehren der mathematischen Wissenschaften},
volume    = {358},
publisher = {Springer},
address   = {Cham},
year      = {2021}
}

@article{lu2022green,
author  = {Lu, G. and Yang, Q.},
title   = {Green's functions of Paneitz and {GJMS} operators on hyperbolic spaces and sharp {H}ardy--{S}obolev--{M}az'ya inequalities on half spaces},
journal = {Adv. Math.},
volume  = {398},
eid     = {108156},
year    = {2022}
}

@article{li2022higher,
author  = {Li, J. and Lu, G. and Yang, Q.},
title   = {Higher order {B}rezis--{N}irenberg problem on hyperbolic spaces: existence, nonexistence and symmetry of solutions},
journal = {Adv. Math.},
volume  = {399},
eid     = {108259},
year    = {2022}
}

@misc{lu2023explicit,
author        = {Lu, G. and Yang, Q.},
title         = {Explicit formulas of fractional {GJMS} operators on hyperbolic spaces and sharp fractional {P}oincar{\'e}--{S}obolev and {H}ardy--{S}obolev--{M}az'ya inequalities},
year          = {2023},
eprint        = {2310.15973},
archivePrefix = {arXiv}
}

@book{Helgason2024,
author    = {Helgason, S.},
title     = {Geometric Analysis on Symmetric Spaces},
edition   = {2},
series    = {Mathematical Surveys and Monographs},
volume    = {39},
publisher = {American Mathematical Society},
address   = {Providence, RI},
year      = {2008}
}

@misc{ChenRLogarithmic,
author        = {Chen, R.},
title         = {The Logarithmic {L}aplacian on Complete Riemannian Manifolds},
year          = {2025},
eprint        = {2506.19311},
archivePrefix = {arXiv}
}

@article{WW25,
author  = {Li, F. and Vaira, G. and Wei, J. and Wu, Y.},
title   = {Construction of bubbling solutions of the {B}rezis--{N}irenberg problem in general bounded domains ({I}): the dimensions 4 and 5},
journal = {J. Lond. Math. Soc. (2)},
volume  = {112},
number  = {2},
eid     = {e70246},
year    = {2025}
}

@article{FK25,
author  = {Frank, R. L. and K{\"o}nig, T. and Kova{\v r}{\'i}k, H.},
title   = {Blow-up of solutions of critical elliptic equations in three dimensions},
journal = {Anal. PDE},
volume  = {17},
number  = {5},
pages   = {1633--1692},
year    = {2024}
}

@article{FK21,
author  = {Frank, R. L. and K{\"o}nig, T. and Kova{\v r}{\'i}k, H.},
title   = {Energy asymptotics in the three-dimensional {B}rezis--{N}irenberg problem},
journal = {Calc. Var. Partial Differential Equations},
volume  = {60},
number  = {2},
eid     = {58},
year    = {2021}
}

@article{PV22,
author  = {Premoselli, B. and V{\'e}tois, J.},
title   = {Sign-changing blow-up for the {Y}amabe equation at the lowest energy level},
journal = {Adv. Math.},
volume  = {410},
eid     = {108769},
year    = {2022}
}

@article{SZ10,
author  = {Schechter, M. and Zou, W.},
title   = {On the {B}r{\'e}zis--{N}irenberg problem},
journal = {Arch. Ration. Mech. Anal.},
volume  = {197},
number  = {1},
pages   = {337--356},
year    = {2010}
}

@article{MR24,
author  = {Musso, M. and Rocci, S. and Vaira, G.},
title   = {Nodal cluster solutions for the {B}rezis--{N}irenberg problem in dimensions {$N\geq7$}},
journal = {Calc. Var. Partial Differential Equations},
volume  = {63},
number  = {5},
eid     = {119},
year    = {2024}
}

@article{DK23,
author  = {De Nitti, N. and K{\"o}nig, T.},
title   = {Critical functions and blow-up asymptotics for the fractional {B}rezis--{N}irenberg problem in low dimension},
journal = {Calc. Var. Partial Differential Equations},
volume  = {62},
number  = {4},
eid     = {114},
year    = {2023}
}

@article{servadei2015brezis,
author  = {Servadei, R. and Valdinoci, E.},
title   = {The {B}rezis--{N}irenberg result for the fractional {L}aplacian},
journal = {Trans. Amer. Math. Soc.},
volume  = {367},
number  = {1},
pages   = {67--102},
year    = {2015}
}

@article{MS15,
author  = {Molica Bisci, G. and Servadei, R.},
title   = {A {B}rezis--{N}irenberg splitting approach for nonlocal fractional equations},
journal = {Nonlinear Anal.},
volume  = {119},
pages   = {341--353},
year    = {2015}
}

@article{BC15,
author  = {Barrios, B. and Colorado, E. and Servadei, R. and Soria, F.},
title   = {A critical fractional equation with concave-convex power nonlinearities},
journal = {Ann. Inst. H. Poincar{\'e} C Anal. Non Lin{\'e}aire},
volume  = {32},
number  = {4},
pages   = {875--900},
year    = {2015}
}

@article{MM17,
author  = {Mawhin, J. and Molica Bisci, G.},
title   = {A {B}rezis--{N}irenberg type result for a nonlocal fractional operator},
journal = {J. Lond. Math. Soc. (2)},
volume  = {95},
number  = {1},
pages   = {73--93},
year    = {2017}
}

@article{GL21,
author  = {Guo, Y. and Li, B. and Pistoia, A. and Yan, S.},
title   = {The fractional {B}rezis--{N}irenberg problems on lower dimensions},
journal = {J. Differential Equations},
volume  = {286},
pages   = {284--331},
year    = {2021}
}

@misc{LuTao2026,
author        = {Lu, G. and Tao, C.},
title         = {Existence and symmetry of extremals for the high order {H}ardy--{S}obolev--{M}az'ya inequalities},
year          = {2026},
eprint        = {2602.05203},
archivePrefix = {arXiv}
}

@article{BrezisLieb1983,
author  = {Br{\'e}zis, H. and Lieb, E.},
title   = {A relation between pointwise convergence of functions and convergence of functionals},
journal = {Proc. Amer. Math. Soc.},
volume  = {88},
number  = {3},
pages   = {486--490},
year    = {1983}
}
	\endgroup

	\noindent\textit{Huyuan Chen}: Center for Mathematics and Interdisciplinary Sciences, Fudan University,\\[1mm]
	Shanghai 200433, China\\[1mm]
	Shanghai Institute for Mathematics and Interdisciplinary Sciences,\\[1mm]
	Shanghai 200433, China\\[1mm]
	\noindent\emph{Email:} \texttt{chenhuyuan@yeah.net, chenhuyuan@simis.cn}
	
	\vspace{1em}
	
	\noindent\textit{Rui Chen}: School of Mathematical Sciences, Fudan University,\\[1mm]
	Shanghai 200433,  China\\[1mm]
	Brandenburg University of Technology Cottbus--Senftenberg,\\[1mm]
	Cottbus 03046, Germany\\[1mm]
	\noindent\emph{Email:} \texttt{chenrui23@m.fudan.edu.cn} 
	
\end{document}